%% file: main.tex
\documentclass{article}
\input{preamblesAndMacros}

\usepackage[margin=1.in]{geometry}

\numberwithin{equation}{section}

\title{Random Subspace Cubic-Regularization Methods, with Applications to Low-Rank Functions}
\author{Coralia Cartis\thanks{Mathematical Institute, Woodstock Road, University of Oxford, Oxford, UK, OX2 6GG.   {\tt cartis@maths.ox.ac.uk.} This author's work was supported by the InnoHK CIMDA.},\quad Zhen Shao\thanks{Mathematical Institute, Woodstock Road, University of Oxford, Oxford, UK, OX2 6GG.  {\tt shaoz@maths.ox.ac.uk.} This author's work was supported by NAG Ltd and the Infomm CDT.} \quad and\quad Edward Tansley\thanks{Mathematical Institute, Woodstock Road, University of Oxford, Oxford, UK, OX2 6GG.  {\tt tansley@maths.ox.ac.uk.} This author's work was supported by the Mathematics of Random Systems CDT.}}

\date{December 30, 2024}

\begin{document}

\maketitle

\begin{abstract}
    We propose and analyze random subspace variants of the 
    second-order Adaptive Regularization using Cubics (ARC) algorithm. 
 These methods iteratively restrict the search space to some random subspace of the parameters, constructing and minimizing a local model only within this subspace. Thus, our variants only require access to (small-dimensional) projections of first- and second-order problem derivatives and calculate a reduced step inexpensively. 

  Under suitable assumptions, the ensuing methods maintain the optimal first-order, and second-order, global rates of convergence of (full-dimensional) cubic regularization, while showing improved scalability both theoretically and numerically, particularly when applied to low-rank functions. 
When applied to the latter, our adaptive variant naturally adapts the subspace size to the true rank of the function, without knowing it a priori.
\end{abstract}

\section{Introduction}

Second-order optimization algorithms for the unconstrained optimization problem
\begin{equation*}
\min_{x\in \R^d} f(x),
\end{equation*}
where $f:\R^d\rightarrow \R$ is a sufficiently smooth, bounded-below function, 
use gradient and curvature information to determine iterates and thus, may have faster convergence than first-order algorithms that only rely on gradient information. However, for high-dimensional problems, when $d$ is large, the computational cost of these methods can be a barrier to their use in practice. We are concerned with the task of scaling up second-order optimization algorithms so that they are a practical option for high-dimensional problems.

An alternative to linesearch and trust-region techniques, Adaptive Regularization framework using Cubics (ARC) \cite{nesterovCubicRegularizationNewton2006, Cartis:2009fq} determines the change $s_{k}$ between the current iterate $x_k$ and $x_{k+1}$ by (approximately) solving a cubically-regularized local quadratic model:
    \begin{equation}
        \argmin_{s \in \R^{d}} m_{k}(s) = f(x_{k}) + \langle \nabla f(x_{k}),\ s\rangle +  \frac{1}{2}\langle s,\ \nabla^{2} f(x_{k})s\rangle + \frac{1}{3\alpha_{k}}||s||_{2}^{3}
    \label{eq:arc}
    \end{equation}
where $\nabla f$ and $\nabla^{2} f$ denote the gradient and Hessian of $f$ and $\sigma_{k}=1/{\alpha_k}>0$ is the regularization parameter\footnote{The scaling of the regularization factor $\frac{1}{3\alpha_{k}}$ may equivalently vary  as $\frac{\sigma_k}{3}$ or $\frac{\sigma_k}{3!}$.}. Assuming Lipschitz continuity of the Hessian on the iterates' path, ARC requires at most $\mathcal{O}\left(\epsilon^{-\frac{3}{2}}\right)$ iterations and function and (first- and second-) derivative evaluations to attain an $\epsilon$-approximate first-order critical point; this convergence rate is optimal over a large class of second-order methods \cite{cartis_evaluation_2022}.

\paragraph{Dimensionality reduction in the parameter space} Approximately minimizing the local model given by \eqref{eq:arc} to find a search step, is computationally expensive as it requires iteratively solving $\mathcal{O}(d)$-dimensional linear systems involving $\nabla^2 f(x_k)$ and $\nabla f(x_k)$ (using factorization-based approaches or iterative, Krylov-type ones \cite{cartis_evaluation_2022}).
Furthermore, the memory cost of (full-dimensional) problem information such as $\nabla^{2}f(x_{k})$ and $\nabla f(x_{k})$ may also be prohibitively expensive in some applications. 
These computational costs involved in the step calculation motivate us to consider dimensionality reduction techniques that access/store reduced problem information of smaller dimensions, and so operate in (random) subspaces of the variables rather than in the full space.

\paragraph{Random subspace variants of ARC} At every iteration of a random subspace method,   the search space is restricted to some smaller-dimensional subspace of the functional domain (here, $\R^d$), and then the search step is computed in this subspace. 
In particular, the random subspace at the $k$th iteration is the row span of a sketching matrix $S_k\in \R^{l\times d}$ \cite{10.1561/0400000060, 10.1561/2200000035} with entries drawn, for example, from a (scaled) Gaussian distribution (or other useful distributions); this is advantageous as such random matrices have length-preserving properties according to Johnson-Lindenstrauss (JL) Lemma \cite{Johnson:1984aa}, which helps essentially preserve the length of gradients through the subspace projection. 
Thus,  Random Subspace ARC (R-ARC) replaces  $s\in \R^d$ in \eqref{eq:arc} by $s= S_k^{\top}\hat{s}$, where $\hat{s}\in \R^l$, $l\ll d$, which leads to the reduced model \eqref{eqn::mKHatSpec_CB} which employs only the reduced gradient 
$S_k \nabla f(x_{k})$ and reduced Hessian 
$S_k\nabla^{2} f(x_{k})S_k^{\top}$, of a sketched objective 
$\hat{f}(y)=f(S^T_ky)$, $y\in \R^l$. As such, R-ARC only requires these projected, low(er)-dimensional problem information,
$S_k \nabla f(x_{k})$ and 
$S_k\nabla^{2} f(x_{k})S_k^{\top}$, rather than their full-dimensional counterparts, allowing application of these methods to problems where full information is not available/too expensive. The former can be calculated efficiently without evaluation of full gradients and Hessians, by requiring gradient inner products and Hessian-vector products with a small subset of (random) directions; for example, $S_k \nabla f(x_{k})$ requires only $l$ directional derivatives of $f$ corresponding to the rows of $S_k$.
These directional derivatives  can be calculated by a small number of finite differences or by using automatic differentiation, requiring only the tangent model application without complete calculation and storage of an adjoint. 

The main focus here is on presenting a random subspace algorithm and associated analysis that recovers, with high probability, the optimal global rate of convergence of cubic regularization, namely, $\mathcal{O}(\epsilon^{-3/2})$ to drive the gradient of the objective below $\epsilon$; and to provide also, a global rate of convergence to second-order critical points, which is novel. To achieve these strong global rate guarantees while also ensuring high scalability, our random subspace variant of cubic regularization, R-ARC, can take advantage/benefits when special problem structure is present, namely, low rank, as we explain below. Then, arbitrarily small subspace dimensions can be chosen, even adaptively (as our R-ARC-D variant proposes), while the optimal worst-case complexity bound of order $\mathcal{O}(\epsilon^{-3/2})$ is guaranteed. We note that under mild(er) assumptions and for general objectives (without special structure), a worst-case complexity bound of order $\mathcal{O}(\epsilon^{-2})$ is guaranteed for R-ARC by the results in \cite{cartis_randomised_2022}. The challenge here is to find sufficient conditions on the algorithm and the problem that ensure that the random subspace ARC variant satisfies the optimal complexity bound of full-dimensional ARC and then to assess which of these conditions also achieve scalable variants.  

\paragraph{Existing literature} As an example, when the sketching matrix $S_k$ is a (uniform) sampling matrix, random subspace methods are random block coordinate variants, which in their first-order instantiations 
and only require a subset of partial derivatives to calculate an approximate gradient direction \cite{Nesterov12, richtarik2014iteration}. Despite documented failures of (deterministic) coordinate techniques on some problems \cite{Powell1973}, the challenges of large-scale applications have brought these methods to the forefront of research in the last decade; see 
\cite{Wright2015} for a survey of this active area, particularly for convex optimization.
Developments of these methods for nonconvex optimization can be found, for example, in \cite{Raza2013, Patrascu2015, lu2017randomized} and more recently, \cite{Yang2020}; with extensions to constrained problems \cite{Birgin1, Birgin2}, and distributed strategies \cite{Facchinei, facchinei2015parallel}.
Subspace methods can be seen as an extension of (block) coordinate methods by allowing the reduced variables
to vary in (possibly randomly chosen) subspaces that are not necessarily aligned with
coordinate directions. (Deterministic) subspace and decomposition methods have been of steadfast interest in the optimization community for decades. In particular, Krylov subspace methods can be applied to calculate -- at typically lower computational cost and in a matrix free way -- an approximate Newton-type search direction over increasing and nested subspaces; see \cite{Nocedal:2006uv}  for Newton-CG techniques (and further references), \cite{GLTR, conn2000trust} for initial trust-region variants and \cite{cartis_evaluation_2022} for cubic regularization variants.  However, these methods still require that the full gradient vector is calculated/available at each iteration, as well as (full) Hessian matrix actions. The requirements of  large-scale problems, where full problem information may not be available, prompt us to investigate the case when only inexact gradient and Hessian action information is available, such as their projections to lower dimensional subspaces. Thus, in these frameworks (which subsume block-coordinate methods), both the problem information and the search direction are inexact and low(er) dimensional. Deterministic proposals can be found in \cite{Yuan}, which also traces a history of these approaches. Our aim here is to investigate {\it random} subspace second-order variants, that exploit random matrix theory properties in constructing the subspaces and analyzing the methods.

Both first-order and second-order methods that use sketching to generate the random subspace have been proposed, particularly for convex and convex-like problems, that only calculate a (random) lower-dimensional projection of the gradient or Hessian. Sketched gradient descent methods have been proposed in 
\cite{Kozak_published, kozak2019stochastic, Gower2020, Grishchenko2021}. 
 The sketched Newton algorithm \cite{pilanci2017newton} requires a sketching matrix 
 that is proportional to the rank of the (positive definite) Hessian, 
 while the sketched online Newton \cite{luo2016efficient}  
 uses streaming sketches to scale up a second-order method, 
 comparable to Gauss–Newton, for solving online learning problems.
Improvements, using sketching, to the solution of large-scale least squares problems are given in  \cite{lacotte2019faster, ergen2019, lacotte2020optimal, kahale2020leastsquares}.
The randomized subspace Newton \cite{gower2019rsn} efficiently sketches 
the full Newton direction for a family of generalized linear models, 
such as logistic regression.
The stochastic dual Newton ascent algorithm in  \cite{qu2016sdna} 
requires a positive definite upper bound $M$ on the Hessian and 
proceeds by selecting random principal submatrices of $M$ 
that are then used to form and solve an approximate Newton system. 
The randomized block cubic Newton method in \cite{doikov2018randomized} 
combines the ideas of randomized coordinate descent with 
cubic regularization and requires the optimization problem to be block separable. In \cite{Gower2022}, a sketched Newton method is proposed for the solution of (square or rectangular) nonlinear systems of equations, with a global convergence guarantee.  The local rate of convergence of sketched Newton's method  has been comprehensively studied in \cite{poirion}; this work carefully captures the cases when this rate can be at least superlinear (when the sketching matrix captures the Hessian's range), and when this is impossible (if the sketch is `too small' dimensional and thus misses important spectral information about the Hessian).

A general random subspace algorithmic framework and associated global rate of convergence for nonconvex optimization is given in \cite{cartis_randomised_2022}\footnote{In revision for journal publication.}, which includes material (Chapter 4) from \cite{Zhen-PhD}, which in turn builds on the succinct material in \cite{zhen:icml_BCGN}. Under  weaker assumptions than for standard probabilistic models \cite{MR3245880, Gratton:2017kz, Cartis:2017fa}, and for more general methods and objectives than in \cite{Kozak_published, kozak2019stochastic},we show that, provided the problem gradient is Lipschitz continuous, this random subspace framework based on sketching, with trust region or quadratic regularization strategies (or linesearch etc),
    has a global worst-case complexity of order $\mathO{\epsilon^{-2}}$ to drive the gradient $\nabla f(x_k)$ of $f$ below the desired accuracy $\epsilon$, with exponentially high probability; this complexity bound matches in  the order of the accuracy that of corresponding deterministic/full dimensional variants of these same methods. The choice of the random subspace $S_k$ (at iteration $k$) only needs to project approximately correctly the length of the full gradient vector, namely, \eqref{JL:eq} holds with $S=S_k$ and $Y:=\{\nabla f(x_k)\}$. This one-dimensional subspace embedding is a mild requirement that can be achieved by several sketching matrices such as scaled Gaussian matrices (see Lemma \ref{lem:JL}), and some sparse embeddings. In these cases, the same embedding properties provide that the dimension of the projected subspace (see $l$ value in Lemma \ref{lem:JL}) is {\it independent} of the ambient dimension $d$ and so the algorithm can operate in a small dimensional subspace at each iteration, making the projected step and gradient much less expensive to compute\footnote{We also address the case of  sampling sketching matrices, when our approach reduces to  randomized block-coordinates, which have weaker embedding properties and so require additional problem assumptions to yield a subspace variant that is almost surely convergent.}. We note that the approach in \cite{cartis_randomised_2022} applies to second-order methods as all (subspace) local models allow (approximate or exact) subspace second-order terms. Thus, random subspace cubic-regularization variants (ARC) are contained in the above framework; however, only a suboptimal complexity bound is obtained for these methods and second-order smoothness of the objectives is not exploited. Hence, in the present work, we address these questions (as well as second-order criticality).

We note that the first probabilistic, inexact cubic regularization variants with optimal global rate of convergence of order $\mathO{\epsilon^{-3/2}}$ can be found in \cite{Cartis:2017fa}. There, probabilistic second-order local models with cubic powers of length of the step are employed, which can be calculated by sampling of function values, but which are different/stronger conditions compared to the subspace sketching conditions here (see Definition \ref{def:true:CBGN}). We have not found a way to directly use the probabilistic model conditions to generate random subspace methods. 

Sketching can also be applied not only to reduce the dimension of the parameter/variable domain, but also the data/observations when minimizing an objective given as a  sum of (many smooth) functions, as it is common when training machine learning systems \cite{Curtis} or in data fitting/regression problems. Then, using sketching inside cubic regularization methods, as mentioned above regarding probabilistic models, we subsample some of these constituent functions and calculate a local, cubically-regularized improvement for this reduced objective. This leads to stochastic cubic regularization  methods and related variants\footnote{Note that in the case of random subspace methods, the objective is evaluated accurately as opposed to observational sketching, where this accuracy is lost through subsampling.} \cite{kohler_sub-sampled_2017, xu_second-order_2018, xu_newton-type_2019, yao_inexact_2018, sadok, bellavia2024}, but not directly relevant to the developments in this paper as here, we only investigate reducing the dimension $d$ of the variable domain.

\paragraph{Random embeddings}

We introduce some basic definitions and results that give conditions under which a sketching matrix approximately preserves lengths. This length preservation is essential in ensuring sufficient function decrease so that R-ARC can attain the same rates of convergence as the ARC algorithm.

\begin{definition}[$\epsilon$-subspace embedding \cite{10.1561/0400000060}]
    An $\epsilon$-subspace embedding for a matrix $B \in \R^{d \times k}$ is a matrix $S \in \R^{l \times d}$ such that
    \begin{equation}
        (1-\epsilon)\normTwo{y}^{2} \leq \normTwo{Sy}^{2} \leq (1+\epsilon)\normTwo{y}^{2} \texteq{for all $y \in Y = \{y :y = Bz, z\in \R^{k}\}$}.
    \end{equation}
    \label{def:sketch}
\end{definition}
\begin{definition}[Oblivious embedding \cite{10.1561/0400000060, 10.1109/FOCS.2006.37}]
    A distribution $\mathcal{S}$ on $S \in \R^{l \times d}$ is an $(\epsilon, \delta)$-oblivious embedding if for a given fixed/arbitrary set of vectors, we have that, with a high probability of at least $1-\delta$, a matrix S from the distribution is an $\varepsilon$-embedding for these vectors. 
\end{definition}
Next we define scaled Gaussian matrices, which we then show to have embedding properties under certain conditions.

\begin{definition}\label{def:Gaussian}
    A matrix $S \in \R^{l \times d}$ is a scaled Gaussian matrix if $S_{i, j}$ are independently distributed as $N(0, l^{-1})$.
\end{definition}

Scaled Gaussian matrices have  oblivious embedding properties provided they are of a certain size, according to the celebrated Johnson-Lindenstrauss (JL) Lemma.

\begin{lemma}[JL Lemma \cite{Johnson:1984aa, MR1943859}] 
    \label{lem:JL}
    Given a fixed, finite set $Y \subset \mathbb{R}^{n}, \epsilon, \delta > 0$, let $S \in \mathbb{R}^{l \times d}$ be a scaled Gaussian matrix, with $l = \mathcal{O}(\epsilon^{-2}\log(\frac{|Y|}{\delta}))$ and where $|Y|$ refers to the cardinality of $Y$. Then we have, with probability at least $1 - \delta$, that
    \begin{equation}
        (1-\epsilon)||y||_{2}^{2} \leq ||Sy||_{2}^{2} \leq (1+\epsilon)||y||_{2}^{2} \quad\text{for all $y \in Y$.}
        \label{JL:eq}
    \end{equation}
\end{lemma}
We note that the distribution of the scaled Gaussian matrices with $l=\mathcal{O}(r\epsilon^{-2}|\log \delta|)$, where $r\leq \min (d,k)$ is the rank of (any given) $B\in \mathbb{R}^{d\times k}$, are  an 
$(\epsilon, \delta)$-oblivious embedding for any vector in the range of $B$ \cite{10.1561/0400000060}. 

Whilst we primarily consider Gaussian matrices in the theoretical sections of this paper, we also give definitions for several other choices of sketching matrices that we refer to or use in the numerical results, in the Appendix. 
Embedding properties hold for other families of matrices other that Gaussian matrices, such as for $s$-hashing matrices, although with a different dependency on $l$; such ensembles are defined in the Appendix as well as other that we use in this paper and a summary of their embedding properties can be found in \cite{Zhen-PhD, 2021arXiv210511815C}.  

\paragraph{Low-rank functions}

We now describe a class of functions that benefit from random subspace algorithms. We  assume that $f$  has \emph{low effective dimension}, namely, $f$ varies only within a (low-dimensional and unknown) linear subspace and is constant along its complement. These functions, also referred to as \emph{multi-ridge} \cite{Tyagi2014}, functions with \emph{active subspaces} \cite{constantine2015book}, or \emph{low-rank} functions \cite{cosson_gradient_2022, Parkinson2023}, arise when tuning (over)parametrized models and processes, such as in hyper-parameter optimization for neural networks \cite{Bergstra2012}, heuristic algorithms for combinatorial optimization problems \cite{Hutter2014}, complex engineering and physical simulation problems  including climate modelling and policy search and control; see \cite{constantine2015book} for more details. 
Recently, they have also been observed in the training landscape of neural networks  and the trained nets as a function of the (training) data \cite{Parkinson2023}. See \cite{cartis_learning_2024} and the references therein for more details. 

The most general definition of low-rank functions is as follows.
\begin{definition}\label{def:low:rank}
    \cite{wang_bayesian_2016} A function $f:\R^{d} \rightarrow \R$ is said to be of rank $r$, with $r \leq d$ if
    \begin{itemize}
        \item there exists a linear subspace $\mathcal{T}$ of dimension $r$ such that for all $x_{\top} \in \mathcal{T} \subset \R^{d}$ and $x_{\perp} \in \mathcal{T}^{\perp} \subset \R^{d}$, we have $f(x_{\top} + x_{\perp}) = f(x_{\top})$, where $\mathcal{T}^{\perp}$ is the orthogonal complement of $\mathcal{T}$;
        \item $r\ll d$ is the smallest integer with this property
    \end{itemize} 
\end{definition}
 We see that $f$ only varies in $\mathcal{T}$, its effective subspace, while being constant in $\mathcal{T}^{\perp}$.
We revisit low-rank functions in \autoref{sec:low_rank} where we derive desirable properties that shows the Hessian of these functions is globally low rank, and prove a convergence result for R-ARC applied to low-rank functions.
It is intuitive to see that if all `interesting' optimization aspects happen in $\mathcal{T}$, sketching with a dimension proportional to $r$  rather than $d$ will capture the key problem information while also ensuring substantial dimensionality reduction.

Random subspace methods have also been studied for the global optimization of nonconvex functions, especially in the presence of low effective dimensionality of the objective; see \cite{adilet1, adilet2, adilet3} and references therein. This special structure assumption has also been investigated in the context of first-order methods for local optimization \cite{subramanian, cosson_gradient_2022}.

\paragraph{Summary of contributions} We give a brief summary of our already-described contributions\footnote{A brief description, without proofs, of a subset of the results of this paper has appeared as part of a six-page conference proceedings paper (without any supplementary materials) in the Neurips Workshop “Order Up ! The Benefits of Higher-Order Optimization in Machine Learning” (2022), see \cite{zhen:neurips_RARC}. We also note that a substantial part of this paper has been included as Chapter 5 of the doctoral thesis \cite{Zhen-PhD} but has not been submitted or published elsewhere. The ideas (without proofs) of the R-ARC-D variant for low-rank functions has appeared as part of a six-page conference proceedings paper (without any supplementary materials) in the Neurips Workshop Optimization for Machine Learning (2024), see \cite{ed:neurips_RARC-D}. }: 
\begin{itemize}
\item We propose R-ARC, a random subspace variant of ARC that generates iterates in subspaces and does not require full derivative information. We show that R-ARC benefits from the optimal rate of convergence, $\mathcal{O}(\epsilon^{-3/2})$, to drive $\|\nabla f(x^k)\|\leq \epsilon$, with high probability, matching the full-dimensional ARC's rate; at the time of completion \cite{Zhen-PhD}, this was the first such analysis using sketching. The assumptions required on the sketching matrices that generate the subspaces, to ensure this rate, include correctly sketching the gradient and Hessian at the current iterate, with some probability. Thus, to ensure dimensionality reduction and benefits of subspace approaches, the Hessian needs to have low rank which leads us to the useful, existing notion of low-rank functions that possess this property.  
\item We propose a novel adaptive variant of R-ARC, R-ARC-D, for low-rank functions that does not require the rank of the function a priori, while benefiting from dimensionality reduction; the size of the sketching subspace is based on the rank of the projected Hessian matrix. The optimal first-order rate is shown to hold here too.
\item We show that R-ARC converges sublinearly to an approximate second-order critical point in the subspace, in the sense that the subspace Hessian at the limit point is nearly positive semidefinite, and bound the convergence rate. Furthermore, R-ARC with Gaussian sketching converges to an approximate second-order critical point in the full space, with the same rate, matching again the corresponding bound for ARC\footnote{Again, this analysis was unprecedented at the time of its completion/online posting \cite{Zhen-PhD}.}.
\item We present numerical results that compare the full-dimensional ARC with R-ARC and the adaptive variant R-ARC-D, on full dimensional CUTEst problems and constructed low-dimensional variants of these problems, with encouraging results.  
\end{itemize}

There have already been follow ups and applications of our work presented here, such as \cite{bellavia2023} for an objective-function-free variant, while \cite{lindon2024} uses the second-order criticality results in novel derivative-free model based developments.

\paragraph{Structure of the paper} The structure of the paper is as follows. Section 2 summarizes the generic framework and its associated analysis introduced in \cite{Zhen-PhD, cartis_randomised_2022}, which we then crucially use in Section 4, to analyze the random subspace cubic regularization algorithms (R-ARC) which we introduce in Section 3. Section 4 proves the optimal global rate of convergence of R-ARC for generating an approximate first-order critical point with high probability, and then shows that this result applies to low-rank functions in which case, R-ARC can provably achieve dimensionality reduction.
In R-ARC, the dimension of the random subspace is fixed throughout the algorithm; Section 5 proposes an adaptive subspace selection strategy for R-ARC, R-ARC-D, that is particularly scalable for low-rank functions, and shows it satisfies the optimal first-order complexity bound. Section 6 proves a global rate for R-ARC for achieving approximate second-order critical points, while Section 7 presents numerical results comparing full-dimensional ARC with R-ARC and R-ARC-D on full rank CUTEst problems and their low-rank modifications.

\section{A Generic Algorithmic Framework and its Analysis (Summary)}

This section summarizes a general algorithm and associated analysis that originates in \cite{Zhen-PhD, cartis_randomised_2022} and which, in terms of algorithm construction, is similar to \cite{Cartis:2017fa}\footnote{A detailed comparison to \cite{Cartis:2017fa} is given in Section 2 in \cite{cartis_randomised_2022}.}. We will then fit our random subspace cubic regularization variants into this algorithmic framework and its analysis, which though more abstract, substantially shortens the proofs of our results.

The scheme (Algorithm \ref{alg:generic}) relies on building a local, reduced model of the objective
function at each iteration, 
minimizing this model or reducing it in a sufficient manner and
considering the step which is dependent on a step size parameter and
which provides the model reduction (the step size parameter may be
present in the model or independent of it). 
This step determines a new candidate point. The function value is then
computed (accurately) at the new candidate point. 
If the function reduction provided by the candidate point is deemed
sufficient, then the iteration is declared successful, the 
candidate point becomes the new iterate and the step size parameter is
increased. Otherwise, the iteration is 
unsuccessful, the iterate is not updated and the step size parameter is reduced\footnote{Throughout the paper, we let $\N^+=\N\setminus \{0\}$ denote the set of positive natural numbers.}. 

\begin{algorithm}[H]
\begin{description}
\item[Initialization] \ \\
Choose a class of  (possibly random) models $m_k\left(w_k(\sHat)\right) = \mKHat{\sHat}$, where $\sHat \in \R^l$ with $l\leq d$ is the step parameter and $w_k: \R^l \to \R^d$ is the prolongation function. 
Choose constants $\gamma_1\in (0,1)$, $\gamma_2 = \gammaOne^{-c}$, for some $c \in \N^+$,  $\theta \in (0,1)$ and $\alpha_{\max}>0$.
Initialize the algorithm by setting $x_0 \in \R^d$, $\alpha_0 = \alphaMax \gamma_1^p$ for some $p \in \N^+$ and $k=0$.

 \item[1. Compute a reduced model and a step] \ \\
Compute a local (possibly random) reduced model $\mKHat{\sHat}$ of $f$ around $x_k$ with $\mKHat{0} = f(x_k)$. \\
Compute a step parameter $\sHat_k(\alpha_k)$, where the parameter $\alpha_k$ is present in the reduced model or the step parameter computation.\\
Compute a potential step $s_k = w_k(\sHat_k)$.

\item[2. Check sufficient decrease]\ \\  
Compute $f(x_k + s_k)$ and check if sufficient decrease (parameterized by $\theta$) is achieved in $f$ with respect to (some function related to) the model decrease $\mKHat{0} - \mKHat{\hat{\sK}(\alpha_k)}$.

\item[3. Update the parameter $\alphaK$ and possibly take the trial step $\sK$]\ \\
If sufficient decrease is achieved, set $\xKPlusOne = \xK + \sK$ and $\alphaKPlusOne = \min \set{\alphaMax, \gammaTwo\alphaK}$ [successful iteration].
Otherwise set $\xKPlusOne = \xK$ and $\alphaKPlusOne = \gammaOne \alphaK$ [unsuccessful iteration].\\
Increase the iteration count by setting $k=k+1$ in both cases. 

\caption{\bf{Generic optimization framework based on  randomly generated reduced models.}} \label{alg:generic} 

\end{description}
\end{algorithm}

Our framework (\autoref{alg:generic}) explicitly states that the model does not need to have the same dimension as the objective function; with the two being connected by a step transformation function $w_k: \R^l \to \R^d$ which typically here will have $l<d$. As an example, note that letting $l=d$ and $w_k$ be the identity function in \autoref{alg:generic} leads to usual, full-dimensional local models, 
which coupled with typical strategies of linesearch and trust-region as parametrized by $\alpha_k$ or regularization  (given by $1/\alpha_k$),
recover classical, deterministic variants of corresponding methods. When $l<d$, we choose $S_k$ be a random embedding and let $s_k = w_k(\sHat_k):=S_k^T\sHat_k$, obtaining random subspace variants of classical methods;
see \cite{cartis_randomised_2022} for more details regarding random subspace trust-region and quadratic regularization variants.

\subsection{A probabilistic convergence result}

Since the local model is (possibly) random, $\xK, \sK, \alphaK$ are in general random variables; we will use $\barXK, \bar{s}_k, \bar{\alpha}_k$ to denote their realizations.  Given  (any) desired accuracy $\epsilon >0$, we define convergence in terms of the random variable
\begin{equation}
    \nEps:= \min \set{k: {\rm criticality \,measure}\leq \epsilon},
   \label{eqn::nEps}
\end{equation}
which represents the first time that some desired criticality measure, which can be first-order ($\normTwo{\gradFK}\leq \epsilon$) or second-order
($-\lambda_{\min}(\hessFK)\leq \epsilon$)
 descends below $\epsilon$. Also note that $k< \nEps$ implies ${\rm criticality\, measure} > \epsilon$, which will be used repeatedly in our proofs. 

Let us suppose that there is a subset of iterations, which we refer to as \textbf{true iterations} such that \autoref{alg:generic} satisfies the following assumptions. The first assumption states that given the current iterate, an iteration $k$ is true at least with a fixed probability, and is independent of the truth value of all past iterations.

\begin{assumption}\label{AA2}
There exists $\delta_S \in (0,1)$ such that for any $\barXK \in \R^d$ and $k=1, 2, \dots$
\begin{equation}
    \probabilityGivenXK{T_k} \geq 1-\delta_S, \notag
\end{equation}
where $T_k$ is defined as
        \begin{equation}
            T_k = 
        \twoCases{1}{\text{if iteration $k$ is true}}{0}{\text{otherwise}}
        \label{eqn::t_k}
        \end{equation}

Moreover, $\probability{T_0} \geq 1-\delta_S$; and $T_k$ is conditionally independent of $T_0, T_1, \dots, T_{k-1}$ given $x_k = \barXK$. 
\end{assumption}

The next assumption says that for $\alphaK$ small enough, any true iteration before convergence is guaranteed to be successful.

\begin{assumption}\label{AA3}
For any $\epsilon >0$, there exists an iteration-independent constant $\alphaLow>0$ (that may depend on $\epsilon$ as well as problem and algorithm parameters) such that 
if iteration $k$ is true, $k< \nEps$, and $\alpha_k \leq \alphaLow$, then iteration $k$ is successful. 

\end{assumption}

The next assumption says that before convergence, true and successful iterations result in an objective decrease bounded below by an (iteration-independent) function $h$, which is monotonically increasing in its two arguments, $\epsilon$ and $\alphaK$. 

\begin{assumption}\label{AA4}
There exists a non-negative, non-decreasing function $h(z_1,z_2)$ such that,
for any $\epsilon>0$, 
if iteration $k$ is true and successful with $k < \nEps$, then 

	\begin{equation}
	f(\xK) - f(\xK + \sK) \geq h(\epsilon, \alpha_k), \label{eqn::generic_model_decrease}
	\end{equation}
where $s_k$ is computed in Step 1 of \autoref{alg:generic}. Moreover, $h(z_1, z_2)>0$ if both $z_1>0$ and $z_2>0$.
\end{assumption}

The final assumption requires that the function values are monotonically decreasing throughout the algorithm. 

\begin{assumption}\label{AA5}
For any $k \in \N$, we have 
\begin{equation}
    f(\xK) \geq f(\xKPlusOne). \label{eqn::fKNonIncreasing}
\end{equation}
\end{assumption}

\begin{lemma}
\label{lem::alphaMin}
Let $\epsilon>0$ and \autoref{AA3} hold with $\alphaLow>0$. Then there exists $\newL \in \N^+$, and $\alphaMin >0$ such that 
\begin{align}
&\newL = \ceil{\logBaseGammaOne{ \minMe{\frac{\alphaLow}{\alphaZero}}{\frac{1}{\gammaTwo}}}}, \label{eqn:tauLDef}\\
&\alphaMin = \alphaZero \gammaOne^\newL, \label{eqn::alphaMin}\\
    &\alphaMin \leq \alphaLow, \notag \\
    &\alphaMin \leq \frac{\alphaZero}{\gammaTwo}, \label{eqn::alphaMinUpperByGammaTwoOverAlphaZero}
\end{align}
where $\gammaOne, \gammaTwo, \alphaZero$ are defined in \autoref{alg:generic}. 
\end{lemma}
\begin{proof}

We have that $\alphaMin \leq \alphaZero \gammaOne^{\logBaseGammaOne{ \frac{\alphaLow}{\alphaZero}}} = \alphaLow$. Therefore by \autoref{AA3}, if iteration $k$ is true, $k< \nEps$, and $\alpha_k \leq \alphaMin$ then iteration $k$ is successful. Moreover, $\alphaMin \leq \alphaZero \gammaOne^{\logBaseGammaOne{\frac{1}{\gammaTwo}}} = \frac{\alphaZero}{\gammaTwo} = \alphaZero \gammaOneC$. It follows from $\alphaMin = \alphaZero \gammaOne^\newL$ that $\newL \geq c$. Since $c \in \N^+$, we have $\newL \in \N^+$ as well.

\end{proof}

\autoref{thm2} is our main result concerning the convergence of  \autoref{alg:generic}. It states a probabilistic bound on the total number of iterations $\nEps$ required by the generic framework to converge to within  $\epsilon$-accuracy of first order optimality.

\begin{theorem}
\label{thm2}
Let \autoref{AA2}, \autoref{AA3}, \autoref{AA4} and \autoref{AA5} hold with $\epsilon>0$, $\delta_S \in (0,1)$, $\alphaLow>0$, $h: \R^2 \to \R$ and $\alphaMin = \alphaZero \gammaOne^\newL$ associated with $\alphaLow$, for some $\newL \in \N^+$; assume also that 
\begin{equation}
    \deltaS < \frac{c}{(c+1)^2}, \label{eqn::deltaSConditionThmTwo}
\end{equation}
where $c$ is chosen at the start of \autoref{alg:generic}. Suppose that \autoref{alg:generic} runs for $N$ iterations\footnote{For the sake of clarity, we stress that $N$ is a deterministic constant, namely, the total number of iterations that we run \autoref{alg:generic}. $\nEps$, the number of iterations needed before convergence, is a random variable.}.

Then, for any $\delta_1 \in (0,1)$ such that 
\begin{equation}
    \gDeltaSDeltaOne >0, \label{eqn:tmp32}
\end{equation}
where 
\begin{equation}
    g(\deltaS, \deltaOne) = \nPreFactorTR, \label{eqn:gDeltaSDeltaOneDef}
\end{equation}
if $N$ satisfies 
\begin{equation}
    N \geq \gDeltaSDeltaOne \squareBracket{
         \fZeroMinusfStarOverH
         + \frac{\newL}{1+c}}, \label{eqn::n_upper_2}
\end{equation}
we have that
\begin{equation}
    \probability{N \geq \nEps} \geq 1 - \chernoffLowerExponential. \label{eqn:tmp33}
\end{equation}
\end{theorem}

The complete proof of \autoref{thm2} can be found in \cite{cartis_randomised_2022}.

\begin{remark}
Note that $c/(c+1)^2 \in (0, 1/4]$ for $c \in \N^+$, and so \eqref{eqn::deltaSConditionThmTwo} and \eqref{eqn:tmp32} can only be satisfied for some $c$ and $\deltaOne$ given that $\deltaS < \frac{1}{4}$. Thus our theory requires that an iteration is true with probability at least $\frac{3}{4}$. Compared to the analysis in \cite{Cartis:2017fa}, which requires that an iteration is true with probability at least $\frac{1}{2}$, our condition imposes a stronger requirement. This is due to the high probability nature of our result, while the convergence result in \cite{Cartis:2017fa} is in expectation\footnote{Our result does imply a bound on $\nEps$, see \cite{cartis_randomised_2022}, Remark 2}. Furthermore, we will see in \autoref{lem:arbitatratyDeltaS} that we are able to impose arbitrarily small value of $\delta_S$, thus satisfying this requirement, by choosing an appropriate dimension of the local reduced model $\mKHat{\sHat}$.
\end{remark}

The general nature of this analytical framework, though abstract, will allow us to fit both first and second-order analyses of convergence into it, while substantially shortening the proofs as we will only need to show that the above four assumptions are satisfied.

\section{A Random Subspace Cubically-Regularized Algorithm}

Now we describe the main algorithm in this paper, a random subspace cubic regularization method, R-ARC (Algorithm \ref{alg:R-ARC}). At each iteration $k$, R-ARC draws a random matrix $S_k\in\R^{l\times d}$ and then a cubic regularization local model of the objective is constructed in the subspace generated by the rows of $S_k$; this model uses a (randomly) projected gradient $S_k\nabla f(x_k)$ and projected Hessian $S_k\nabla^2 f(x_k)S_k^T$ onto the relevant subspace, which only requires access to a subset of directional derivatives/finite differences, rather than the full gradient/Hessian. This local model (in \eqref{eqn::mKHatSpec_CB}) is then approximately and locally minimized in the reduced space, in $\R^l$, to first- (and when needed, second-) order criticality for the reduced model.
The model minimization conditions in \eqref{tmp:CBGN:3}--\eqref{hessMKGeq0} are subspace applications of standard termination criteria for (full-dimensional) cubic regularization methods \cite[Chapter 3]{cartis_evaluation_2022}.
The step is then accepted (or rejected), using similar criteria of sufficient objective decrease (or otherwise) to full dimensional ARC, and the regularization weight $1/\alpha_k$ is increased when insufficient progress is made,  in order to shorten the step to increase the chance of decreasing the objective subsequently.

\begin{algorithm}
\begin{description}

\item[Initialization] \ \\
 Choose a matrix distribution $\cal{S}$ of matrices $S\in \rLTimesD$. 
 Choose constants $\gamma_1\in (0,1)$, $\gamma_2 >1$, $\theta \in (0,1)$, $\kappaT, \kappaS \geq 0$
 and $\alpha_{\max}>0$ such that $
        \gamma_2 = \oneOverGammaOneC,
        $
    for some $c\in \N^+$.
 Initialize the algorithm by setting 
 $x_0 \in \R^d$, 
 $\alpha_0 = \alphaMax \gamma^p$
 for some $p\in \N^+$ and $k=0$.

\item[1. Compute a reduced model and a trial step] \ \\
 In Step 1 of \autoref{alg:generic}, draw a random matrix $S_k \in \R^{l \times d}$ from $\cal{S}$, and let
 \begin{align}
     \mKHat{\sHat} 
     & = \fK + \innerProduct{S_k\gradFK}{ \sHat} + 
     \frac{1}{2} \innerProduct{\sHat}{S_k\hK S_k^T\sHat}
     + \frac{1}{3\alphaK}\normTwo{S_k^T \sHat}^3 \notag \\
     & = \hat{q}_k(\hat{s}) + \frac{1}{3\alphaK}\normTwo{S_k^T \sHat}^3, \label{eqn::mKHatSpec_CB} 
 \end{align}
 where $\hat{q}_k(\hat{s})$ is the second-order Taylor series of $f(x_k + S_k^T \sKHat)$ around $x_k$;
 
 Compute $\sKHat$ by approximately minimizing \eqref{eqn::mKHatSpec_CB} such that
 \begin{align}
     \mKHat{\sKHat} \leq \mKHat{0} \label{tmp:CBGN:3}\\
     \normTwo{\grad \mKHat{\sKHat}} \leq 
     \kappa_T \normTwo{S_k^T \sKHat}^2 \label{tmp:CBGN:5} \\
     \grad^2 \mKHat{\sKHat} \succeq -\kappaS \normTwo{S_k^T \sKHat}, \label{hessMKGeq0}
 \end{align}
 where we may drop \eqref{hessMKGeq0} if only convergence to a first order critical point is desired. 
 
 Compute a trial step 
 \begin{align}
     s_k = w_k(\sHat_k) = S_k^T \sHat_k.\label{eqn::wKSpec_CB}
 \end{align}

\item[2. Check sufficient decrease]\ \\  
In Step 2 of \autoref{alg:generic}, check sufficient decrease as defined by the condition
\begin{equation}
    \fK - \fKPlusOne \geq \theta \squareBracket{\hat{q}_k(0) - 
    \hat{q}_k(\hat{s})
    }. \label{eq:cubic:sufficient_decrease}
\end{equation}

\item[3, Update the parameter $\alphaK$ and possibly take the trial step $\sK$]\ \\
If \eqref{eq:cubic:sufficient_decrease} holds, set $\xKPlusOne = \xK + \sK$ and $\alphaKPlusOne = \min \set{\alphaMax, \gammaTwo\alphaK}$ [successful iteration]. 

Otherwise set $\xKPlusOne = \xK$ and $\alphaKPlusOne = \gammaOne \alphaK$ [unsuccessful iteration]

Increase the iteration count by setting $k=k+1$ in both cases.

\caption{\bf{Random subspace cubic regularization algorithm (R-ARC) }} \label{alg:R-ARC} 
\end{description}
\end{algorithm}

There are two main strategies for computing $\sKHat$ by minimizing \eqref{eqn::mKHatSpec_CB}, by requiring a factorization of $S_k \hessFK S_k^T$ (in a Newton-like algorithm), or repeated matrix-vector products involving $S_k \hessFK S_k^T$ (in a Lanczos-based algorithm). Although we note that the iteration  complexity, and the evaluation complexity of $f$ and its derivatives are unaffected by the computation complexity of calculating $\sKHat$, \autoref{alg:R-ARC} significantly reduces the computation of this inner problem by reducing the dimension of the Hessian from $d \times d$ to $l \times l$, compared to the full-space counterpart; this is in addition to reducing the gradient and Hessian evaluation cost per iteration (due to only needing projected/directional components).

 \autoref{alg:R-ARC} is a specific variant of \autoref{alg:generic}. Therefore the convergence result in \autoref{thm2} can be applied, provided that the four assumptions of the theorem can be shown to hold here. In subsequent sections, we give different definitions of the two key terms in the convergence result \autoref{thm2}: $\nEps$ and true iterations. These lead to different requirements for the matrix distribution $\cal{S}$, and iteration complexities to drive $\normTwo{\gradFK} < \epsilon$, $\lambdaMin{S_k \hessFK S_k^T} > -\epH $ and/or $\lambdaMin{\hessFK} > -\epH$. For each of these different optimality measures, we will need to show that the four assumptions in \autoref{thm2} hold.

\section{Optimal Global Rate of R-ARC to First Order Critical Points} 
\label{sec:rarc-convergence}

Our first convergence result shows \autoref{alg:R-ARC} drives $\normTwo{\gradFK}$ below $
\epsilon$ in $\mathO{\epsilon^{-3/2}}$ iterations, given that $\cal{S}$ has an embedding property (a necessary condition of which is that $\cal{S}$ is an oblivious subspace embedding for matrices of rank $r+1$, where $r$ is the maximum rank of $\hessFK$ across all iterations).

\paragraph{Define $\nEps$ and true iterations based on (one-sided) subspace embedding}
In order to prove convergence of \autoref{alg:R-ARC}, we show that \autoref{AA2}, \autoref{AA3}, \autoref{AA4}, \autoref{AA5} that are needed for \autoref{thm2} to hold are satisfied. To this end, we first define $\nEps$, the criterion for convergence, as $\min \{k: \normTwo{\grad f(x_{k+1})} \leq \epsilon \} $.
Then, we define the true iterations based on achieving an embedding of the Hessian and the gradient.  

\begin{definition}\label{def:true:CBGN}
Let $\epSTwo \in (0,1)$, $\sMax >0$. Iteration $k$ is ($\epSTwo, \sMax)$-true if

\begin{align}
    &\normTwo{S_k \MK \zK}^2 \geq (1-\epSTwo) \normTwo{\MK \zK}^2, \texteq{for all $z_k \in \R^{d+1}$} \label{tmp:CBGN:1}\\
    &\normTwo{S_k} \leq \sMax, \label{tmp:CBGN:10}
\end{align}

where $\MK = \squareBracket{\gradFK \quad \hessFK} \in \R^{d \times (d+1)}$. Note that all vectors are column vectors.
\end{definition}

\begin{remark}
 In \cite{cartis_randomised_2022}, a general random subspace framework is presented and analysed in terms of global rates of convergence, that can include cubic regularization variants, and for which the ensuing complexity bound is $\mathcal{O}(\epsilon^{-2})$. The latter result only requires  that the gradient of $f$ is Lipschitz continuous and that the embedding ensemble $\mathcal{S}$ satisfies \eqref{tmp:CBGN:10} and \eqref{tmp:CBGN:1} with $M_k:=[\gradFK]$ (so $S_k$ only needs to embed correctly the gradient vector, one-sidedly and with some probability). The stronger requirement in \eqref{tmp:CBGN:1}, namely, that
 $S_k$ embeds correctly not only the gradient but also the Hessian matrix, as well as the typical second-order assumption of Lipschitz continuous Hessian matrix, is used in order to achieve the optimal complexity of order $\epsilon^{-3/2}$ for R-ARC.
\end{remark}

\subsection{Some useful lemmas}

In this subsection we provide some useful results needed to prove our assumptions in \autoref{thm2}.

\begin{lemma} \label{succStepDecrease}
In \autoref{alg:R-ARC}, if iteration $k$ is successful, then
\[f(x_{k+1}) \leq \fun{f}{x_k} -\frac{\theta}{3\alphaK} \normTwo{S_k^T \sKHat}^3 \].
\end{lemma}

\begin{proof}
From the definition of successful iterations and \eqref{eq:cubic:sufficient_decrease}
\begin{align}
    f(x_{k+1}) 
    = f(x_k + s_k)
    & \leq \fK - \theta \squareBracket{\mKHat{0} - \mKHat{\sKHat}} - \frac{\theta}{3\alphaK} \normTwo{S_k^T \sKHat}^3 \notag \\
    & \leq \fK - \frac{\theta}{3\alphaK}\normTwo{S_k^T \sKHat}^3, \label{tmp:CBGN:4}
\end{align}
where in the last inequality, we used \eqref{tmp:CBGN:3}.
\end{proof}

The gradient of the model has the expression
\begin{equation}
    \grad \mKHat{\sKHat} = S_k \gradFK + S_k \hessFK S_k^T \sKHat + 
    \frac{1}{\alphaK} S_k S_k^T \sKHat \normTwo{\SK^T \sKHat}. 
    \label{mkGradCubic}
\end{equation}

The following lemma bounds the size of the step at true iterations.
\begin{lemma} \label{stepSizeSubEmbed}
Assume $f$ is twice continuosly differentiable with $\LH$-Lipschitz Hessian $\grad^2 f$ 
and $k < \nEps$.
Suppose that iteration $k$ is ($\epSTwo, \sMax)$-true. We have
\begin{equation}
\NsKTSKHat^2 \geq \frac{\epsilon}{2}
\min\set{\frac{2}{\LH}, 
\bracket{\frac{1}{\alphaK}\sMax + \kappaT}^{-1}
\sqrt{1 - \epSTwo}} \label{Lemma5.2.2_eqn}
\end{equation}

\end{lemma}

\begin{proof}
\eqref{mkGradCubic} and the triangle inequality give
\begin{align}
    \normTwo{\SK \gradFK + \SK \hessFK \SK^T \sKHat}  \notag
    &= \normTwo{\frac{1}{\alphaK}\SK \SK^T \sKHat \normTwo{\SK^T \sKHat} - \grad\mKHat{\sKHat}} \notag \\
    & \leq \frac{1}{\alphaK}\normTwo{\SK}\normTwo{\SK^T\sKHat}^2 + \normTwo{\grad \mKHat{\sKHat}}\notag \\
    & \leq \bracket{\frac{1}{\alphaK}\normTwo{\SK}+\kappaT} \normTwo{\SK^T \sKHat^2}^2 \texteq{by \eqref{tmp:CBGN:5}} \\
    & \leq \bracket{\frac{1}{\alphaK}\sMax+\kappaT} \normTwo{\SK^T \sKHat^2}^2,
    \label{tmp:CBGN:9}
\end{align}
where we used \eqref{tmp:CBGN:10}.
On the other hand, we have that
\begin{align}
    & \normTwo{\SK \gradFK + \SK \hessFK \SK^T \sKHat} \notag \\
    & = \normTwo{\SK\MK \squareBracket{1, 
    (\SK^T \sKHat)^T}^T} \notag \\
    & \geq \sqrt{\oneMinusEpSTwo} 
    \normTwo{\gradFK + \hessFK s_k} 
    \texteq{by \eqref{tmp:CBGN:1} with $z_k
    = \squareBracket{1, (\SK^T \sKHat)^T}^T$} 
    \notag\\
    & = \sqrt{\oneMinusEpSTwo} \normTwo{ \grad f(x_{k+1}) - \squareBracket{\grad f(x_{k+1}) - \gradFK - \hessFK s_k}} \\
    & \geq \sqrt{\oneMinusEpSTwo} 
        \left| 
            \normTwo{ \grad f(x_{k+1})} - 
            \normTwo{\squareBracket{\grad f(x_{k+1}) -     \gradFK - \hessFK s_k}} 
        \right|
    \label{tmp:CBGN:8}
\end{align}
Note that by Taylor's Theorem, because $f$ is twice continuously differentiable with $\LH$-Lipschitz $\grad^2 f$, we have that $\grad f (x_k + s_k) = \gradFK + \int_0^1 \grad^2 f(x_k + ts_k) s_k dt$. 
Therefore, we have 
\begin{align}
    \normTwo{\grad f(x_{k+1}) - \gradFK - \hessFK \sK} & 
    = \normTwo{\int_0^1 \squareBracket{\grad^2 f(x_k + ts_k) - \hessFK}s_k dt} \\
    & \leq \int_0^1 \normTwo{s_k} \normTwo{\grad^2 f(x_k + ts_k) - \hessFK} dt \\
    & \leq \normTwo{s_k} \int_0^1 \LH t \normTwo{s_k} dt \\
    & = \frac{1}{2}\LH\normTwo{s_k}^2
\end{align}
by Lipschitz continuity of $\grad^2 f$.
Next we discuss two cases,
\begin{enumerate}
    \item If $\LH \normTwo{\sK}^2 > \epsilon$, 
    then we have the desired result in \eqref{Lemma5.2.2_eqn}.
     
    \item If ${\LH} \normTwo{\sK}^2 \leq \epsilon$, then \eqref{tmp:CBGN:8}, 
    and the fact that $\normTwo{\grad f(x_{k+1})} \geq \epsilon$ 
    by $k < \nEps$, imply that 
        \[
            \normTwo{\SK\gradFK + \SK \hessFK \SK^T \sKHat} 
            \geq \sqrt{\oneMinusEpSTwo}\frac{\epsilon}{2}.
        \]
    Then \eqref{tmp:CBGN:9} implies
        \begin{equation}
            \normTwo{\SK^T \sK}^2 \geq \bracket{\frac{1}{\alphaK}\sMax + \kappaT}^{-1} 
            \sqrt{1 - \epSTwo}\frac{\epsilon}{2}.
            \notag
        \end{equation}
    This again gives the desired result.
\end{enumerate}

\end{proof}

\subsection{Satisfying the assumptions of \autoref{thm2}}
Here, for simplicity, we only address the case where $\cal{S}$ is the distribution of scaled Gaussian matrices; see our remark about other ensembles at the end of Section 4.3. 
Concerning scaled Gaussian matrices, we have the following standard results.

\begin{lemma}[Lemma 4.4.6 in \cite{Zhen-PhD}]\label{lem:GaussSMax:CubicSubspace}
Let $S \in \R^{l\times d}$ be a scaled Gaussian matrix (\autoref{def:Gaussian}). Then for any $\deltaSTwo > 0$, 
$S$ satisfies \eqref{tmp:CBGN:10} with probability $1-\deltaSTwo$ and 

\begin{equation}
    \sMax = 1 + \sqrt{\frac{d}{l}} + \sqrt{\frac{2\logOneOverDeltaSTwo}{l}}. \notag
\end{equation}

\end{lemma}

\begin{lemma}[Theorem 2.3 in \cite{10.1561/0400000060}] \label{lem:Gauss_embedding}
Let $\epSTwo \in (0,1)$ and $S \in \R^{l \times d}$ 
be a scaled Gaussian matrix. 
Then for any fixed $d \times (d+1)$ matrix $M$ with rank at most $r+1$, 
with probability $1-\deltaSThree$ we have that simultaneously for all 
$z \in \R^{d+1}$, $\normTwo{SMz}^2 \geq (1-\epSTwo) \normTwo{Mz}^2$, 
where
\begin{equation}
    \deltaSThree = \deltaSThreeExpression \label{deltaSThree:eq}
\end{equation}
and $C_l$ is an absolute constant. 
\end{lemma}

This result implies that scaled Gaussian matrices provide a subspace embedding of $M_k$ if $l\sim (r+1)/(\epsilon_S^{(2)})^2$.  

\paragraph{Satisfying \autoref{AA2} (page \pageref{AA2})}
\begin{lemma}
    Suppose that $\hessFK$ has rank at most $r\leq d$ for all $k$; $S \in \R^{l \times d}$ is drawn as a scaled Gaussian matrix. Let $\epSTwo, \deltaSTwo \in (0,1)$ such that $\deltaSTwo + \deltaSThree < 1$ where $\deltaSThree$ is defined in \eqref{deltaSThree:eq}.
    Then \autoref{alg:R-ARC} satisfies \autoref{AA2} with $\deltaS = \deltaSTwo + \deltaSThree$ and $S_{max} = 1 + \sqrt{\frac{d}{l}} + \sqrt{\frac{2\logOneOverDeltaSTwo}{l}}$, with true iterations defined in \autoref{def:true:CBGN}.
    \label{lem:arbitatratyDeltaS}
\end{lemma}

\begin{proof}
Let $x_k = \barXK \in \R^d$ be given. This determines $\gradFK, \hessFK$ and hence $M_k$. As $\hessFK$ has rank at most $r$, $M_k$ has rank at most $r+1$. 
Consider the events
\begin{align*}
    &\aKOne = \set{\normTwo{S_k M_k z}^2 \geq (1-\epSTwo) \normTwo{\MK z}^2, \quad \forall z\in \R^{d+1}} \\
    &\aKTwo = \set{\normTwo{S_k} \leq S_{max}}.
\end{align*}
Note that iteration $k$ is true if and only if $\aKOne$ and $\aKTwo$ occur. It follows from \autoref{lem:Gauss_embedding} that $\probabilityGivenXK{\aKOne} \geq 1-\deltaSThree$; and from \autoref{lem:GaussSMax:CubicSubspace} that $\probability{\aKTwo} \geq 1-\deltaSTwo$. Since $\aKTwo$ is independent of $x_k$, we have $\probabilityGivenXK{\aKTwo} = \probability{\aKTwo} \geq 1-\deltaSTwo$. 

Hence, we have $\probabilityGivenXK{{\aKOne}\intersect{\aKTwo}} \geq 1 - \probabilityGivenXK{\complement{\aKOne}} - \probabilityGivenXK{\complement{\aKTwo}} \geq 1- \deltaSTwo - \deltaSThree$. A similar argument shows that $\probability{{\aZeroOne}\intersect{\aZeroTwo}} \geq 1 - \deltaSTwo - \deltaSThree$, as $x_0$ is fixed. 

Moreover, given $x_k = \barXK$, $\aKOne$ and $\aKTwo$ only depend on $S_k$, which is drawn randomly at iteration $k$. Hence given $x_k = \barXK$, ${\aKOne}\intersect{\aKTwo}$ is independent of whether the previous iterations are true or not. Hence \autoref{AA2} is true.
\end{proof}

\paragraph{Satisfying \autoref{AA3} (page \pageref{AA3})}

\begin{lemma}\label{lem:cubic:sub:A2}
Let $f$ be twice continuously differentiable with $\LH$-Lipshitz continuous Hessian $\grad^2 f$. \autoref{alg:R-ARC} satisfies \autoref{AA3} with
    \begin{equation}
        \alphaLow = \frac{2(1-\theta)}{\LH} \label{eq:alphaLow:CBGN}
    \end{equation}
\end{lemma}

\begin{proof}
From \eqref{tmp:CBGN:3}, we have that 
    \begin{equation}
        f(x_k) - \qKHat{\sKHat} \geq \frac{1}{3\alphaK} 
        \normTwo{\SK^T \sKHat}^3.
        \notag
    \end{equation}
Using \autoref{stepSizeSubEmbed}, in true iterations with $k < \nEps$, we have that 
$
    \normTwo{\SK^T \sKHat} >0.\notag
$
Then, since $f(x_k)=\qKHat{0}$, $\rho_k$ below\footnote{Note that the right-hand side of \eqref{eq:cubic:sufficient_decrease}, which is
 the denominator of \eqref{eqn:rho_k:CBGN}, may be zero before termination, on account of sketching/subspace techniques being used.},
    \begin{equation}
        \rho_k := \frac{f(x_k) - f(x_k + s_k)}{f(x_k) - \qKHat{\sKHat}},
        \label{eqn:rho_k:CBGN}
    \end{equation}
 is well defined,   and
    \begin{equation}
        \abs{1 - \rho_k} = \frac{\abs{f(x_k + s_k) - \qKHat{\sKHat}}}
        {\abs{f(x_k) - \qKHat{\sKHat}}}. \notag
    \end{equation}
The numerator can be bounded by 
    \begin{align}
        \abs{f(x_k + s_k) - \qKHat{\sKHat}}  \leq \frac{1}{6}\LH \normTwo{\sK}^2,\notag
    \end{align}
    by Corollary A.8.4 in \cite{cartis_evaluation_2022}.

Therefore, we have 
\begin{equation}
    \abs{1 - \rho_k} \leq \frac{\frac{1}{6}\LH \normTwo{\sK}^3}{\frac{1}{3\alphaK}
    \normTwo{\sK}^3} = \frac{1}{2}\alphaK \LH \leq 1-\theta 
    \texteq{by \eqref{eq:alphaLow:CBGN} and $\alphaK \leq \alphaLow$}. 
\end{equation}

Thus $1-\rho_k \leq \abs{1-\rho_k} \leq 1-\theta$ so $\rho_k \geq \theta $ and iteration $k$ is successful.
\end{proof}

\paragraph{Satisfying \autoref{AA4} (page \pageref{AA4})}

\begin{lemma}
    Let $f$ be twice continuously differentiable with $\LH$-Lipschitz continuous Hessian. 
    \autoref{alg:R-ARC} with true iterations defined in \autoref{def:true:CBGN} satisfies \autoref{AA4} with
    
    \begin{equation}
        h(\epsilon, \alphaK) = \frac{\theta}{3\alphaMax} 
        \bracket{\frac{\epsilon}{2}}^{3/2} \min \set{ \frac{2^{3/2}}{\LH^{3/2}}, \bracket{\frac{\sqrt{\oneMinusEpSTwo}}{\frac{1}{\alphaK}\sMax +\kappaT}}^{3/2}}.
        \label{eq:hEpsAlphaCB}
    \end{equation}
    
\end{lemma}

\begin{proof}
For true and successful iterations with $k < \nEps$, use \autoref{stepSizeSubEmbed} with \autoref{succStepDecrease} and $\alphaK \leq \alphaMax$.
\end{proof}

\paragraph{Satisfying \autoref{AA5} (page \pageref{AA5})}
The next lemma shows that the function value following \autoref{alg:R-ARC} is non-increasing.

\begin{lemma}\label{tmp-2022-1-15-1}
\autoref{alg:R-ARC} satisfies \autoref{AA5}.
\end{lemma}

\begin{proof}
In \autoref{alg:R-ARC}, we either have $x_{k+1} = x_k$ when the step is unsuccessful, 
in which case $f(x_k) = f(x_{k+1})$; or the step is successful,
in which case we have $\fun{f}{x_{k+1}} - \fun{f}{x_k} \leq 0$ 
by \autoref{succStepDecrease}.
\end{proof}

\subsection{Iteration complexity of \autoref{alg:R-ARC} for first-order criticality}
We have shown that \autoref{alg:R-ARC} satisfies \autoref{AA2}, \autoref{AA3}, \autoref{AA4} and \autoref{AA5}. Noting that \autoref{alg:R-ARC} is a particular case of \autoref{alg:generic}, we apply \autoref{thm2} to arrive at the main result of this section.

\begin{theorem} \label{thm:CBGN_subspace_first}
    Let $\cal{S}$ be the distribution of scaled Gaussian matrices $S \in \R^{l \times d}$ as in \autoref{def:Gaussian}. Suppose that $f$ is bounded below by $f^*$, is twice continuously differentiable with Lipschitz-continuous Hessian $\grad^2 f$ (with Lipschitz constant $\LH$), and that $\hessFK$ has rank at most $r\leq d$ for all $k$ and let $\epsilon>0$.
    Choose $l = 4 C_l (\log16 + r + 1); 
    \epSTwo=\frac{1}{2}; 
    \deltaSTwo = \frac{1}{16};
    $ so that
    $\deltaSThree = \deltaSThreeExpression = \frac{1}{16};
    \deltaS = \frac{1}{8}<\frac{c}{(c+1)^2}; 
    \sMax = 1 + 
    \frac{ \sqrt{d} + \sqrt{2\log16} }
    {\sqrt{4 C_l \bracket{ \log 16 + r + 1 } }}$, where $C_l$ is defined in \eqref{deltaSThree:eq}. Apply \autoref{alg:R-ARC}  to minimizing $f$ for $N$ iterations. 
	Then for any $\delta_1 \in (0,1)$ with 
    \begin{equation}
        g(\deltaOne) >0, \nonumber
    \end{equation}
    where 
    \begin{equation}
        g(\deltaOne) = \nPreFactorTRCubic, \nonumber
    \end{equation}
    if $N\in \N$ satisfies 
    \begin{equation}
        N \geq g(\deltaOne) \squareBracket{
             \fZeroMinusfStarOverH
             + \frac{4 C_l (\log16 + r + 1)}{1+c}}, \nonumber
    \end{equation}
    where $h(\epsilon, \alpha_k)\sim \epsilon^{3/2}$ is defined in \eqref{eq:hEpsAlphaCB} with $\epSTwo, \sMax$ defined in the theorem statement, $\alphaLow$ is given in \eqref{eq:alphaLow:CBGN} and
     $\alphaMin = \alphaZero \gammaOne^\newL$ associated with $\alphaLow$,  for some $\newL \in \N^+$,
     we have that
    \begin{equation}
        \probability{\min_{k\leq N} \{\normTwo{\grad f(x_{k+1})}\} \leq \epsilon } \geq 1 - e^{-\frac{7\delta_1^2}{16} N}. \nonumber
    \end{equation}
\end{theorem}
Thus, with exponentially high probability, R-ARC takes $\mathcal{O}(\epsilon^{-3/2})$  iterations to drive the gradient below desired accuracy $\epsilon$. However, since the subspace dimension $l$ is fixed and proportional to the rank of the Hessian matrix $r$, the Hessian matrix $\grad^2 f$ needs  to have a lower rank $r$ than the full space dimension $d$, as otherwise \autoref{alg:R-ARC} does not save computational effort/gradient/Hessian evaluations (per iteration) compared to the full-dimensional version.

\paragraph{Dependency on problem dimension} The number of iterations required by \autoref{alg:R-ARC} for \autoref{thm:CBGN_subspace_first} depends on the dimension $d$, even though the required sketch dimension size $l$ only depends on upon the bound $r$ of the Hessian's rank. This dependency arises through \autoref{eq:hEpsAlphaCB}, where we have that $h(\epsilon, \alphaK) \propto \sMax^{-3/2}$. Hence we have the requirement in \ref{thm:CBGN_subspace_first} that $N \propto \sMax^{3/2}$. For Gaussian matrices, we have that $\sMax \propto \sqrt{\frac{d}{l}}$ and hence $N \propto \left( \frac{d}{l} \right)^{3/4}$.

\paragraph{Using other random ensembles than the scaled Gaussian matrices in \autoref{alg:R-ARC}}
Although \autoref{thm:CBGN_subspace_first} requires $\cal{S}$ to be the distribution of scaled Gaussian matrices, qualitatively similar results, namely, convergence  rates of order $\mathO{\epsilon^{-3/2}}$ with exponentially high probability, can be established for example, for $s$-hashing matrices with $s=\mathcal{O}(\log d)$  nonzeros per column, Subsampled Randomized Hadamard Transforms, Hashed Randomized Hadamard Transforms, Haar matrices and more; see \autoref{sec:sketching_matrices} for definitions of these ensembles. The proof for satisfying \autoref{AA2} needs to be modified, using the upper bounds for $S_{max}$ and the subspace embedding properties of these ensembles instead. Consequently, the constants in \autoref{thm:CBGN_subspace_first} will change, but the convergence rate and the form of the result stays the same; see Section 4.4.2 in \cite{Zhen-PhD} for more details.
We note though that scaled sampling matrices require much stronger problem assumptions than the rest of the ensembles and so are not to be preferred in general \cite{cartis_randomised_2022, Zhen-PhD}.

\subsection{R-ARC applied to low-rank functions}
\label{sec:low_rank}

\autoref{thm:CBGN_subspace_first} requires that the rank of $\hessFK$ is bounded by $r$ for all iterates $\xK$; this trivially holds for $r = d$ where $d$ is the problem dimension, but raises the questions for what classes of functions does this assumption hold with $r \ll d$. In these cases, a significant dimensionality reduction can be achieved by R-ARC whilst still attaining the optimal $\mathcal{O}(\epsilon^{-\frac{3}{2}})$ convergence rate to reach an $\epsilon$-approximate first-order minimizer.

We return to low-rank functions as defined in \autoref{def:low:rank}. This is a broad class of functions that necessarily have a bound on the rank of $\hessFK$, meaning that \autoref{thm:CBGN_subspace_first} provides a significant dimensionality reduction. An alternative and equivalent characterization  is now given.

\begin{lemma}[\cite{cartis_learning_2024, cosson_gradient_2022}]\label{lem:low:rank:characterisation}
    A function $f:\R^{d} \rightarrow \R$ has rank $r$, with $r \leq d$ if and only if there exists a matrix $A \in \R^{r \times d}$ and a map $\sigma:\R^{r} \rightarrow \R$ such that $f(x) = \sigma(Ax)$ for all $x \in \R^{d}$.
\end{lemma}

Using this alternative characterization, we are able to show that low-rank functions  have a Hessian with that maximal rank.

\begin{lemma}\label{lem:low:rank:hessian}
   If $f:\R^{d} \rightarrow \R$ has rank $r$, and $f$ is $C^2$, then for all $x \in \R^{d}$, $\grad^{2} f(x)$ has rank at most $r$.
\end{lemma}
\begin{proof}
Using \autoref{lem:low:rank:characterisation}, we can write
\begin{equation}
    f(x) = \sigma(Ax)
\end{equation}
where $A \in \mathbb{R}^{r \times d}$ and is hence of rank $<= r$. We have
\begin{equation}
    \grad f(x) =
    \frac{\partial \sigma(Ax)}{\partial x} =
    A^{\top} \cdot \frac{\partial \sigma(Ax)}{\partial (Ax)} =
    A^{\top} \grad \sigma(Ax)
\end{equation}
Furthermore, as $f$ is $C^2$ by assumption, we have

\begin{equation}
    \grad^{2} f(x) =
    \frac{\partial \grad f(x)}{\partial x^T} =
    A^{\top} \cdot  \frac{\partial \grad\sigma(Ax)}{\partial x^T} =
    A^{\top} \cdot \frac{\partial \grad\sigma(Ax)}{\partial (Ax)^T} \cdot A = A^{\top}[\grad^{2}\sigma(Ax)]A.
\end{equation}
As $\rank(A) \leq r$, we can conclude that $\grad^{2} f(x)$ has rank bounded by $r$.
\end{proof}

Using \autoref{lem:low:rank:hessian}, we show the following optimal convergence result for low-rank functions.

\begin{corollary}\label{cor:low_rank_first_convergence}
    Let $\cal{S}$ be the distribution of scaled Gaussian matrices $S \in \R^{l \times d}$ as in \autoref{def:Gaussian}. Suppose that $f$ is a low-rank function of rank $r$, bounded below by $f^*$, and twice continuously differentiable with Lipschitz continuous Hessian $\grad^2 f$ (with constant $\LH$) and let $\epsilon>0$.
    Choose $l = 4 C_l (\log16 + r + 1); 
    \epSTwo=\frac{1}{2}; 
    \deltaSTwo = \frac{1}{16};
    $ so that
    $\deltaSThree = \deltaSThreeExpression = \frac{1}{16};
    \deltaS = \frac{1}{8}<\frac{c}{(c+1)^2}; 
    \sMax = 1 + 
    \frac{ \sqrt{d} + \sqrt{2\log16} }
    {\sqrt{4 C_l \bracket{ \log 16 + r + 1 } }}$, where $C_l$ is defined in \eqref{deltaSThree:eq}. Apply \autoref{alg:R-ARC} to minimizing $f$ for $N$ iterations. 
	Then for any $\delta_1 \in (0,1)$ with 
    \begin{equation}
        g(\deltaOne) >0, \nonumber
    \end{equation}
    where 
    \begin{equation}
        g(\deltaOne) = \nPreFactorTRCubic, \nonumber
    \end{equation}
    if $N\in \N$ satisfies 
    \begin{equation}
        N \geq g(\deltaOne) \squareBracket{
             \fZeroMinusfStarOverH
             + \frac{4 C_l (\log16 + r + 1)}{1+c}}, \nonumber
    \end{equation}
    where $h(\epsilon, \alpha_k)\sim \epsilon^{3/2}$ is defined in \eqref{eq:hEpsAlphaCB} with $\epSTwo, \sMax$ defined in the theorem statement, $\alphaLow$ is given in \eqref{eq:alphaLow:CBGN} and
     $\alphaMin = \alphaZero \gammaOne^\newL$ associated with $\alphaLow$,  for some $\newL \in \N^+$,
     we have that
    \begin{equation}
        \probability{\min_{k\leq N} \{\normTwo{\grad f(x_{k+1})}\} \leq \epsilon } \geq 1 - e^{-\frac{7\delta_1^2}{16} N}. \nonumber
    \end{equation}
\end{corollary}
\begin{proof}
    As $f$ is a low-rank function of rank $r$, by \autoref{lem:low:rank:hessian} we have that $\hessFK$ has rank at most $r$ for all $k$ and hence we can apply \autoref{thm:CBGN_subspace_first} to $f$.
\end{proof}

We have hence shown that low-rank functions are one such class of functions that enable \autoref{alg:R-ARC} to drive $\normTwo{\gradFK}$ below $
\epsilon$ in $\mathO{\epsilon^{-3/2}}$ iterations while allowing significant dimensionality reduction per iteration.

\section{R-ARC-D: a Dynamic R-ARC Variant}

Throughout the proof of \autoref{thm:CBGN_subspace_first}, the sketch dimension $l$ was fixed for the Gaussian sketching matrices. In \autoref{cor:low_rank_first_convergence}, it was therefore assumed that $l = \mathcal{O}(r)$ where $r\leq d$ was the rank of the function. It is difficult to ensure this condition holds unless $r$, or at least an upper bound for it, is known. In this section, we propose a R-ARC variant that chooses the sketch dimension adaptively, based on the  rank of the sketched local Hessian, which can be easily calculated. The following simple modification of R-ARC dispenses with the need to know $r$ a priori by increasing the sketch dimension across iterations.

\begin{algorithm}[H]
\begin{description}

\item[Initialization] \ \\
 Choose an initial sketch dimension $l_1$, a family of matrix distributions $\mathcal{S}_{l}$ of matrices $S\in \rLTimesD$ for $l \geq 1$.
 Choose constants $\gamma_1\in (0,1)$, $\gamma_2 >1$, $\theta \in (0,1)$, $\kappaT, \kappaS \geq 0$
 and $\alpha_{\max}>0$ such that $
        \gamma_2 = \oneOverGammaOneC,
        $
    for some $c\in \N^+$.
 Initialize the algorithm by setting 
 $x_0 \in \R^d$, 
 $\alpha_0 = \alphaMax \gamma^p$
 for some $p\in \N^+$ and $k=0$.

\item[1. Run an iteration of R-ARC] \ \\
    Perform steps 1-3 of \autoref{alg:R-ARC} with a sketch matrix $S_k\in \R^{l_k\times d}$.

\item[2. Increase dimension of sketch]\ \\  
    Choose $l_{k+1} \geq l_k$, let $k:=k+1$.

\caption{\bf{R-ARC with varying sketch dimension (R-ARC-D)}} \label{alg:R-ARC-D} 
\end{description}
\end{algorithm}

The advantage of \autoref{alg:R-ARC-D} is that we do not need to know the function rank $r$ beforehand. Assuming the function $f$ is of rank $r\leq d$, and we increase $l_k$ with each iteration, we will eventually reach the appropriate sketch dimension $l$ in \autoref{cor:low_rank_first_convergence}. We therefore have the following theorem, where the sketch dimension is increased by one on each iteration.

\begin{theorem}
\label{thm:low_rank_first_convergence_increase_sketch}
    Let $\mathcal{S}_{l}$ be the distribution of scaled Gaussian matrices $S \in \R^{l \times d}$ defined in \autoref{def:Gaussian}, $l\geq 1$. Suppose that $f$ is a low-rank function of rank $r\leq d$, bounded below by $f^*$, and twice continuously differentiable with Lipschitz-continuous Hessian $\grad^2 f$ (with constant $\LH$) and let $\epsilon>0$.
    Choose $\epSTwo=\frac{1}{2}; 
    \deltaSTwo = \frac{1}{16};
    $ so that
    $\deltaSThree = \deltaSThreeExpression = \frac{1}{16};
    \deltaS = \frac{1}{8}<\frac{c}{(c+1)^2}; 
    \sMax = 1 + 
    \frac{ \sqrt{d} + \sqrt{2\log16} }
    {\sqrt{4 C_l \bracket{ \log 16 + r + 1 } }}$, where $C_l$ is defined in \eqref{deltaSThree:eq}. Apply \autoref{alg:R-ARC-D} to minimizing $f$ for $N$ iterations with $l_k = k$ for all $k \geq 1$. 
	Then for any $\delta_1 \in (0,1)$ with 
    \begin{equation}
        g(\deltaOne) >0, \nonumber
    \end{equation}
    where 
    \begin{equation}
        g(\deltaOne) = \nPreFactorTRCubic, \nonumber
    \end{equation}
    if $N\in \N$ satisfies 
    \begin{equation}
        N \geq g(\deltaOne) \underbrace{\squareBracket{
             \fZeroMinusfStarOverH
             + \frac{4 C_l (\log16 + r + 1)}{1+c}}}_{N'} + 4 C_l (\log16 + r + 1) - 1, \nonumber
    \end{equation}
    where $h(\epsilon, \alpha_k)\sim \epsilon^{3/2}$ is defined in \eqref{eq:hEpsAlphaCB} with $\epSTwo, \sMax$ defined in the theorem statement, $\alphaLow$ is given in \eqref{eq:alphaLow:CBGN} and
     $\alphaMin = \alphaZero \gammaOne^\newL$ associated with $\alphaLow$,  for some $\newL \in \N^+$,
     we have that
    \begin{equation}
        \probability{\min_{k\leq N} \{\normTwo{\grad f(x_{k+1})}\} \leq \epsilon } \geq 1 - e^{-\frac{7\delta_1^2}{16} N'}. \nonumber
    \end{equation}
\end{theorem}
\begin{proof}
    After the first $4 C_l (\log16 + r + 1) - 1$ iterations of algorithm \autoref{alg:R-ARC-D}, we have that $l = 4 C_l (\log16 + r + 1)$, meaning the conditions of \autoref{cor:low_rank_first_convergence} are met. For $l > 4 C_l (\log16 + r + 1)$, we have that \autoref{lem:GaussSMax:CubicSubspace} and \autoref{lem:Gauss_embedding} still hold for the particular $\sMax$ and $\deltaSThree$ as stated for $l = 4 C_l (\log16 + r + 1)$.
\end{proof}

The dependence of the required number of iterations $N$ on the dimension $d$ is at most proportional now to $\left ( \frac{d}{r} \right )^{3/4}$ as $l_k$ is at least proportional to $r$ asymptotically.

\subsection{An adaptive sketch size update rule for R-ARC-D}

In \autoref{thm:low_rank_first_convergence_increase_sketch}, we showed that if the sketch dimension is increased by $1$ on each iteration, we attain the optimal $\mathcal{O}(\epsilon^{-\frac{3}{2}})$ rate of convergence for low-rank functions or general ones. However, the sketch dimension update rule is not practical as the sketch dimension shall increase until it is equal to the dimension of the function domain $d$, which is not needed if the function has rank $r<d$. We now detail an update rule that can be used to update the dimension of the sketches $l_k$ in  such a way that the sketch sizes do not grow beyond what is strictly needed for a similar result to \autoref{thm:CBGN_subspace_first} to be applicable.

In \autoref{alg:R-ARC-D}, we define the following for $k \geq 0:$
\begin{equation*}
    r_k := rank(\hessFK);\quad \hat{r}_k := rank(\SK \hessFK \SK^{\top});\quad \hat{R}_k := \max_{0 \leq j \leq k}\hat{r}_j.
\end{equation*}
Using these, we can give the following update rule:
\begin{equation}\label{eq:l_k:update:step}
    l_{k+1} = \begin{cases}
        \max(C\hat{R}_{k} + D, l_{k})&\text{if $\hat{R}_k > \hat{R}_{k-1}$, or $k=0$} \\
        l_k &\text{otherwise.}
    \end{cases}    
\end{equation}
where $C, D \geq 1$ are user-chosen constants\footnote{We may set $l_{k+1} = \ceil{C\hat{R}_{k} + D}$ in the update scheme if $C$ or $D$ are non integer.}. In the case that $C = 4C_{l}$, $D = 4C_{l}(\log16 + 1)$, we have that $C\hat{R}_{k} + D = 4 C_l (\log16 + \hat{R}_{k} + 1)$, closely resembling the $4 C_l (\log16 + r + 1)$ sketch size in \autoref{thm:CBGN_subspace_first}. By setting $C, D$ to these particular values, we shall prove a $\mathcal{O}(\epsilon^{-\frac{3}{2}})$ convergence rate result for R-ARC-D using this update rule.
 
A key part is proving that if $l_{k} < Cr_{k} + D$, the sketch size will increase for subsequent iterations with probability 1. As we only have access to $r_{k}$ through $\hat{r}_{k}$, we must show that $\SK \hessFK \SK^{\top}$ preserves information on the rank of $\hessFK$, so that $\hat{r}_{k}$ preserves information from $r_{k}$. We state the following two Lemmas before proving \autoref{lem:rank_preserve} which gives the required result.

\begin{lemma} \label{lem:orth:invariance}
    Let $S \in \mathbb{R}^{l \times d}$ be a scaled Gaussian matrix and let $Q \in \mathbb{R}^{d \times d}$ be an orthogonal matrix, then $SQ$ is also a scaled Gaussian matrix.
\end{lemma}

\begin{lemma}[\cite{caron_zero_2005}]\label{lem:polynomial:zeros}
   For any $n \geq 0$, a polynomial $p$ from $\R^{n} \rightarrow \R$, is either identically 0, or non-zero almost everywhere. That is, $\mu(Z) = 0$ where $Z$ is the set of zeros and $\mu$ is the Lebesgue measure on $\R^{n}$. 
\end{lemma}

We can use these two results in proving the following:

\begin{lemma}\label{lem:rank_preserve}
    Let $A \in \R^{d \times d}$ be a symmetric matrix of rank $r$ and let $S \in \R^{l \times d}$ be a scaled Gaussian matrix with $l \leq d$. Let $m := \min(l, r)$, then we have:
    \begin{equation}
        \mathbb{P}(\rank(SAS^{\top}) = m) = 1.
    \end{equation}
\end{lemma}

\begin{proof}
    It is immediate from the dimension of $S$ that $\rank(SAS^{\top}) \leq m$ with probability 1. As $A$ is symmetric, we can take an eigenvalue decomposition $A = V\Lambda V^{\top}$ where $\Lambda$ is a diagonal matrix with $\Lambda_{ii} \neq 0 \iff i \leq r$ and $V$ is an orthogonal matrix. Using this eigenvalue decomposition and \autoref{lem:orth:invariance}, we have:
    \begin{equation}
        SAS = SV\Lambda V^{\top}S^{\top} \stackrel{d}{\sim} S\Lambda S^{\top}
    \end{equation}

    Let $\tilde{S} \in \R^{m \times d}$ denote the submatrix formed by taking the first $m$ rows of $S$. We can observe that $\tilde{S}\Lambda \tilde{S}^{\top}$ is the leading $m \times m$ submatrix of $S\Lambda S^{\top}$. Hence we have that 
    \begin{equation}
        \rank(\tilde{S}\Lambda \tilde{S}^{\top}) \leq \rank(S\Lambda S^{\top}) \leq m.
        \label{eq:sandwich}
    \end{equation}
    
    We now make the observation that $\rank(\tilde{S}\Lambda \tilde{S}^{\top}) = m \iff \det(\tilde{S}\Lambda \tilde{S}^{\top}) \neq 0$. By flattening $\tilde{S}$, we can consider the elements as forming a vector $v_{\tilde{S}}$ in $\R^{dm}$. We can further consider $\det(\tilde{S}\Lambda \tilde{S}^{\top})$ as a polynomial $p: \R^{dm} \rightarrow \R$ with $p(v_{\tilde{S}}) = \det(\tilde{S}\Lambda \tilde{S}^{\top})$.

    To apply \autoref{lem:polynomial:zeros}, we must show that $p$ is non-zero. To see this, suppose $\tilde{S} = [I_{m\times m} : 0_{m \times (d - m)}]$, then $\tilde{S}\Lambda \tilde{S}^{\top}$ is equal to the leading $m \times m$ submatrix of $\Lambda$, which is full rank (and hence has non-zero determinant). Hence $p$ is non-zero and we have that the set of zeros of $p$ is a null set with respect to the Lebesgue measure on $\R^{dm}$.
    
    As the entries of $\tilde{S}$ are Gaussian, the law of $v_{\tilde{S}}$ is absolutely continuous with respect to the Lebsegue measure on $\R^{dm}$. Hence, $\mathbb{P}(\det (\tilde{S}\Lambda \tilde{S}^{\top}) = 0) = 0$. It follows that $\mathbb{P}(\rank(\tilde{S}\Lambda \tilde{S}^{\top}) = m) = 1$ and hence by \autoref{eq:sandwich}, $\mathbb{P}(\rank(SAS^{\top}) = m) = 1$.
\end{proof}

This Lemma demonstrates that in the case of Gaussian sketching matrices, access to $\hat{r}_{k}$ preserves information on $r_{k}$. We can prove an additional Lemma demonstrating how this information enables the desired behaviour of the update scheme.
\begin{lemma}\label{lem:lk:increases}
    Set $l_0 \geq 1$ and suppose that the update rule in \autoref{eq:l_k:update:step} is applied to $l_k$. For all $k \geq 1$, if $l_{k} < Cr_{k} + D$, then with probability 1, $\hat{R}_{k} > \hat{R}_{k-1}$.
\end{lemma}
\begin{proof}
    By the update rule in \autoref{eq:l_k:update:step}, we have that $l_{k} < Cr_{k} + D \implies \hat{R}_{k-1} < r_{k}$. Also, $l_{k} \geq \hat{R}_{k-1} + 1$ as $C, D \geq 1$. It follows that $l_{k} < Cr_{k} + D \implies \hat{R}_{k-1} < \min(l_{k}, r_{k})$.  Hence, by \autoref{lem:rank_preserve}, $\hat{R}_{k} = \min(l_{k}, r_{k}) > \hat{R}_{k-1}$ with probability 1.
\end{proof}

Before stating and proving a convergence result, we first introduce the variable $\delta_{S, k}^{(3)} := \pow{e}{-\frac{l_{k}(\epSTwo)^2}{C_l} + r_{k} +1 }$ to adjust the constant $\deltaSThree$ to account for the fact that the sketch dimension can increase through iterations. 

\begin{theorem}\label{thm:efficient:dimension:increase}
    Let $\cal{S}$ be the distribution of scaled Gaussian matrices $S \in \R^{l \times d}$ defined in \autoref{def:Gaussian}. Suppose that $f$ is a low-rank function of rank $r$, bounded below by $f^*$, twice continuously differentiable with Lipschitz continuous $\grad^2 f$ (with constant $\LH$) and let $\epsilon>0$.
    Choose $l_{0} \geq 1; \epSTwo=\frac{1}{2}; 
    \deltaSTwo = \frac{1}{16};
    $ so that
    $\delta_{S, k}^{(3)} := \pow{e}{-\frac{l_{k}(\epSTwo)^2}{C_l} + r_{k} +1 };
    \deltaS = \frac{1}{8} <\frac{c}{(c+1)^2}; 
    \sMax = 1 + 
    \frac{ \sqrt{d} + \sqrt{2\log16} }
    {\sqrt{l_{0}}}$, where $C_l$ is defined in \eqref{deltaSThree:eq}. Apply \autoref{alg:R-ARC-D} to minimizing $f$ for $N$ iterations with the update for $l_k$ given by \autoref{eq:l_k:update:step} with $C = 4C_{l}$, $D = 4C_{l}(\log16 + 1)$. 
	Then for any $\delta_1 \in (0,1)$ with 
    \begin{equation}
        g(\deltaOne) >0, \nonumber
    \end{equation}
    where 
    \begin{equation}
        g(\deltaOne) = \nPreFactorTRCubic, \nonumber
    \end{equation}
    if $N\in \N$ satisfies 
    \begin{equation}
        N \geq \underbrace{g(\deltaOne) \squareBracket{
             \fZeroMinusfStarOverH
             + \frac{4 C_l (\log16 + r + 1)}{1+c}}}_{N'} + r + 1, \nonumber
    \end{equation}
    where $h(\epsilon, \alpha_k)\sim \epsilon^{3/2}$ is defined in \eqref{eq:hEpsAlphaCB} with $\epSTwo, \sMax$ defined above, $\alphaLow$ is given in \eqref{eq:alphaLow:CBGN} and
     $\alphaMin = \alphaZero \gammaOne^\newL$ associated with $\alphaLow$,  for some $\newL \in \N^+$.
    Then we have that
    \begin{equation}
        \probability{\min_{k\leq N} \{\normTwo{\grad f(x_{k+1})}\} \leq \epsilon } \geq 1 - e^{-\frac{7\delta_1^2}{16} N'}. \nonumber
    \end{equation}    
\end{theorem}

\begin{proof}
    By the statement theorem, we have that if $l_{k} \geq 4 C_l (\log16 + r_{k} + 1)$, then $\delta_{S, k}^{(3)} \leq \frac{1}{16}$. Applying
    \autoref{lem:lk:increases} with $C = 4C_{l}$, $D = 4C_{l}(\log16 + 1)$ gives that with probability 1, $l_{k} < 4 C_l (\log16 + r_{k} + 1)$ implies $\hat{R}_{k} > \hat{R}_{k-1}$, namely, $\hat{R}_{k}$ increases. As $0 \leq \hat{R}_{k} \leq r$, there can be at most $r + 1$ values of $k$ for which $\hat{R}_{k}$ increases.

    Suppose the algorithm is ran for $K + r + 1$ iterations for some $K \leq 0$, then with probability 1, there is some subsequence $(k_{p})_{1 \leq p \leq K}$ such that $l_{k_{p}} \geq 4 C_l (\log16 + r_{k} + 1)$ for all $p$ and hence $\delta_{S, k_{p}}^{(3)} \leq \frac{1}{16}$ for these iterations.

    By \autoref{tmp-2022-1-15-1}, we have that the function value does not increase on the remaining $r + 1$ iterates. Hence at most $r + 1$ additional iterations are required from the bound in \autoref{thm:CBGN_subspace_first}.
\end{proof}

The result of \autoref{cor:low_rank_first_convergence} is significant because it demonstrates that \autoref{alg:R-ARC-D} can attain the optimal $\mathcal{O}(\epsilon^{-\frac{3}{2}})$ rate of convergence for low-rank functions, even in the case where the rank $r$ is not known beforehand. Furthermore, the update scheme is efficient in that the sketch dimension $l_{k}$ will not increase more than what is needed. In the worst case, the dependence of the required number of iterations $N$ on the dimension $d$ is proportional now to $\left ( \frac{d}{l_0} \right )^{3/4}$.

\section{Global Convergence Rate of R-ARC to 
Second-Order Critical Points}

\subsection{R-ARC convergence to second-order (subspace) critical points}

In this section, we show that 
\autoref{alg:R-ARC} converges to a (subspace) second-order critical point of
$f(x)$.
We aim to upper bound the number of iterations until the following approximate second-order criticality condition is achieved,
\begin{equation}
    \nEps = \nTwoEpsH = \min \set{k: \lambdaMin{S_k \hessFK S_k^T} \geq -\epH},
    \label{tmp-2021-12-31-6}
\end{equation}
where $\epH>0$ is an accuracy tolerance for the second-order condition (and may be different than the $\epsilon>0$ for first-order optimality criteria).
And we define $\sMax$-true iterations as follows.

\begin{definition} \label{def:true:cubic:subspaceHess}
Let $\sMax>0$. An iteration $k$ is true if $\normTwo{S_k} \leq \sMax$.
\end{definition}

Compared to Section \ref{sec:rarc-convergence}, here we have a less restrictive definition of true iterations. Consequently, it is easy to show \autoref{AA2} is true.

\paragraph{Satisfying \autoref{AA2}}
For $S$ being a scaled Gaussian matrix, \autoref{lem:GaussSMax:CubicSubspace} gives that \autoref{alg:R-ARC} satisfies \autoref{AA2} 
with any $\deltaS \in (0,1)$ and 
\begin{equation}
    \sMax = 1 + \sqrt{\frac{d}{l}} + \sqrt{\frac{2\logOneOverDeltaS}{l}}. \notag
\end{equation} 

\paragraph{Satisfying \autoref{AA3}}

\begin{lemma}
Let $f$ be twice continuously differentiable with Lipschitz continuous Hessian (with constant $\LH$).
\autoref{alg:R-ARC} satisfies \autoref{AA3} with
\begin{equation}
    \alphaLow = \frac{2(1-\theta)}{\LH} \label{eq:alphaLow:CBGN:secondOrder}
\end{equation}

\end{lemma}

The proof is similar to \autoref{lem:cubic:sub:A2}, where the condition $\normTwo{S_k^T \sKHat} > 0$ on true iterations before convergence is ensured by  \autoref{lem:SubHessNegImpStepLower}.

\paragraph{Satisfying \autoref{AA4}}
We calculate
\begin{equation}
    \grad^2 \mKHat{\sHat} = 
    S_k \hessFK S_k^T
    + \frac{1}{\alphaK} \squareBracket{
    \normTwo{S_k^T \sHat}^{-1} \bracket{
    S_k S_k^T \sHat} \bracket{S_kS_k^T \sHat}^T
    + \normTwo{S_k^T \sHat} S_k S_k^T}.
    \label{hessMkExpr}
\end{equation}
Therefore for any $y \in \R^l$, we have
\begin{equation}
    y^T \grad^2 \mKHat{\sHat} y 
    = y^T S_k \hessFK S_k^T y
    + \frac{1}{\alphaK} \squareBracket{
    \normTwo{S_k^T \sHat}^{-1} 
    \squareBracket{\bracket{
    S_kS_k^T \sKHat}^T y}^2 +
    \normTwo{S_k^T \sHat} \bracket{
    S_k^T y}^2}. \label{hessMkExprWithY}
\end{equation}
The following Lemma says that if the subspace Hessian has
negative curvature, then the step size is bounded below 
by the size of the negative curvature (and also depends on $\alphaK$).

\begin{lemma} \label{lem:SubHessNegImpStepLower}
If $\lambdaMin{S_k \hessFK S_k^T } < -\epH$; 
and $\normTwo{S_k} \leq \sMax$, 
then 
\begin{equation*}
    \normTwo{S_k^T \sKHat} \geq \epH \squareBracket{\frac{2\sMax^2}{\alphaK} + \kappaS}^{-1}.
\end{equation*}
\end{lemma}

\begin{proof}
Let $y \in \R^l$. Using \eqref{hessMKGeq0} and
\eqref{hessMkExprWithY} we have that
\begin{equation}
    y^T S_k \hessFK S_k^T y \geq
    - \frac{1}{\alphaK} \squareBracket{
    \NsKTSKHat^{-1} 
    \squareBracket{\bracket{
    S_kS_k^T \sKHat}^T y}^2 +
    \NsKTSKHat \bracket{
    S_k^T y}^2} - \kappaS\normTwo{S_k^T \sKHat}.  \notag
\end{equation}

Given $\normTwo{S_k} \leq \sMax$, 
we have that $S_k^T y \leq \sMax \normTwo{y}$. 
So we have
\begin{equation}
    y^T S_k \hessFK S_k^T y \geq
    - \frac{1}{\alphaK} \squareBracket{
    \NsKTSKHat^{-1} 
    \squareBracket{\bracket{
    S_kS_k^T \sKHat}^T y}^2 +
    \NsKTSKHat \sMax^2 \normTwo{y}^2} - \kappaS\normTwo{S_k^T \sKHat}. \notag
\end{equation} 

Minimizing over $\normTwo{y}=1$, noting that
$\max_{\normTwo{y}=1} \bracket{\bracket{
S_k S_k^T \sKHat}^T y}^2 = 
\normTwo{S_kS_k^T \sKHat}^2$, we have
\begin{align}
    -\epH > \lambdaMin{S_k \hessFK S_k^T }
    & \geq -\frac{1}{\alphaK}\squareBracket{
    \NsKTSKHat^{-1} \normTwo{
    S_k S_k^T \sKHat}^2 + \NsKTSKHat
    \sMax^2} - \kappaS\normTwo{S_k^T \sKHat} \notag \\
    & \geq -\frac{1}{\alphaK}\squareBracket{
    \NsKTSKHat^{-1} \sMax^2 \NsKTSKHat^2
    + \NsKTSKHat \sMax^2} - \kappaS\normTwo{S_k^T \sKHat} \notag\\
    &= -\frac{2\sMax^2}{\alphaK}\NsKTSKHat - \kappaS\normTwo{S_k^T \sKHat}. \notag
\end{align}
Rearranging the above gives the result.
\end{proof}

\begin{lemma}
\autoref{alg:R-ARC} satisfies \autoref{AA4} with
\begin{equation}
    h(\epsilon_H, \alphaK) = \frac{\theta\epH^3}{3\alphaK}\squareBracket{\frac{2\sMax^2}{\alphaK} + \kappaS}^{-3}.\label{h:cubic:secondOrder:subspace}
\end{equation}

\end{lemma}

\begin{proof}
Using \autoref{succStepDecrease}, on successful iterations, we have 
$\fK - \fKPlusOne \geq \frac{\theta}{3\alphaK} \normTwo{S_k^T \sKHat}^3$.
Consequently, $k \leq \nEps$ (note the definition \eqref{tmp-2021-12-31-6} of $\nEps$ in this section)
and \autoref{lem:SubHessNegImpStepLower} give the lower bound $h$ \reply{that} holds
in true, successful and $k \leq \nEps$ iterations.
\end{proof}

\paragraph{Satisfying \autoref{AA5}}

\begin{lemma}
\autoref{alg:R-ARC} satisfies \autoref{AA5}.
\end{lemma}

The proof of this lemma is identical to \autoref{tmp-2022-1-15-1}.

\paragraph{Convergence result}

Applying \autoref{thm2}, we have a convergence result for \autoref{alg:R-ARC} to a point where the subspace Hessian has approximately non-negative curvature. While the statement is for scaled Gaussian matrices, it is possible to obtain in a similar way, a similar result for other sketching matrices.

\begin{theorem}
    Let $\epH > 0$ and $l \in \N^+ $.
    Let $\cal{S}$ be the distribution of scaled Gaussian matrices $S \in \R^{l \times d}$. 
    Suppose that $f$ is bounded below by $f^*$, twice continuously differentiable with Lipschitz continuous Hessian $\grad^2 f$ (with Lipschitz constant $\LH$).
    Choose $\deltaS = \frac{1}{16}< \frac{c}{(c+1)^2}$ so that
    $
    \sMax = 1 + 
    \frac{ \sqrt{d} + \sqrt{2\log16} }
    { \sqrt{l} }$. 
    Let $h\sim \epsilon_H^{3}$ be defined in \eqref{h:cubic:secondOrder:subspace}, $\alphaLow$ be given in \eqref{eq:alphaLow:CBGN:secondOrder} and
     $\alphaMin = \alphaZero \gammaOne^\newL$ associated with $\alphaLow$, for some $\newL \in \N^+$. Apply \autoref{alg:R-ARC} to minimizing $f$ for $N$ iterations.
	Then for any $\delta_1 \in (0,1)$ with 
    \begin{equation}
        \gDeltaSDeltaOne >0, \nonumber
    \end{equation}
    where 
    \begin{equation}
        \gDeltaSDeltaOne = \nPreFactorTR; \nonumber
    \end{equation}
    if $N\in \N$ satisfies 
    \begin{equation}
        N \geq \gDeltaSDeltaOne \squareBracket{
             \frac{f(x_0) - f^*}{h(\epsilon_H, \alphaZero\gammaOne^{c+\newL}
)}
             + \frac{\newL}{1+c}}, \nonumber
    \end{equation}
    we have that
    \begin{equation}
        \probability{
            \min \{
                k: \lambdaMin{S_k^T \hessFK S_k}  \geq -\epH 
            \} \leq N 
            } \geq 1 - \chernoffLowerExponential. \nonumber
    \end{equation}
\end{theorem}

The above result shows that the global convergence rate to a (subspace) second-order critical point
is $\mathcal{O}(\epH^{-3})$, matching full-dimensional corresponding results for cubic regularization in the order of the accuracy $\epH$
and sharing the same problem assumptions.

\subsection{R-ARC convergence to second-order (full space) critical points}

In this section, we show that, if $\cal{S}$ is the distribution of scaled Gaussian matrices, \autoref{alg:R-ARC} will converge to a (full-space) second-order critical point, with a rate matching the standard full-dimensional cubic regularization algorithm.

We define 
\begin{equation}
    \nEps = \nThreeEpH 
= \min \set{k: \lambdaMin{ \hessFK } \geq -\epH } \label{eq:nEps:secondOrderFull}
\end{equation}

The following definition of true iterations assumes
that $\hessFK$ has rank $r\leq d$ (which was not needed for Section 6.1).

\begin{definition}\label{true:cubic:full:secondOrder}
Let $\sMax> 0, \epsOne \in (0,1)$. 
An iteration $k$ is $(\epsOne, \sMax)$-true if the following two conditions hold
\begin{enumerate}
    \item $\normTwo{S_k} \leq \sMax$. 
    \item There exists an eigen-decomposition of $\hessFK 
    = \sum_{i=1}^r \lambda_i u_i u_i^T$ with $\lambda_1 
    \geq \lambda_2 \geq \dots \geq \lambda_r$ and $r\leq d$, such that
    with $w_i = S_k u_i$,
    \begin{align}
        & 1-\epsOne \leq \normTwo{w_r}^2 \leq 1+ \epsOne, \label{gauss:norm:one}\\
        & (w_i^T w_r)^2 \leq 16l^{-1} (1+\epsOne)
        \texteq{ for all $i \neq r$}. \label{gauss:approx:orthogonal}
    \end{align}
\end{enumerate}
\end{definition}
Note that since $S_k \in \R^{l \times d}$ is a (scaled) Gaussian matrix, 
and the set $\set{u_i}$, $i\in \{1,\ldots,r\}$, is orthonormal, 
we have that $\set{w_i}$, $i\in \{1,\ldots,r\}$, are independent Gaussian vectors, 
with entries drawn from $N(0, l^{-1} )$. Equation \eqref{gauss:approx:orthogonal} simply
requires that these high-dimensional Gaussian vectors are approximately 
orthogonal, which is known to hold with high probability \cite{MR3837109}.
We now show that the four assumptions needed for \autoref{thm2} hold, and then apply \autoref{thm2} for this particular
definition of $\nEps$. 

\paragraph{Satisfying \autoref{AA2}}
As before, similarly to Lemma \ref{lem:arbitatratyDeltaS}, we show that each of the two conditions in \autoref{true:cubic:full:secondOrder}
holds with high probability, and then use a union bound argument to show \autoref{AA2} is true. Note that the conditional independence between
iterations is clear here because given $x_k$, whether the iteration is true
or not only depends on the random matrix $S_k$ and is independent of all the previous iterations.

For the first condition in \autoref{true:cubic:full:secondOrder}, \autoref{lem:GaussSMax:CubicSubspace} gives that \autoref{alg:R-ARC} satisfies \autoref{AA2} 
with any $\deltaSTwo \in (0,1)$ and 
\begin{equation}
    \sMax = 1 + \sqrt{\frac{d}{l}} + \sqrt{\frac{2\logOneOverDeltaSTwo}{l}}. \label{tmp-2022-1-14-1}
\end{equation}

\autoref{lem:WrWr} shows that \eqref{gauss:norm:one} holds with high probability.
\begin{lemma} \label{lem:WrWr}
Let $w_i \in \R^l$ with $w_{ij}$ be independently following $N(0, l^{-1})$;
let $\epsOne \in (0,1)$. 
Then, for some problem-independent constant $C$, we have
\begin{equation}
    \probability{ \abs{\normTwo{w_i}^2-1 } \leq \epsOne }
    \geq 1 - 2\pow{e}{-\frac{l\epsOne^2}{C}}. 
    \label{w_r}
\end{equation}

\end{lemma}

\begin{proof}
The proof is standard, see \cite{MR1943859} or the proof of Lemma 3.3 in \cite{cartis_randomised_2022}.
We note that
$C \approx 4$.
\end{proof}

Next, we show that conditioning on \eqref{gauss:norm:one} being true, \eqref{gauss:approx:orthogonal} holds with high probability. We first study the case of a single fixed $i$ (instead of all $i$).

\begin{lemma} \label{lem:WiWr}
Let $\epsOne \in (0,1)$ and suppose $w_r$ satisfies \eqref{w_r} with $i=r$.
Then, with (conditional) probability at least 0.9999, 
independent of $w_r$,
we have $(w_i^T w_r)^2 \leq 16l^{-1} (1+\epsOne)$.
\end{lemma}

\begin{proof}
We have $(w_i^T w_r)^2 = 
\normTwo{w_r}^2 \bracket{
w_i^T \frac{w_r}{\normTwo{w_r}} }^2
$. The term inside the 
bracket is an $N(0, l^{-1})$ random variable independent 
of $w_r$, because sum of independent normal random variables
is still normal. Note that for a normal random variable $N(0, \sigma^2)$, 
with probability at least $0.9999$, its absolute value lies within
$\pm 4 \sigma$. Therefore we have that
with probability at least $0.9999$, 
\begin{equation}
    \bracket{
w_i^T \frac{w_r}{\normTwo{w_r}} }^2 \leq 
16 l^{-1}. 
\end{equation}

Combining with \eqref{w_r} gives the result. 
\end{proof}
\autoref{lem:wIwRallIneqR} shows that conditioning on \eqref{gauss:norm:one} being true, \eqref{gauss:approx:orthogonal} is true with high probability.
\begin{corollary}\label{lem:wIwRallIneqR}
Let $\epsOne \in (0,1)$ and suppose $w_r$ satisfies \eqref{w_r} with $i=r$.
Then, with (conditional) probability at least $0.9999^{(r-1)}$, 
we have that 
$(w_i^T w_r)^2 \leq (1+\epsOne)16l^{-1}$ for all $i \neq r$.
\end{corollary}

\begin{proof}
Note that conditioning on $\normTwo{w_r}^2$,  
$w_i^T w_r$ are independent events. 
Therefore we simply multiply the probability. 
\end{proof}

The following lemma shows that the second condition in \autoref{true:cubic:full:secondOrder} is true with high probability.

\begin{lemma}\label{lem:wrAndAllwi}
Let $\epsOne > 0$.
Let $A_1 = \set{\abs{\normTwo{w_r}^2-1}\leq \epsOne}$, 
and $A_2 = \set{\bracket{w_i^T w_r}^2\leq 
16l^{-1}(1+\epsOne)}$, 
for all $i\neq r$. 
Then with probability at least $(0.9999)^{r-1}
\bracket{1-2e^{-\frac{l\epsOne^2}{C}}}$, 
we have that $A_1$ and $A_2$ hold simultaneously.
\end{lemma}

\begin{proof}
We have 
$\probability{A_1 \intersect A_2} 
= \conditionalP{A_2}{A_1} \probability{A_1}$.
Using \autoref{lem:WrWr}
and \autoref{lem:wIwRallIneqR} 
gives the result.  
\end{proof}

Therefore, using \eqref{tmp-2022-1-14-1}, \autoref{lem:wrAndAllwi} and the union bound we have the following

\begin{lemma}\label{tmp-2022-1-14-5}
Let $\epsOne > 0$, $l \in \N^+$, $\deltaSTwo > 0$ such that 
\begin{equation}
    \deltaS 
= (0.9999)^{r-1} \bracket{1 - 2e^{-\frac{l \epsOne^2}{C} } } + \deltaSTwo
< 1. \label{tmp-2022-1-14-4}
\end{equation}
Then \autoref{alg:R-ARC} with $(\sMax, \epsOne)$-true iterations
defined in \autoref{true:cubic:full:secondOrder} satisfies \autoref{AA2} where $\sMax$ is defined in \eqref{tmp-2022-1-14-1}.
\end{lemma}

\paragraph{Satisfying \autoref{AA3}}

\begin{lemma}\label{tmp-2022-1-14-6}
Let $f$ be twice continuously differentiable with $\LH$-Lipschitz continuous Hessian. Then
\autoref{alg:R-ARC} with true iterations defined in 
\autoref{true:cubic:full:secondOrder} satisfies \autoref{AA3} with
\begin{equation}
    \alphaLow = \frac{2(1-\theta)}{\LH} \label{eq:alphaLow:CBGN:secondOrder:1}
\end{equation}

\end{lemma}

The proof is identical to \autoref{lem:cubic:sub:A2}\footnote{Note that we need to show that
$\normTwo{S_k^T \sKHat} > 0$ in true iterations before convergence, which is shown in \eqref{tmptmp}.}.

\paragraph{Satisfying \autoref{AA4}}

\autoref{lem:lambdaMinSketchedHess} is a key ingredient. It shows that in true iterations, the subspace Hessian's negative curvature ($\lambdaMin{S_k \hessFK S_k^T}$) is proportional to the full Hessian's negative curvature ($\lambdaMin{\hessFK}$). 
\begin{lemma} \label{lem:lambdaMinSketchedHess}
Suppose iteration $k$ is true with $\epSOne \in (0,1)$ and $k < \nEps$. 
Let $\kappa_H = \min\set{0, \lambda_1/\lambda_r}$. Suppose 
\begin{equation}
    1 - \epsOne + 16\frac{r-1}{l}\frac{1+\epsOne}{1-\epsOne}\frac{\lambda_1}{\lambda_r} \geq 0. \label{tmp-2022-1-14-3}
\end{equation}
Then 
we have that
\begin{equation}
    \lambdaMin{S_k \hessFK S_k^T} \leq
    -\epH m(\epsOne, r, l, \kappa_H), \notag
\end{equation}

where
\begin{equation}
    m(\epsOne, r, l, \kappa_H) = 
    \bracket{1-\epsOne
    + 16\frac{r-1}{l}\frac{1+\epsOne}{1-\epsOne}\kappa_H
    }. \label{eq:mDef}
\end{equation}

\end{lemma}

\begin{proof}
Using the eigen-decomposition of $\hessFK$, we have that
$S_k \hessFK S_k^T = \sum_{i=1}^r \lambda_i w_i w_i^T$. We now 
use the Rayleigh quotient expression of minimal 
eigenvalue (with $w_r$ being the trial vector):
\begin{equation}
    \lambdaMin{S_k \hessFK S_k^T} \leq  
\frac{\sum_{i=1}^r \lambda_i \bracket{w_i^T w_r}^2}{
w_r^T w_r} \notag
\end{equation}
Furthermore,
\begin{align}
    & \frac{\sum_{i=1}^r \lambda_i \bracket{w_i^T w_r}^2}{
    w_r^T w_r} 
     = \bracket{w_r^T w_r }\lambda_r +
    \frac{\sum_{i=1}^{r-1} \lambda_i \bracket{w_i^T w_r}^2}{
    w_r^T w_r} \notag\\
    & \leq \bracket{1-\epsOne}\lambda_r + 
    \lambda_1 \frac{\sum_{i=1}^{r-1} \bracket{w_i^T w_r}^2}{
    w_r^T w_r} 
     \leq \bracket{1-\epsOne}\lambda_r + 
    16\frac{r-1}{l}\frac{1+\epsOne}{1-\epsOne}\lambda_1
    ,\label{tmp-2022-1-14-2}
\end{align}
where the two inequalities follow from \eqref{gauss:norm:one}
    and \eqref{gauss:approx:orthogonal} because iteration $k$ is true.
Next we discuss two cases. 

\begin{enumerate}
    \item If $\lambda_1 < 0$, then $\kappa_H = 0$ because $\lambda_r < -\epH < 0$. Thus, $m(\epsOne, r, l, \kappa_H) = 1-\epsOne$. The desired result follows from \eqref{tmp-2022-1-14-2} by noting that the second term $16\frac{r-1}{l}\frac{1+\epsOne}{1-\epsOne}\lambda_1 < 0$ and $\lambda_r < -\epH$.
    \item If $\lambda_1 \geq 0$, then $\kappa_H = \frac{\lambda_1}{\lambda_r}$ and from \eqref{tmp-2022-1-14-2}, we have
        \begin{align*}
            \bracket{1-\epsOne}\lambda_r + 
            16\frac{r-1}{l}\frac{1+\epsOne}{1-\epsOne}\lambda_1
            & = \lambda_r \bracket{1 - \epsOne + 16\frac{r-1}{l}\frac{1+\epsOne}{1-\epsOne}\frac{\lambda_1}{\lambda_r}} \\
            & \leq -\epH  m(\epsOne, r, l, \kappa_H),
        \end{align*}
\end{enumerate}
where we used $\eqref{tmp-2022-1-14-3}$ and $\lambda_r < -\epH$ to derive the inequality. 
And the desired result follows. 
\end{proof}

\begin{remark}
Note that \eqref{tmp-2022-1-14-3} always holds if $\lambda_1 \leq 0$ (where recall that $\lambda_i$ are eigenvalues of $\hessFK$ and $k < \nEps$ implies $\lambda_r<-\epH < 0$.). If  $\lambda_1 >0$, then \eqref{tmp-2022-1-14-3} holds if we have $\kappa\bracket{\hessFK} \frac{r-1}{l} \leq \frac{(1-\epsOne)^2}{16(1+\epsOne)}$ where $\kappa\bracket{\hessFK} = \abs{\frac{\lambda_1}{\lambda_r}}$ is the condition number of $\hessFK$.
\end{remark}

We conclude that \autoref{AA4} is satisfied. 
\begin{lemma} \label{tmp-2022-1-14-7}
\autoref{alg:R-ARC} with $\cal{S}$ being the distribution of scaled Gaussian
matrices, true iterations defined in \autoref{true:cubic:full:secondOrder}
and $\nEps$ defined in \eqref{eq:nEps:secondOrderFull} satisfies
\autoref{AA4} with
\begin{equation}
    h(\epH, \alphaK) = 
    \frac{\theta \epH^3 m(\epsOne, r, l, \kappa_H)^3}{3\alphaK} 
\squareBracket{\frac{2\sMax^2}{\alphaK} + \kappaS}^{-3}. 
     \label{h:cubic:secondOrder:fullspace}
\end{equation}
\end{lemma}

\begin{proof}
Let iteration $k$ be true and successful with $k < \nEps$. 
\autoref{lem:lambdaMinSketchedHess} gives that 
\begin{equation}
    \lambdaMin{S_k \hessFK S_k^T} \leq
    -\epH m(\epsOne, r, l, \kappa_H). \notag
\end{equation}
Then we have 
\begin{equation}
    \normTwo{S_k^T \sKHat}
    \geq 
    \epH m(\epsOne, r, l, \kappa_H)
    \squareBracket{\frac{2\sMax^2}{\alphaK} + \kappaS}^{-1}, \label{tmptmp}
\end{equation}
by applying \autoref{lem:SubHessNegImpStepLower} with
$\epH := \epH m(\epsOne, r, l, \kappa_H)$.
The desired result follows by applying \autoref{succStepDecrease}.
\end{proof}

\paragraph{Satisfying \autoref{AA5}}
The proof is identical to the one in last section because \autoref{AA5} is not affected by the change of definitions of $\nEps$ and true iterations.

\paragraph{Convergence of \autoref{alg:R-ARC} to a second-order (full-space) critical point}
Applying \autoref{thm2}, the next theorem shows that using \autoref{alg:R-ARC} with scaled Gaussian matrices
achieves convergence to a second-order critical point, with a worst-case rate matching the classical full-dimensional method. 

  Let $\epH > 0$ and $l \in \N^+ $.
    Let $\cal{S}$ be the distribution of scaled Gaussian matrices $S \in \R^{l \times d}$. 
    Suppose that $f$ is bounded below by $f^*$, twice continuously differentiable with Lipschitz continuous Hessian $\grad^2 f$ (with Lipschitz constant $\LH$).
    Choose $\deltaS = \frac{1}{16}< \frac{c}{(c+1)^2}$ so that
    $
    \sMax = 1 + 
    \frac{ \sqrt{d} + \sqrt{2\log16} }
    { \sqrt{l} }$. 
    Let $h\sim \epsilon_H^{3}$ be defined in \eqref{h:cubic:secondOrder:subspace}, $\alphaLow$ be given in \eqref{eq:alphaLow:CBGN:secondOrder} and
     $\alphaMin = \alphaZero \gammaOne^\newL$ associated with $\alphaLow$, for some $\newL \in \N^+$. Apply \autoref{alg:R-ARC} to minimizing $f$ for $N$ iterations.
	Then for any $\delta_1 \in (0,1)$ with 

\begin{theorem}\label{thm:second_order_fullspace}
In \autoref{alg:R-ARC}, let $\cal{S}$ be the distribution of scaled Gaussian matrices $S \in \R^{l \times d}$ and  let $\epH > 0$.  
Suppose $f$ is  bounded below by $f^*$, twice continuously differentiable with Lipschitz continuous Hessian $\grad^2 f$ (with Lipschitz constant $\LH$) and $\hessFK$ has rank at most $r$ for all $k$.    Choose $
    \deltaSTwo = \frac{1}{16};
    $ so that
    $
    \sMax = 1 + 
    \frac{ \sqrt{d} + \sqrt{2\log16} }
    { \sqrt{l} }$. 
    Let $\deltaS$ be defined in \eqref{tmp-2022-1-14-4} and assume that
   $ \delta_S <\frac{c}{(c+1)^2}$.
    Let $h\sim \epH^3$ be defined in \eqref{h:cubic:secondOrder:fullspace}, $\alphaLow$ be given in \eqref{eq:alphaLow:CBGN:secondOrder:1} and
     $\alphaMin = \alphaZero \gammaOne^\newL$ associated with $\alphaLow$,  for some $\newL \in \N^+$. Apply \autoref{alg:R-ARC} to minimizing $f$ for $N$ iterations.
	Then for any $\delta_1 \in (0,1)$ with 
    \begin{equation}
        \gDeltaSDeltaOne >0, \nonumber
    \end{equation}
    where 
    \begin{equation}
        \gDeltaSDeltaOne = \nPreFactorTR \nonumber;
    \end{equation}
    if $N$ satisfies 
    \begin{equation}
        N \geq \gDeltaSDeltaOne \squareBracket{
             \frac{f(x_0) - f^*}{h(\epsilon_H, \alphaZero\gammaOne^{c+\newL}
)}
             + \frac{\newL}{1+c}}, \nonumber
    \end{equation}
    we have that
    \begin{equation}
        \probability{
            \min \{
                k: \lambdaMin{ \hessFK }  \geq -\epH 
            \} \leq N 
            } \geq 1 - \chernoffLowerExponential. \nonumber
    \end{equation}
\end{theorem}
\begin{proof}
Applying \autoref{tmp-2022-1-14-5}, \autoref{tmp-2022-1-14-6}, \autoref{tmp-2022-1-14-7},  and recalling that $\nEps$ is defined in \eqref{eq:nEps:secondOrderFull} and
true iterations  in \autoref{true:cubic:full:secondOrder}, we have that the four assumptions in \autoref{thm2} are satisfied. Then applying \autoref{thm2} gives the desired result.
\end{proof}

\paragraph{Application to low-rank functions} Similarly to \autoref{cor:low_rank_first_convergence} for first order convergence, we can use the Hessian rank bound for low-rank functions to apply \autoref{thm:second_order_fullspace} to low-rank functions. If $f$ is a low-rank function of rank $r\leq d$, then the Hessian at $x_k$ has rank at most $r$ and so  \autoref{thm:second_order_fullspace} applies.

\section{Numerical Experiments}

In this section, we test the performance of R-ARC and R-ARC-D compared to ARC. We use the CUTEst collection of test problems \cite{gould2015cutest}. As well as problem-by-problem experiments, we measure algorithm performance using data profiles \cite{moreBenchmarkingDerivativeFreeOptimization2009}, which themselves are a variant of performance profiles \cite{dolan2002benchmarking}. We use \textit{relative Hessians seen}\footnote{The relative Hessians seen by an iteration $k$ is $(l_{k}/d)^{2}$ where $l_{k}$ is the sketch dimension and $d$ is the problem dimension.} as well as runtime for our data profiles, maintaining the notation in \cite{cartis_randomised_2022}, for a given solver $s$, test problem $p \in \mathcal{P}$ and tolerance $\tau \in (0, 1)$, we determine the number of relative Hessians seen $N_{p}(s, \tau)$ required for a problem to be solved:
\begin{equation*}
    N_{p}(s, \tau) :=\ \text{\# of relative Hessians seen until}\ f(x_{k}) \leq f(x^{*}) + \tau (f(x_{0}) - f(x^{*})).
\end{equation*}
We set $N_{p}(s, \tau) = \infty$ in the cases where the tolerance is not reached within the maximum number of iterations, which we take to be 2000.

To produce the data profiles, we plot
\begin{equation*}
    \pi_{s, \tau}^{N}(\alpha):= \frac{|\{p \in \mathcal{P}\ :\ N_{p}(s, \tau) \leq \alpha\}|}{|\mathcal{P}|}\ \text{for}\ \alpha \in [0, 100],
\end{equation*}
namely, the fraction of problems solved after $\alpha$ relative Hessians seen. For runtime data profiles, we replace relative Hessians seen with the runtime in the above definitions.

We primarily test R-ARC and R-ARC-D on a set of 28 problems from the CUTEst collection \cite{Lindon22, cartis_randomised_2022}, each of dimension approximately $d = 1000$, we refer to these as \textit{full-rank problems}; the problem information can be found in \autoref{tab:cutest_fullrank}. Additionally, we testR-ARC and R-ARC-D on a set of \textit{low-rank problems}; these are CUTEst problems of dimension approximately $d = 100$, that have had their dimension increased to $1000$ through adding artificial variables. Additional details can be found in the Appendix.  

The ARC code used is from adapted from the implementation in \cite{cartisEfficientImplementationThirdorder2024}\footnote{Where ARC is referred to as AR2 in the AR\textsubscript{p} framework}, which itself is an updated version of the MATLAB implementation of ARC in \cite{Cartis:2009fq}, using the $\sigma_{k}$ update rule in \cite{gouldUpdatingRegularizationParameter2012}; we used  default settings of this code in our experiments. We then implemented  R-ARC and R-ARC-D by modifying the ARC code as follows:

\begin{itemize}
    \item Regularize the subproblem with $\|\hat{s}\|_{2}^{3}$ rather than $\|S_{k}^{\top}\hat{s}\|_{2}^{3}$ (this empirically improves behaviour in the subproblem solver)
    \item Terminate using conditions on $\|\SK\nabla f(x_{k})\|$ rather than $\|\nabla f(x_{k})\|$ (this implies explicit access to $\nabla f(x_{k})$ is not required)
    \item When applying the adaptive sketch size update rule, set $C, D = 1$, that is $l_{k} = \hat{R}_{k} + 1$
    \item Only draw $S_{k + 1}$ randomly if the step $k$ is successful, else set $S_{k+1} = S_{k}$ (this means $S_{k+1}\nabla^{2}f(x_{k+1})S_{k+1}^{\top} = \SK \hessFK \SK^{\top}$, avoiding re-calculation and saving computational time)
    \end{itemize}

For both full-rank and low-rank problems, we compare the R-ARC and R-ARC-D algorithms for different sketching dimensions with the original ARC algorithm. We also experiment by using different choices of sketching matrices.
In all cases, we terminate when $\|\SK\nabla f(x_{k})\| < 10^{-5}$. If not otherwise specified, we use scaled Gaussian sketching matrices.

\subsection{Full-rank problems}

Here we consider the problems  in \autoref{tab:cutest_fullrank}. We conduct experiments varying the sketch size and sketching matrix distribution. For sketch size, we consider $l = 10\%, 20\%, 50\%$ of the original problem dimension (so for $d=1000$, we consider $l = 100, 200, 500$). For the sketching matrix distribution, we consider \textit{scaled Gaussian, scaled sampling} and \textit{scaled Haar} matrices.

\paragraph{R-ARC} We first test the R-ARC algorithm given by \autoref{alg:R-ARC}. We initially consider Gaussian sketching matrices only across a range of sketch dimensions. We plot two regimes: the initial behaviour for low precision (1e-2) and an extended view for a high(er) precision (1e-5). The plots are given in \autoref{fig:gaussian_rarc}.

\begin{figure}
    \centering
    \includegraphics[width=0.4\linewidth]{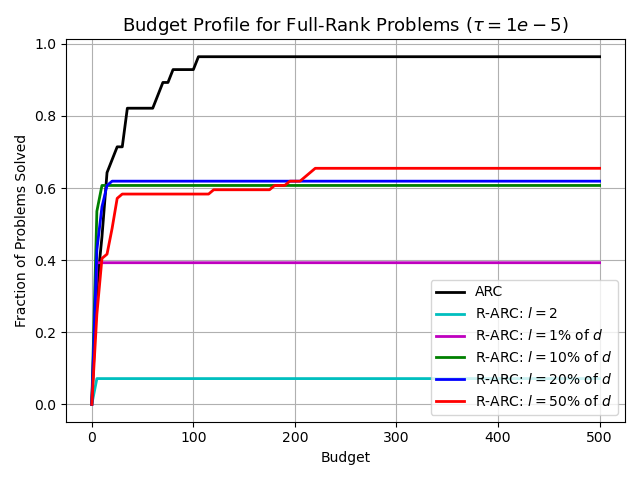}
    \includegraphics[width=0.4\linewidth]{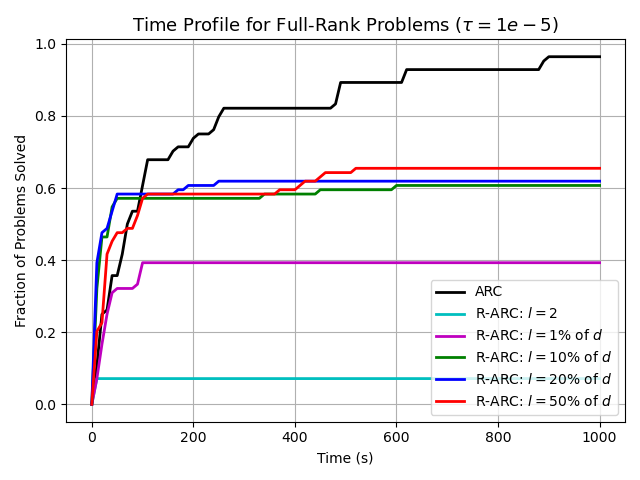}
    \includegraphics[width=0.4\linewidth]{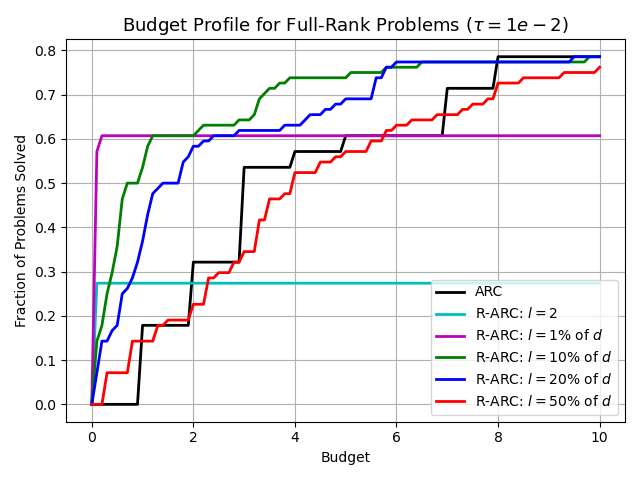}
    \includegraphics[width=0.4\linewidth]{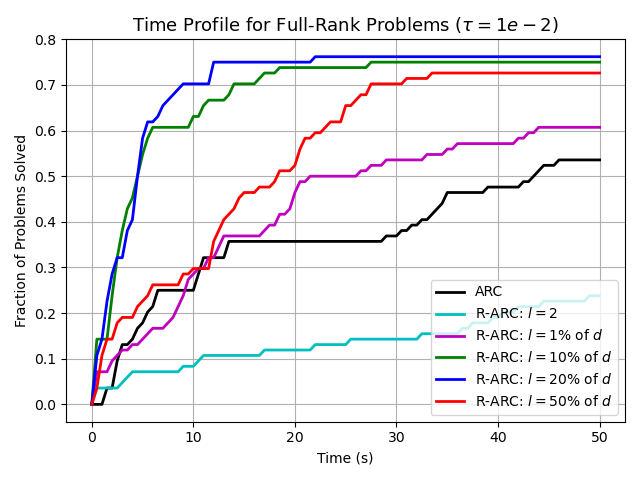}
    \caption{Varying the sketch dimension using scaled Gaussian matrices, comparing R-ARC vs ARC; plotting budget and time; full-rank problems}
    \label{fig:gaussian_rarc}
\end{figure}

We see that for the low-precision solution, initially, small sketching matrices  ($l = 1, 10\%$) perform well when considering budget, although do not outperform when considering runtime. For the higher precision plot, sketching at $10\%$ or more of the problem dimension yields similar results. Sketching to $50\%$ does lead to a slightly greater fraction of problems being solved, but this is still lower than the original ARC algorithm.

\paragraph{R-ARC-D}

Here we experiment using the R-ARC-D algorithm, varying the sketch dimension. As the maximum number of iterations is 2000 and the function dimensions are typically $d \approx 1000$, the sketch dimension $l_k$ usually increases until it reaches the full dimension or the algorithm converges. The plots can be found in \autoref{fig:gaussian_rarcd}.

\begin{figure}
    \centering
    \includegraphics[width=0.4\linewidth]{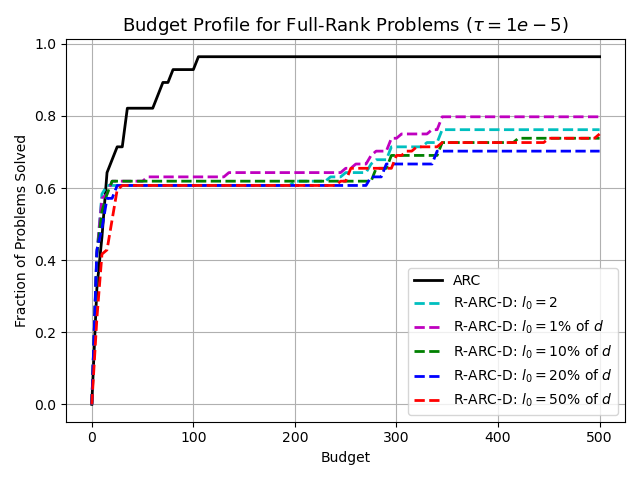}
    \includegraphics[width=0.4\linewidth]{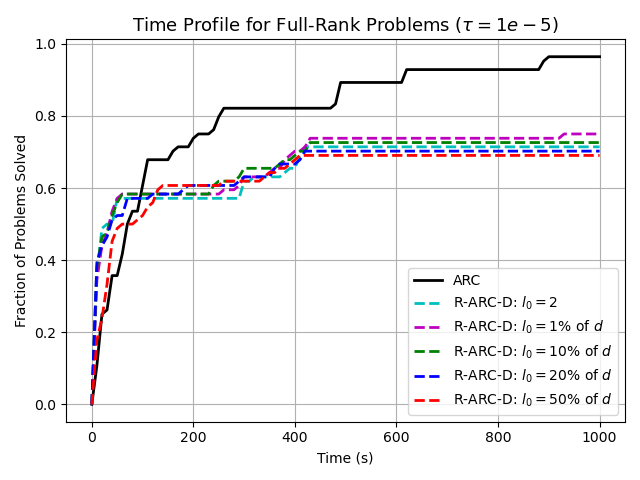 }
    \includegraphics[width=0.4\linewidth]{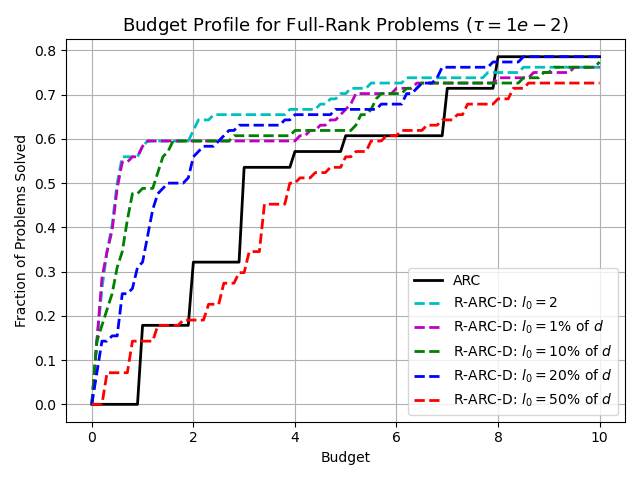}
    \includegraphics[width=0.4\linewidth]{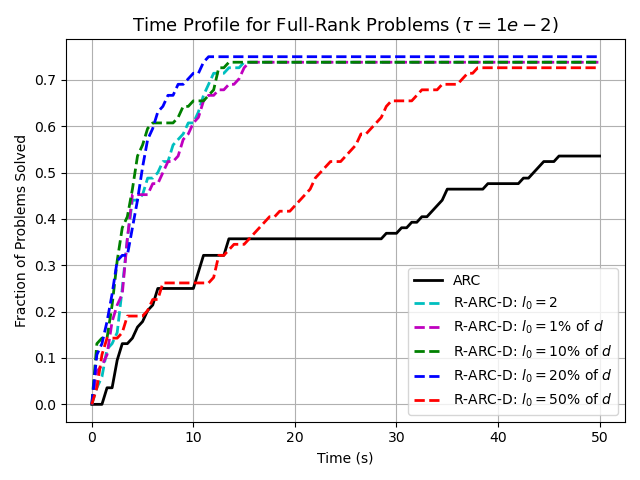}
    \caption{Varying the sketch dimension using scaled Gaussian matrices, comparing R-ARC-D vs ARC; plotting budget and time; full-rank problems}
    \label{fig:gaussian_rarcd}
\end{figure}

When considering budget, we see that in the full plot, small starting sketch dimensions ($l_0 = 2$, $l_0 = 1\%$) typically work well; these small starting sketches also work well when considering the initial behaviour for low-precision solutions. The gap between profiles is smaller when considering time, with smaller initial sketch dimension still appearing to perform better. 

\paragraph{R-ARC-D vs R-ARC}

In these plots, we compare the performance of R-ARC-D with R-ARC (and ARC) for selected initial sketch dimensions. We plot R-ARC with sketch dimension of $l = 20, 50\%$ as when considering runtime, these dimension perform initially and asymptotically. For R-ARC-D, we plot initial sketch dimensions of $l_0 = 1, 20\%$ as $l_0 = 1\%$ performs the best for the high-precision plots, whilst $l_0 = 20\%$ provides a point of comparison to R-ARC with $l = 20$. The data profiles are plotted in \autoref{fig:gaussian_rarc_vs_rarcd}. Additional problem-by-problem plots for both R-ARC and R-ARC-D in \autoref{fig:full_rank_individual_problems}.

\begin{figure}
    \centering
    \includegraphics[width=0.4\linewidth]{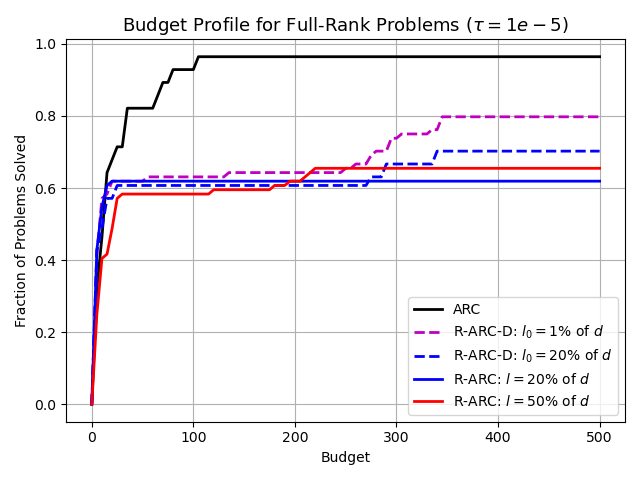}
    \includegraphics[width=0.4\linewidth]{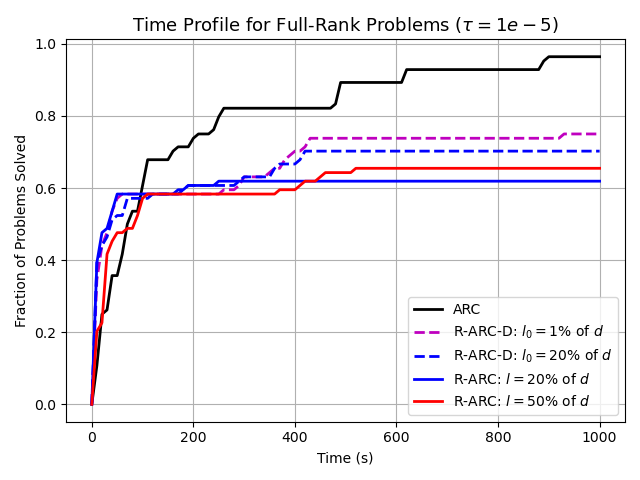 }
    \includegraphics[width=0.4\linewidth]{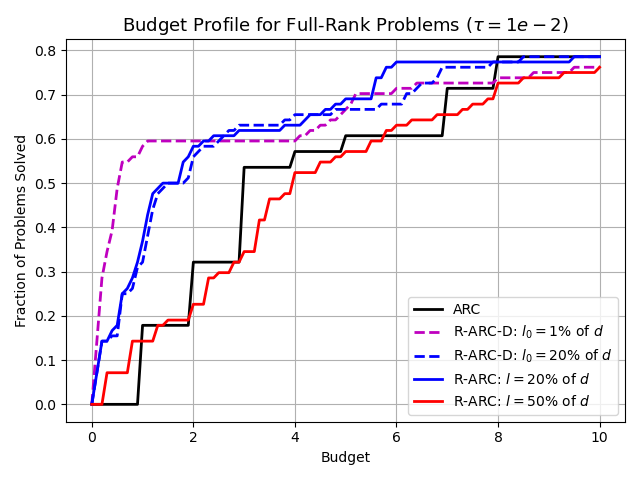}
    \includegraphics[width=0.4\linewidth]{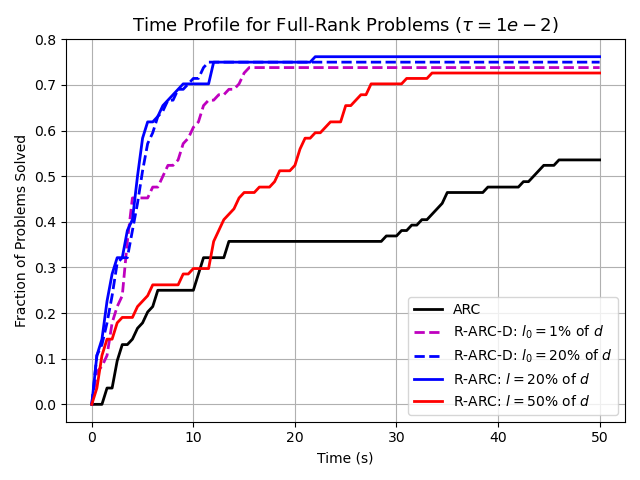}
    \caption{Varying the sketch dimension using scaled Gaussian matrices, comparing R-ARC-D vs R-ARC vs ARC; plotting budget and time; full-rank problems}
    \label{fig:gaussian_rarc_vs_rarcd}
\end{figure}

\paragraph{Other sketching matrices} We have so far considered only Gaussian matrices. We include comparable plots for Haar matrices in Figures \ref{fig:haar_rarc}, \ref{fig:haar_rarcd} and \ref{fig:haar_rarc_vs_rarcd}; plots for sampling matrices can be found in Figures \ref{fig:sampling_rarc}, \ref{fig:sampling_rarcd} and \ref{fig:sampling_rarc_vs_rarcd}. We now include a subset of plots, showing the best performance achieved by each of the sketching matrices; in each case, this was achieved by R-ARC-D. We plot the results in \autoref{fig:comparing_matrices}.

\begin{figure}
    \centering
    \includegraphics[width=0.4\linewidth]{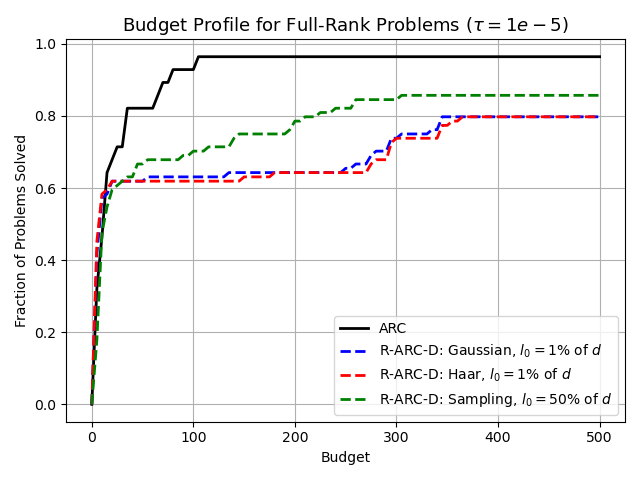}
    \includegraphics[width=0.4\linewidth]{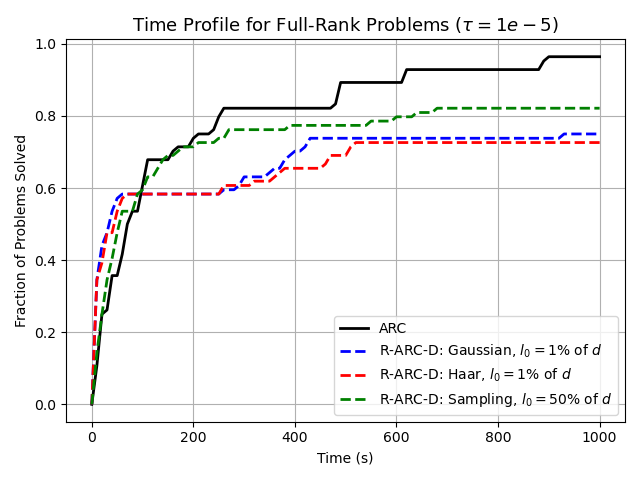 }
    \includegraphics[width=0.4\linewidth]{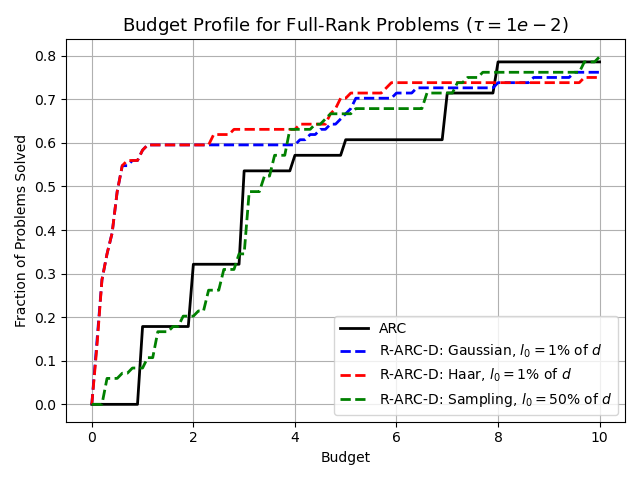}
    \includegraphics[width=0.4\linewidth]{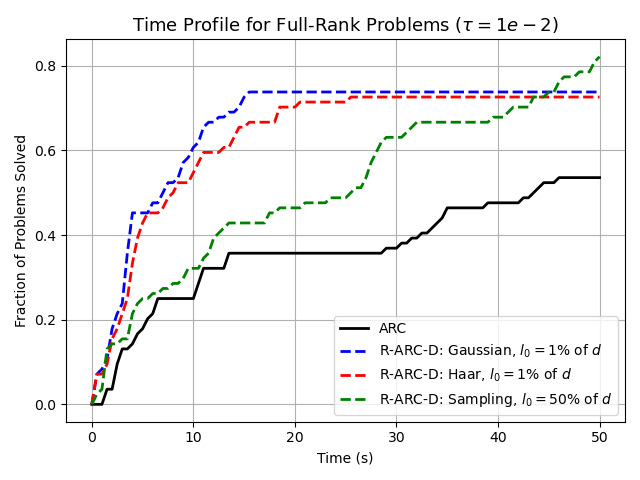}
    \caption{Comparing Gaussian matrices with scaled Haar and scaled sampling matrices in R-ARC-D; plotting budget and time for the best $l_0$ for each matrix type; full-rank problems}
    \label{fig:comparing_matrices}
\end{figure}

We see that Gaussian and Haar matrices perform similarly. This could be expected, as a Haar matrices can be sampled by taking a QR factorization of a Gaussian matrix, which retains the image. Sampling matrices perform better than the other two matrices. This suggests that the coordinate directions provide useful information in CUTEst problems. 

\subsection{Low-rank problems}

\paragraph{R-ARC}

We plot the performance of the R-ARC algorithm on the low-rank problems, comparing the performance to that of ARC. We plot sketch dimensions $l = 2$ and $l=1, 5, 7.5\%$ of the problem dimension $d$. The results are plotted in \autoref{fig:lr_gaussian_rarc}.

\begin{figure}
    \centering
    \includegraphics[width=0.4\linewidth]{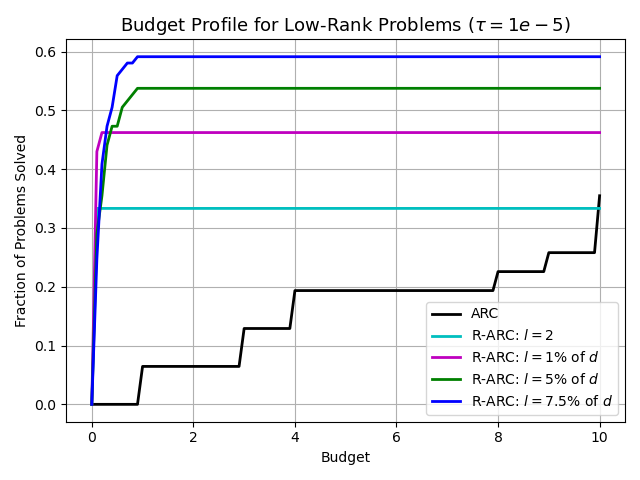}
    \includegraphics[width=0.4\linewidth]{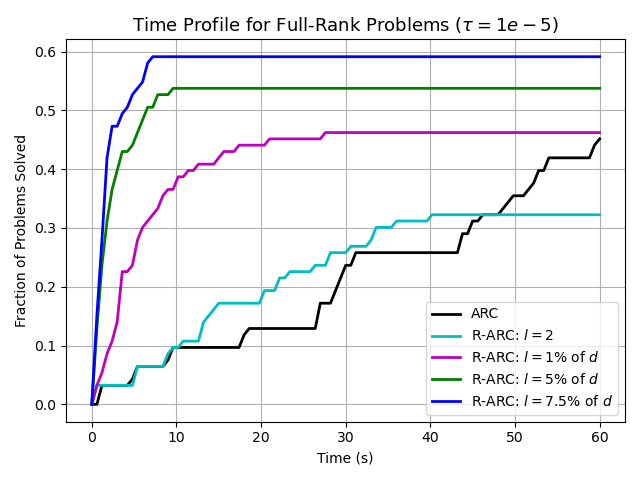}
    \includegraphics[width=0.4\linewidth]{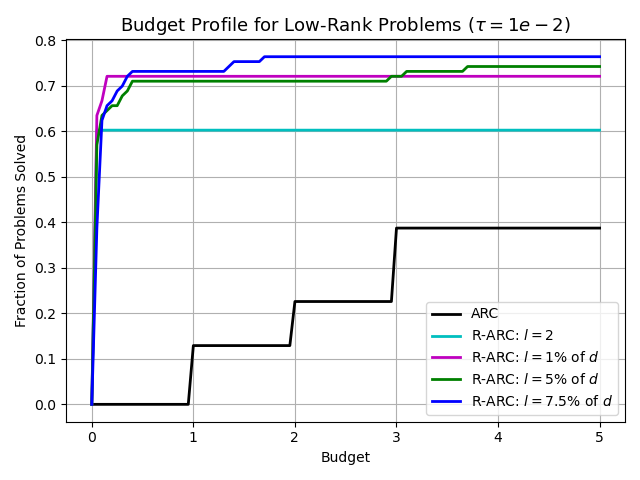}
    \includegraphics[width=0.4\linewidth]{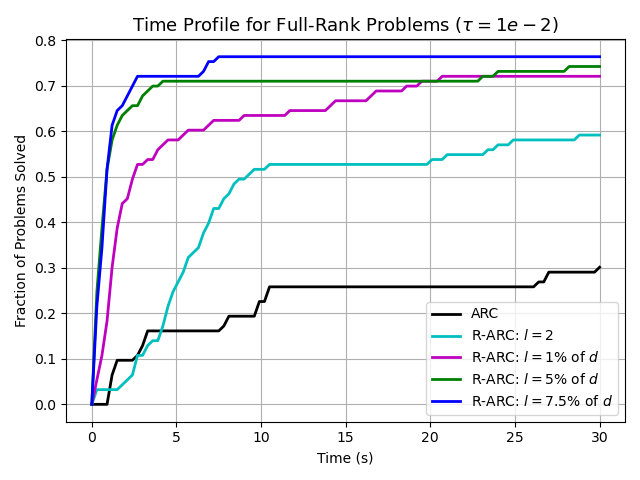}
    \caption{Varying the sketch dimension using scaled Gaussian matrices in R-ARC; plotting budget and time; low-rank problems}
    \label{fig:lr_gaussian_rarc}
\end{figure}

We see that when considering the runtime, R-ARC performs better as the sketch dimension $l$ increases. This is similarly observed when considering the budget profile.

\paragraph{R-ARC-D}

We now plot the performance of R-ARC-D against ARC on low-rank problems. The results are plotted in \autoref{fig:lr_gaussian_rarcd}.

\begin{figure}
    \centering
    \includegraphics[width=0.4\linewidth]{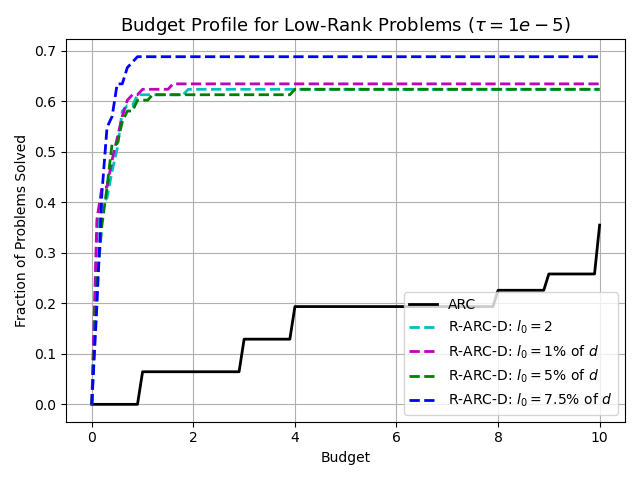}
    \includegraphics[width=0.4\linewidth]{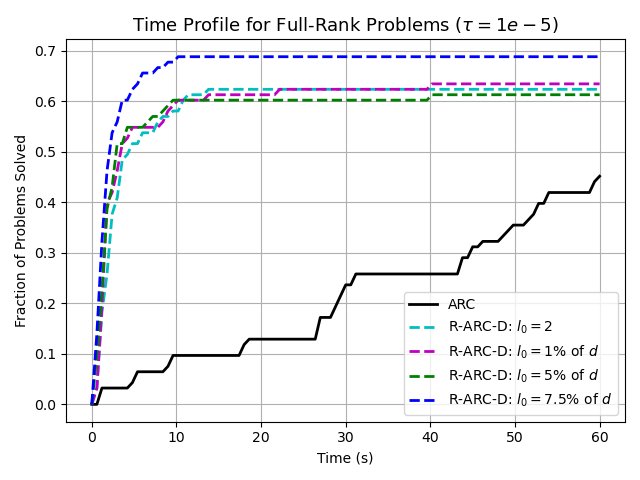 }
    \includegraphics[width=0.4\linewidth]{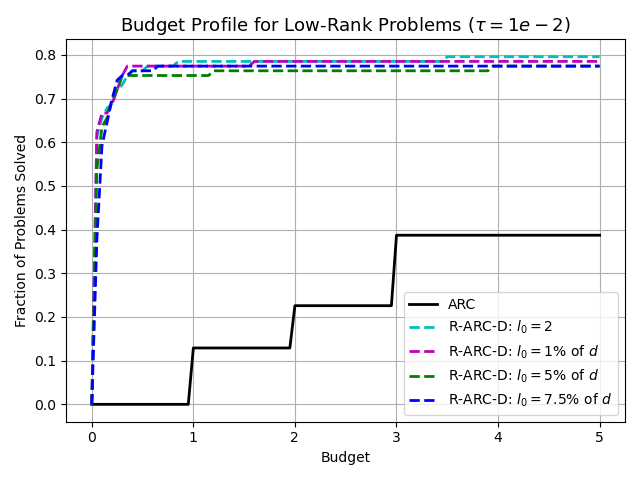}
    \includegraphics[width=0.4\linewidth]{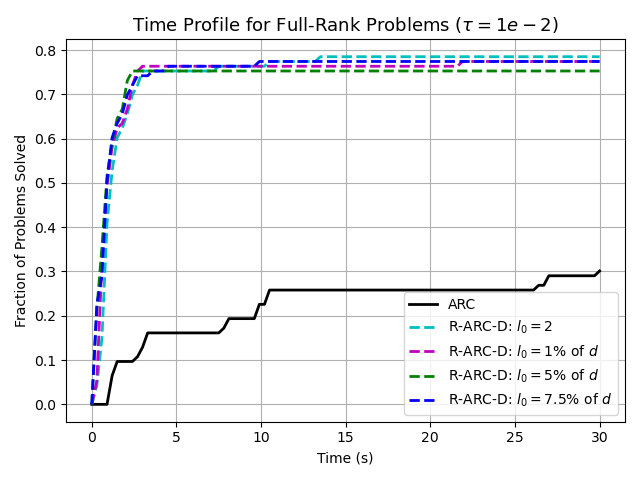}
    \caption{Varying the sketch dimension using scaled Gaussian matrices in R-ARC-D; plotting budget and time; low-rank problems}
    \label{fig:lr_gaussian_rarcd}
\end{figure}

Similarly to the R-ARC plots, we see that sketching with $l_0 = 7.5\%$ of $d$ performs best. The remaining starting sketch dimensions all perform similarly.

\paragraph{R-ARC-D vs R-ARC}

Here we compare the performance of R-ARC-D against R-ARC (and ARC). We plot R-ARC with $l=7.5\%$ of $d$ and R-ARC-D with $l_0 = 2, l_0 = 7.5\%$ of $d$. The results are shown in \autoref{fig:lr_gaussian_rarc_vs_rarcd}. Additional problem-by-problem plots for both R-ARC and R-ARC-D in Figure \ref{fig:low_rank_individual_problems}.

\begin{figure}
    \centering
    \includegraphics[width=0.4\linewidth]{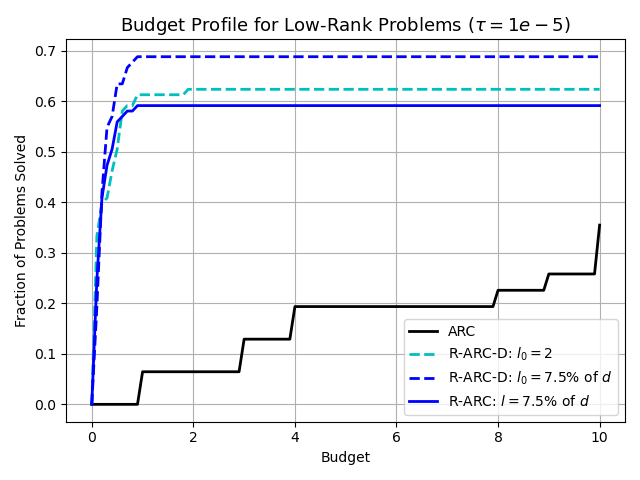}
    \includegraphics[width=0.4\linewidth]{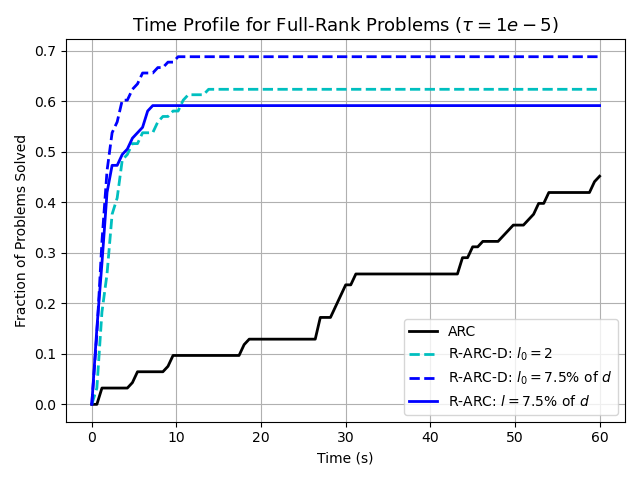 }
    \includegraphics[width=0.4\linewidth]{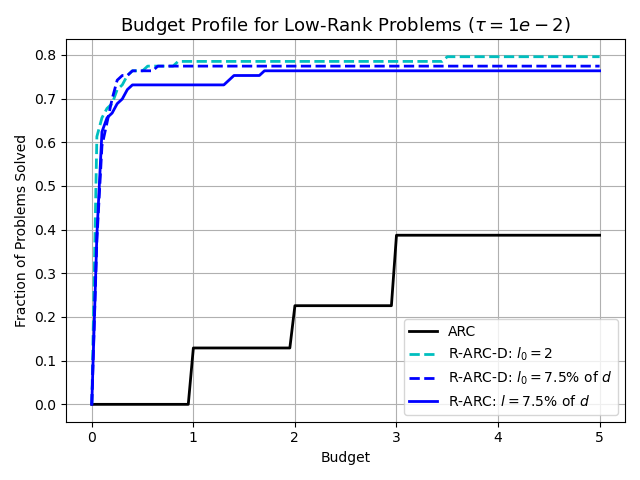}
    \includegraphics[width=0.4\linewidth]{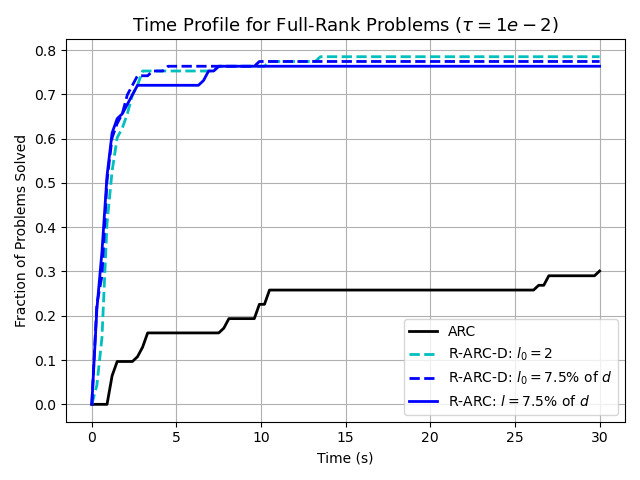}
    \caption{Varying the sketch dimension using scaled Gaussian matrices in R-ARC/R-ARC-D; plotting budget and time; low-rank problems}
    \label{fig:lr_gaussian_rarc_vs_rarcd}
\end{figure}

We see that even starting with $l_0 = 2$ leads to better performance than that of R-ARC with $l=7.5\%$ of $d$, demonstrating the benefits of the adaptive sketch dimension scheme.

\paragraph{Other sketching matrices} We now consider the performance of Haar and sampling matrices. We include comparable plots for sampling matrices can be found in Figures \ref{fig:lr_sampling_rarc}, \ref{fig:lr_sampling_rarcd} and \ref{fig:lr_sampling_rarc_vs_rarcd}; plots for Haar matrices in Figures \ref{fig:lr_haar_rarc}, \ref{fig:lr_haar_rarcd} and \ref{fig:lr_haar_rarc_vs_rarcd}. We now include a subset of plots, showing the best performance achieved by each of the sketching matrices; in each case, this was achieved by R-ARC-D. We plot the results in \autoref{fig:lr_comparing_matrices}.

\begin{figure}
    \centering
    \includegraphics[width=0.4\linewidth]{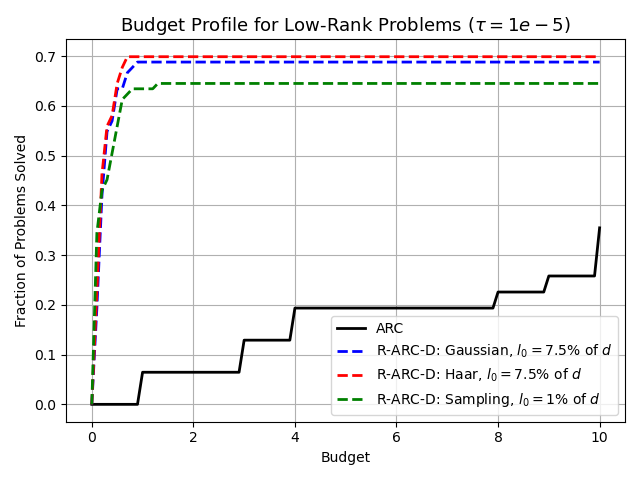}
    \includegraphics[width=0.4\linewidth]{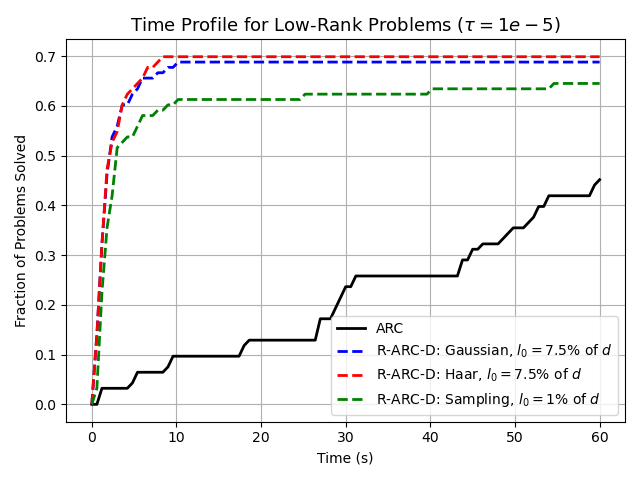 }
    \includegraphics[width=0.4\linewidth]{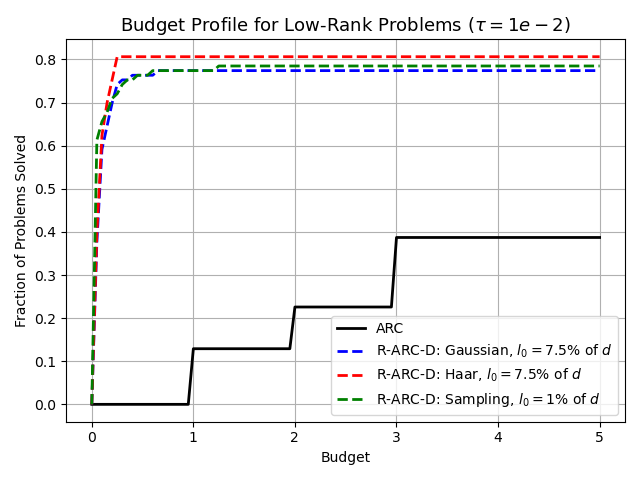}
    \includegraphics[width=0.4\linewidth]{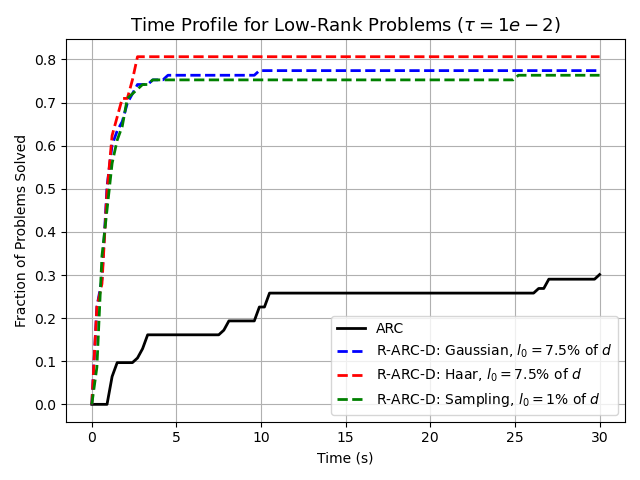}
    \caption{Comparing Gaussian matrices with Haar and sampling matrices in R-ARC-D; plotting budget and time; low-rank problems}
    \label{fig:lr_comparing_matrices}
\end{figure}

Similarly to for the full-rank problems, we see similar performance for the Gaussian and Haar matrices. However, now we see that sampling matrices perform worse than the other matrices. This further suggests that for the full-rank problems, the good performance of sampling was caused by axis alignment of CUTEst problems.

\subsection{Conclusion}
Our numerical experiments show that there are some problems for which R-ARC significantly outperforms ARC when accuracy is measured against either relative Hessians or time. In these problems, the gains in computational efficiency on the subproblem outweigh the fact that a less optimal step can be selected due to the restriction to a subspace. However, there are some problems for which R-ARC stagnates. In these problems, it appears that more problem information is required on each iteration for a step to make progress. Further algorithm development is needed to improve this aspect by for example, adding deterministic past directions and past sketched gradients.

{\small{\bibliography{combined}}}
\bibliographystyle{abbrv}

\newpage
\appendix

\section{Alternative choices of sketching matrices}
\label{sec:sketching_matrices}
Embedding properties hold for other families of matrices other that Gaussian matrices, such as for $s$-hashing matrices defined below, although with a different dependency on $l$; such ensembles are defined below as well as other that we use in this paper and a summary of their embedding properties can be found in \cite{Zhen-PhD, 2021arXiv210511815C}.

Whilst we primarily consider Gaussian matrices in the theoretical sections of this paper, we also give definitions for several other choices of sketching matrices that we refer to or use in the numerical results. We first provide a definition for \textit{$s$-hashing matrices}.

\begin{definition} \cite{10.1561/0400000060}  \label{def:s-hashing}
We define $S \in \R^{l \times d}$  to be an $s$-hashing matrix if, 
independently for each $j \in [d]$, we sample without replacement 
$i_1, i_2, \dots, i_s \in [l]$ uniformly at random and 
let $S_{i_k j} = \pm 1/\sqrt{s}$, $k = 1, 2, \dots, s$.
\end{definition}

Comparing to Gaussian matrices, $s$-hashing matrices, including in the case when $s=1$,
are sparse, having $s$ nonzero entries per column, and  they preserve the sparsity (if any) of the vector/matrix they act on; 
and the corresponding linear algebra is computationally faster. 

Another family of sparse matrices are the well known (Scaled) sampling matrices $S\in\R^{l\times d}$, which randomly select entries/rows of the vector/matrix it acts on (and scale it). 
\begin{definition}\label{def:sampling}
We define $S=(S_{ij}) \in \R^{l \times d}$ to be a scaled sampling matrix if, independently for each $i \in [l]$, we sample $j \in [d]$ uniformly at random and let $S_{ij}=\sqrt{\frac{d}{l}}$. 
\end{definition}

\begin{definition}\label{def:SRHT}
    A Subsampled-Randomized-Hadamard-Transform (SRHT) \cite{10.1145/1132516.1132597, Tropp:wr} is an $l \times d$ matrix of the form $S = S_{s}HD$ with $l \leq d$, where
    \begin{itemize}
        \item $D$ is a random $d \times d$ diagonal matrix with $\pm 1$ independent entries.
        \item $H$ is an $d \times d$ Walsh-Hadamard matrix defined by
        \begin{equation}
            H_{ij} = d^{-1/2}(-1)^{\langle (i-1)_{2}, (j-1)_{2}\rangle},
        \end{equation}
        where $(i-1)_{2}, (j-1)_{2}$ are binary representation vectors of the numbers $(i-1), (j-1)$ respectively.
        \item $S_{h}$ is a random $l \times d$ scaled sampling matrix, independent of $D$.
   \end{itemize}
\end{definition}

\begin{definition}\label{def:HRHT}
    A Hashed-Randomized-Hadamard-Transform (HRHT)  \cite{Zhen-PhD} is an $l \times d$ matrix of the form $S = S_{h}HD$ with $m \leq n$, where
    \begin{itemize}
        \item $D$ is a random $d \times d$ diagonal matrix with $\pm 1$ independent entries.
        \item $H$ is an $d \times d$ Walsh-Hadamard matrix defined by
        \begin{equation}
            H_{ij} = d^{-1/2}(-1)^{\langle (i-1)_{2}, (j-1)_{2}\rangle},
        \end{equation}
        where $(i-1)_{2}, (j-1)_{2}$ are binary representation vectors of the numbers $(i-1), (j-1)$ respectively.
        \item $S_{h}$ is a random $l \times d$ $s$-hashing or $s$-hashing variant matrix, independent of $D$.
    \end{itemize}
\end{definition}

We finally give the definition of a Haar matrices; these can be thought of as random orthogonal matrices in that $SS^{\top} = I$. 

\begin{definition}[\cite{meckesRandomMatrixTheory2019}]\label{def:haar}
    Letting $\mathbb{O}(d)$ denote the set of $d \times d$ orthogonal matrices. The Haar measure $\mu$ is the unique translation invariant probability measure on $\mathbb{O}(d)$. That is, for any fixed $O \in \mathbb{O}(d)$
    \begin{equation}
        UO \stackrel{d}{\sim} OU \stackrel{d}{\sim} O
    \end{equation}
    where $U$ is distributed according to $\mu$. A Haar-distributed random matrix is a random matrix distributed according to the Haar measure. 
\end{definition}

We now give a definition of scaled Haar matrices, which reflects the use of Haar matrices in \cite{dzahiniDirectSearchStochastic2024}.

\begin{definition}
    A scaled Haar-distributed random matrix $S$ is an $l \times d$ matrix of the form $S = \sqrt{\frac{d}{l}}\tilde{S}$ such that $[\tilde{S} \tilde{S}^{\perp}]^{\top}$ is a Haar-distributed random matrix.
\end{definition}
We note that under our definition, scaled Haar-distributed random matrices can be rectangular and will be for our purpose of random subspace methods.

Other crucial random ensembles with JL properties are the so-called
 Subsampled-Randomized-Hadamard-Transform \cite{10.1145/1132516.1132597, Tropp:wr} and Hashed-Randomized-Hadamard-Transform 
 \cite{cartis_hashing, Zhen-PhD} matrices.

\section{Fast convergence rate assuming sparsity of the Hessian matrix}

This section is mostly conceptual and is an attempt to show the fast convergence rate of \autoref{alg:R-ARC} can be achieved without assuming subspace embedding of the Hessian matrix. 
Here, we maintain $\nEps$ as $\min \{k: \normTwo{\grad f(x_{k+1})} \leq \epsilon \} $, similarly to the last section. 
However, in the definition of true iterations, we replace the condition \eqref{tmp:CBGN:1} on subspace embedding of the Hessian with the condition that the sketched Hessian $S_k \hessFK$ \reply{has a small norm}. This may be achieved when the Hessian matrix has sparse rows and we choose $S_k$ to be a scaled sampling matrix.
We show that this new definition of true iterations still allows the same $\mathO{\epsilon^{-3/2}}$ iteration complexity to drive the norm of the objective's gradient norm below $\epsilon$.
Specifically, true iterations are defined as follows.

\begin{definition} \label{def:true:CBGN:SparseHess}
Let $\epS \in (0,1)$, $\sMax >0$. Iteration $k$ is ($\epS, \sMax)$-true if 
\begin{align}
    &\normTwo{S_k \hessFK} \leq c_k \epsHalf, 
    \label{smallSketchedHessian} \\
    &\normTwo{S_k \gradFK}^2 \geq \bracket{1-\epS} \epsilon^2, 
    \label{skGradLowerBound_sparse_Hess} \\
    &\normTwo{S_k} \leq \sMax \label{S_k_norm_bound_Sparse_hess},
\end{align}
where $c_k = \sqrt{\frac{4\bracket{1-\epS}^{1/2} \sMax}{3\alphaMax}}$ and $\alphaMax$ is a user-chosen constant in \autoref{alg:R-ARC}. 
\end{definition}

Note that the desired accuracy $\epsilon$ appears in this particular definition of true iterations. 

Consequently, the requirements on the objective and the sketching dimension $l$ may be stronger for smaller $\epsilon$. For simplicity, we assume $\kappa_T = 0$ (where $\kappaT$ is a user chosen parameter in \eqref{tmp:CBGN:5} in \autoref{alg:R-ARC}) in this section, namely $\grad \mKHat{\sKHat} = 0$, and it follows from \eqref{tmp:CBGN:5} that
\begin{equation}
    S_k \gradFK = \frac{1}{\alphaK} S_k S_k^T 
    \sKHat \normTwo{S_k^T \sKHat} - S_k\hessFK S_k^T \sKHat.
    \label{exactModelGradZero}
\end{equation}

The proofs that \autoref{AA3} and \autoref{AA5} are satisfied are identical to the previous section, while the following technical lemma helps us to 
satisfy \autoref{AA4}. 

\begin{lemma}\label{spHessLem}
Let $\epsilon>0$.
Let $\epS \in (0,1)$, $\kappaT=0$. 
Suppose we have \eqref{smallSketchedHessian}, \eqref{skGradLowerBound_sparse_Hess} and
\eqref{S_k_norm_bound_Sparse_hess}.
Then 
\begin{equation}
    \normTwo{S_k^T \sKHat} \geq \alphaK
    \sqrtFrac{\bracket{1-\epS}^{1/2}\epsilon}{3\sMax\alphaMax}. \notag
\end{equation}
\end{lemma}

\begin{proof}
Let $b = \normTwo{S_k \hessFK}, x = \normTwo{S_k^T \sKHat}$,
then taking 2-norm of \eqref{exactModelGradZero} with
\eqref{skGradLowerBound_sparse_Hess}, $\normTwo{S_k}\leq \sMax$ and the triangle inequality
gives 
\begin{align}
    & \oneMinusEpSHalf \epsilon \leq \normTwo{S_k \gradFK}
    \leq \frac{\sMax}{\alphaK}x^2 + bx \notag\\
    \implies 
    &\frac{\sMax}{\alphaK} x^2 + bx - \oneMinusEpSHalf\epsilon \geq 0
    \notag \\
    \implies 
    & x^2 + \frac{\alphaK b}{\sMax}x - \frac{\oneMinusEpSHalf\epsilon
    \alphaK}{\sMax} \geq 0 \notag \\
    \implies
    & \bracket{x + \frac{\alphaK b}{2\sMax}}^2 
    \geq \frac{\oneMinusEpSHalf\epsilon\alphaK}{\sMax}
    + \frac{\alphaK^2 b^2}{4\sMax^2} \notag \\
    \impliesSince{x,b\geq 0} 
    & x \geq \sqrt{
    \frac{\oneMinusEpSHalf\epsilon\alphaK}{\sMax}
    + \frac{\alphaK^2 b^2}{4\sMax^2} } - \frac{\alphaK b}{2\sMax}.
    \notag
\end{align}

Introduce $a = \frac{\oneMinusEpSHalf\epsilon\alphaK}{\sMax}$ and the function $
y(b) = \frac{\alphaK b}{2\sMax}$, then the above gives
\begin{equation}
    x \geq \sqrt{a + y(b)^2} - y(b). \label{tmp-2022-1-14-9}
\end{equation}

We note, by taking derivative, that given $y(b) \geq 0$, the RHS of \eqref{tmp-2022-1-14-9} is monotonically decreasing
with $y(b)$. Therefore given $b\leq c_k 
\epsilon^{\frac{1}{2}}$ and thus $y(b) \leq y(c_k \epsilon^{\frac{1}{2}})$, we have
\begin{equation}
    x \geq \sqrt{a + y(c_k \pow{\epsilon}{\frac{1}{2}})^2}
    - y(c_k \pow{\epsilon}{\frac{1}{2}}).
    \notag
\end{equation}

The choice of $c_k = \sqrt{\frac{4\bracket{1-\epS}^{1/2} \sMax}{3\alphaMax}} \leq \sqrt{\frac{4\bracket{1-\epS}^{1/2} \sMax}{3\alpha_k}}$ gives 
$a \geq 3 y\bracket{c_k \epsilon^{\frac{1}{2}}}^2$.
And therefore we have $x \geq \fun{y}{c_k \pow{\epsilon}{\frac{1}{2}}}$. Noting that $x = \normTwo{S_k^T \sKHat}$ and substituting the expression for $y$ and $c_k$ gives the desired result. 
    
\end{proof}

\begin{lemma}
Following the framework of \autoref{alg:R-ARC} 
with $\kappaT = 0$, let $\epsilon > 0.$ Define
true iterations as iterations
that satisfy \eqref{smallSketchedHessian}, \eqref{skGradLowerBound_sparse_Hess},
and 
\eqref{S_k_norm_bound_Sparse_hess}
. Then
\autoref{AA4} is satisfied with
\begin{equation}
    h(\epsilon, \alphaK) = 
    \frac{\theta \alphaK^2 \epsilon^{3/2}}{3}\squareBracket{
        \frac{\bracket{1-\epS}^{1/2}}{3\sMax\alphaMax}
    }^{3/2}. \label{tmp-2022-1-14-10}
\end{equation}

\end{lemma}

\begin{proof}
A true and successful iteration $k$ gives
\begin{equation}
    \fK - f(x_k +s_k) \geq \frac{\theta}{3\alphaK}\normTwo{S_k^T\sKHat}^3
    \notag
\end{equation}

by \autoref{succStepDecrease} and combining with
the conclusion of \autoref{spHessLem} gives the result. 
\end{proof}

With \autoref{AA3}, \autoref{AA4} and \autoref{AA5} satisfied, applying \autoref{thm2} gives the following result for \autoref{alg:R-ARC}.

\begin{theorem}
    Let $f$ be bounded below by $f^*$ and twice continuously differentiable with $\LH$-Lipschitz continuous Hessian. Run \autoref{alg:R-ARC} for $N$ iterations. Suppose \autoref{AA2} hold with $\deltaS\in (0,1)$ and true iterations defined in \autoref{def:true:CBGN:SparseHess}. Suppose $\deltaS < \frac{c}{(c+1)^2}$.
    
    Then for any $\deltaOne \in (0,1)$ such that $g(\deltaS, \deltaOne) >0$ where $g(\deltaS, \deltaOne)$ is defined in \eqref{eqn:gDeltaSDeltaOneDef}. If $N$ satisfies
    \begin{equation}
         N \geq \gDeltaSDeltaOne \squareBracket{
         \fZeroMinusfStarOverH
         + \frac{\newL}{1+c}},
    \end{equation}
    where $h$ is given in \eqref{tmp-2022-1-14-10}, $\alphaLow$ is given in \eqref{eq:alphaLow:CBGN} and $\alphaMin, \newL$ are given in \autoref{lem::alphaMin}; then we have
    \begin{equation}
        \probability{\min_{k\leq N} \{\normTwo{\grad f(x_{k+1})}\} \leq \epsilon } \geq 1 - e^{-\frac{\delta_1^2}{2} (1-\deltaS)N}. \nonumber
    \end{equation}
\end{theorem}

\begin{remark}
In order to satisfy \autoref{AA2}, we require that at each iteration, with positive probability, \eqref{smallSketchedHessian}, \eqref{skGradLowerBound_sparse_Hess} and \eqref{S_k_norm_bound_Sparse_hess} hold. This maybe achieved for objective functions whose Hessian only has a few non-zero rows, with $S$ being a scaled sampling matrix. Because if $\hessFK$ only has a few non-zero rows, we have that $S_k \hessFK=0$ with positive probability, thus satisfying  \eqref{smallSketchedHessian}. Scaled sampling matrices also satisfy \eqref{skGradLowerBound_sparse_Hess} and \eqref{S_k_norm_bound_Sparse_hess} (See Lemmas 4.4.13 and 4.4.14 in \cite{Zhen-PhD}).
\end{remark}

\newpage
\section{CUTEst Problems}

\begin{table}[H]
\centering
\begin{tabular}{c l c c c c}  
\toprule
\# & Problem & $d$ & $f(x_{0})$ & $f(x^{*})$ & Parameters \\
\midrule
1 & ARWHEAD & 1000 & 2997.000000 & 0.000000 & N = 1000 \\
2 & BDEXP & 1000 & 270.129225 & 0.000000 & N = 1000 \\
3 & BOX & 1000 & 0.000000 & -177.370594 & N = 1000 \\
4 & BOXPOWER & 1000 & 8805.056207 & 0.000000 & N = 1000 \\
5 & BROYDN7D & 1000 & 3518.841961 & 345.217283 & N/2 = 500 \\
6 & CHARDIS1 & 998 & 132.816667 & 0.000000 & NP1 = 500 \\
7 & COSINE & 1000 & 876.704979 & -99.000000 & N = 1000 \\
8 & CURLY10 & 1000 & -0.063016 & -100316.300000 & N = 1000 \\
9 & CURLY20 & 1000 & -0.134062 & -100316.200000 & N = 1000 \\
10 & DIXMAANA1 & 1500 & 14251.000000 & 1.000000 & M = 500 \\
11 & DIXMAANF & 1500 & 20514.875000 & 1.000000 & M = 500 \\
12 & DIXMAANP & 1500 & 35635.810853 & 1.000000 & M = 500 \\
13 & ENGVAL1 & 1000 & 58941.000000 & 0.000000 & N = 1000 \\
14 & FMINSRF2 & 1024 & 27.712415 & 1.000000 & P = 32 \\
15 & FMINSURF & 1024 & 28.430936 & 1.000000 & P = 32 \\
16 & NCB20 & 1010 & 2002.002000 & 916.060534 & N = 1000 \\
17 & NCB20B & 1000 & 2000.000000 & 1676.011207 & N = 1000 \\
18 & NONCVXU2 & 1000 & 2592247505.400722 & 2316.808400 & N = 1000 \\
19 & NONCVXUN & 1000 & 2672669991.246090 & 2316.808400 & N = 1000 \\
20 & NONDQUAR & 1000 & 1006.000000 & 0.000000 & N = 1000 \\
21 & ODC & 992 & 0.000000 & -0.011327 & (NX, NY) = (31, 32) \\
22 & PENALTY3 & 1000 & 998001817780.515259 & 0.001000 & N/2 = 500 \\
23 & POWER & 1000 & 250500250000.000000 & 0.000000 & N = 1000 \\
24 & RAYBENDL & 1022 & 98.025849 & 96.242400 & NKNOTS = 512 \\
25 & SCHMVETT & 1000 & -2854.345107 & -2994.000000 & N = 1000 \\
26 & SINEALI & 1000 & -0.841471 & -99901.000000 & N = 1000 \\
27 & SINQUAD & 1000 & 0.656100 & -3.000000 & N = 1000 \\
28 & TOINTGSS & 1000 & 8992.000000 & 10.102040 & N = 1000 \\
\bottomrule
\end{tabular}
\caption{The 28 full-rank CUTEst test problems used for data profiles}\label{tab:cutest_fullrank}
\end{table}

\begin{table}[H]
\centering
\begin{tabular}{c l c c c c c}  
\toprule
\# & Problem & $d$ & $r$ & $f(x_{0})$ & $f(x^{*})$ & Parameters \\
\midrule
1 & l-ARTIF & 1000 & 100 & 18.295573 & 0.000000 & N = 100 \\
2 & l-ARWHEAD & 1000 & 100 & 297.000000 & 0.000000 & N = 100 \\
3 & l-BDEXP & 1000 & 100 & 26.525716 & 0.000000 & N = 100 \\
4 & l-BOX & 1000 & 100 & 0.000000 & -11.240440 & N = 100 \\
5 & l-BOXPOWER & 1000 & 100 & 866.246207 & 0.000000 & N = 100 \\
6 & l-BROYDN7D & 1000 & 100 & 350.984196 & 40.122840 & N/2 = 50 \\
7 & l-CHARDIS1 & 1000 & 98 & 12.816667 & 0.000000 & NP1 = 50 \\
8 & l-COSINE & 1000 & 100 & 86.880674 & -99.000000 & N = 100 \\
9 & l-CURLY10 & 1000 & 100 & -0.006237 & -10031.630000 & N = 100 \\
10 & l-CURLY20 & 1000 & 100 & -0.012965 & -10031.630000 & N = 100 \\
11 & l-DIXMAANA1 & 1000 & 90 & 856.000000 & 1.000000 & M = 30 \\
12 & l-DIXMAANF & 1000 & 90 & 1225.291667 & 1.000000 & M = 30 \\
13 & l-DIXMAANP & 1000 & 90 & 2128.648049 & 1.000000 & M = 30 \\
14 & l-ENGVAL1 & 1000 & 100 & 5841.000000 & 0.000000 & N = 100 \\
15 & l-FMINSRF2 & 1000 & 121 & 25.075462 & 1.000000 & P = 11 \\
16 & l-FMINSURF & 1000 & 121 & 30.430288 & 1.000000 & P = 11 \\
17 & l-NCB20 & 1000 & 110 & 202.002000 & 179.735800 & N = 100 \\
18 & l-NCB20B & 1000 & 100 & 200.000000 & 196.680100 & N = 100 \\
19 & l-NONCVXU2 & 1000 & 100 & 2639748.043569 & 231.680840 & N = 100 \\
20 & l-NONCVXUN & 1000 & 100 & 2727010.761416 & 231.680840 & N = 100 \\
21 & l-NONDQUAR & 1000 & 100 & 106.000000 & 0.000000 & N = 100 \\
22 & l-ODC & 1000 & 100 & 0.000000 & -0.019802 & (NX, NY) = (10, 10) \\
23 & l-OSCIGRNE & 1000 & 100 & 306036001.125000 & 0.000000 & N = 100 \\
24 & l-PENALTY3 & 1000 & 100 & 98017980.901195 & 0.001000 & N/2 = 50 \\
25 & l-POWER & 1000 & 100 & 25502500.000000 & 0.000000 & N = 100 \\
26 & l-RAYBENDL & 1000 & 126 & 98.027973 & 96.242400 & NKNOTS = 64 \\
27 & l-SCHMVETT & 1000 & 100 & -280.286393 & -2994.000000 & N = 100 \\
28 & l-SINEALI & 1000 & 100 & -0.841471 & -9901.000000 & N = 100 \\
29 & l-SINQUAD & 1000 & 100 & 0.656100 & -3.000000 & N = 100 \\
30 & l-TOINTGSS & 1000 & 100 & 892.000000 & 10.102040 & N = 100 \\
31 & l-YATP2SQ & 1000 & 120 & 91584.340973 & 0.000000 & N = 10 \\
\bottomrule
\end{tabular}
\caption{The 31 low-rank CUTEst test problems used for data profiles}\label{tab:cutest_lowrank}
\end{table}

\section{Additional Numerical Experiments}
\subsection{Individual problem plots for Gaussian matrices}
\subsubsection{Full-rank problems}

\begin{figure}[H]
    \centering
    \includegraphics[width=0.39\linewidth]{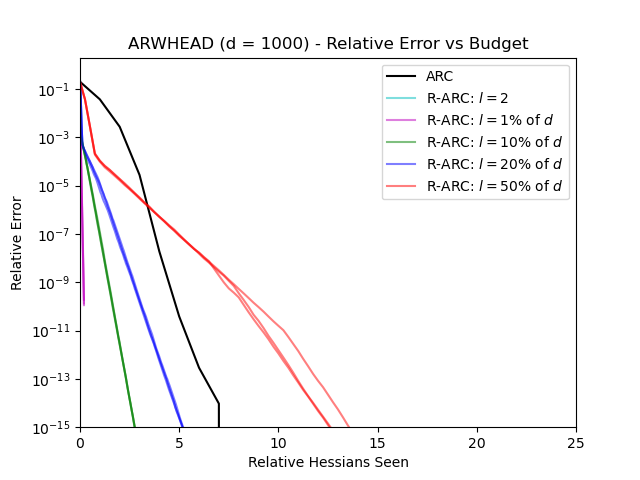}
    \includegraphics[width=0.39\linewidth]{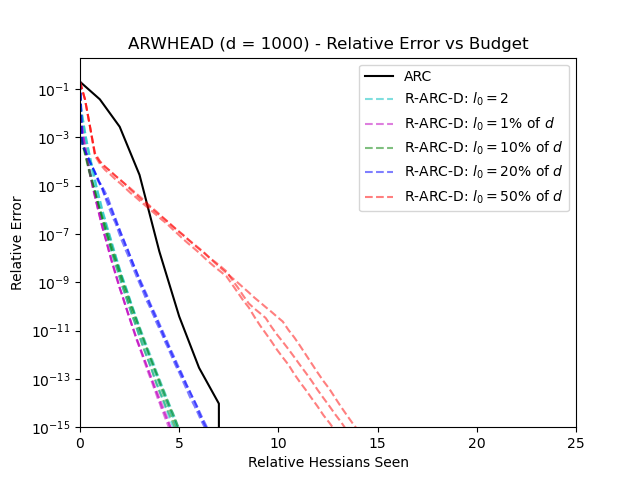}
    \includegraphics[width=0.39\linewidth]{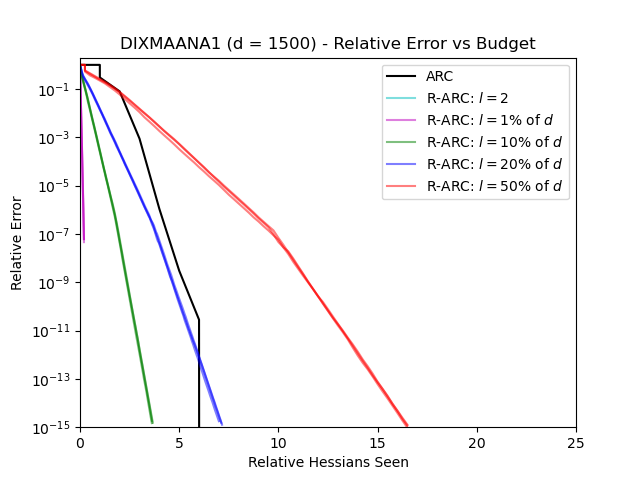}
    \includegraphics[width=0.39\linewidth]{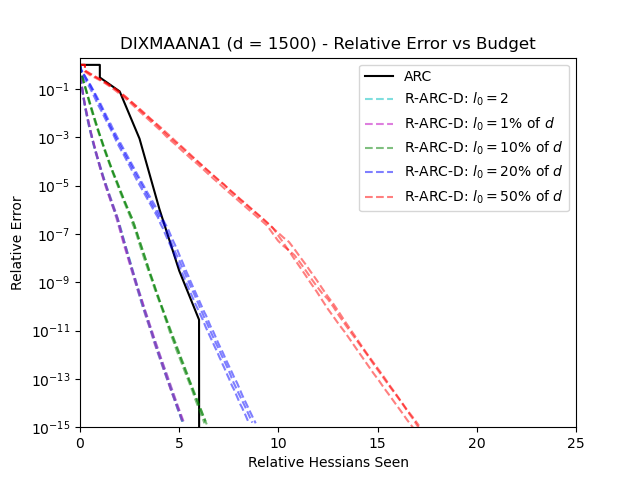}
    \includegraphics[width=0.39\linewidth]{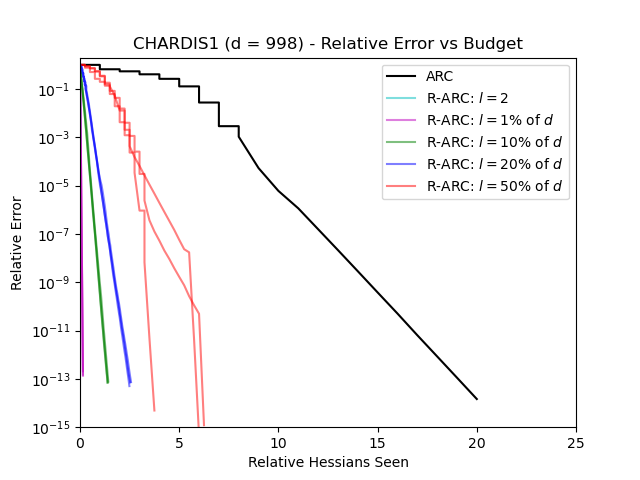}
    \includegraphics[width=0.39\linewidth]{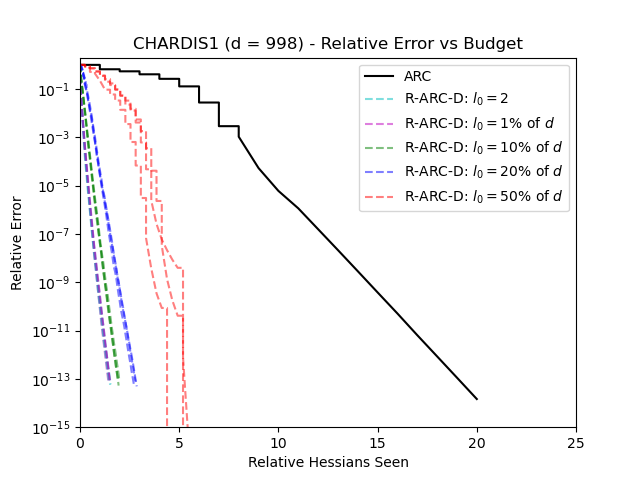}
    \includegraphics[width=0.39\linewidth]{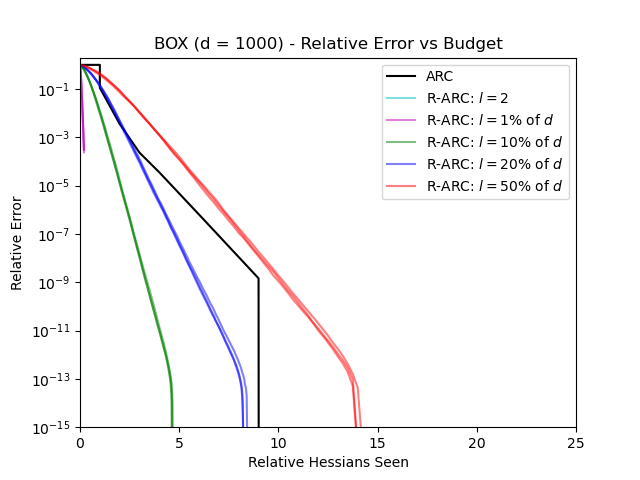}
    \includegraphics[width=0.39\linewidth]{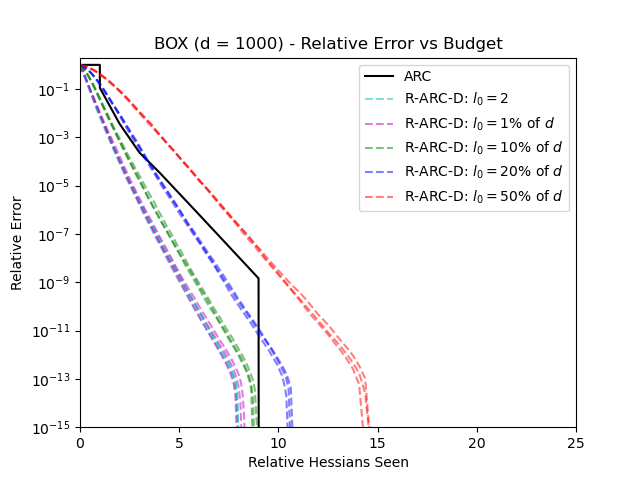}
    \caption{Full-rank individual problems from CUTEst, comparing R-ARC with ARC: budget plots}\label{fig:full_rank_individual_problems}
\end{figure}

\clearpage
\subsubsection{Low-rank problems}

\begin{figure}[H]
    \centering
    \includegraphics[width=0.39\linewidth]{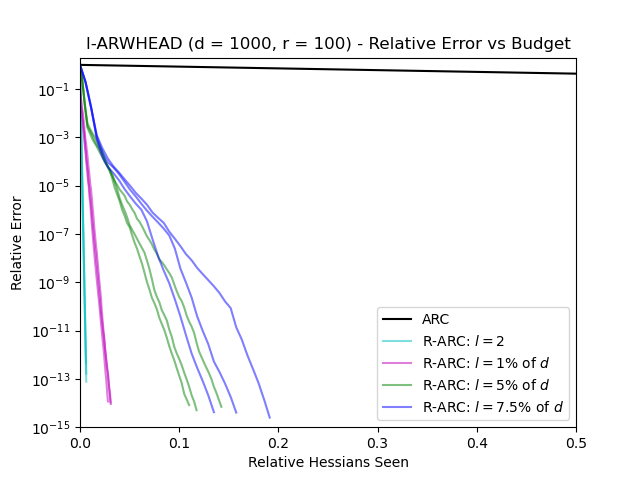}
    \includegraphics[width=0.39\linewidth]{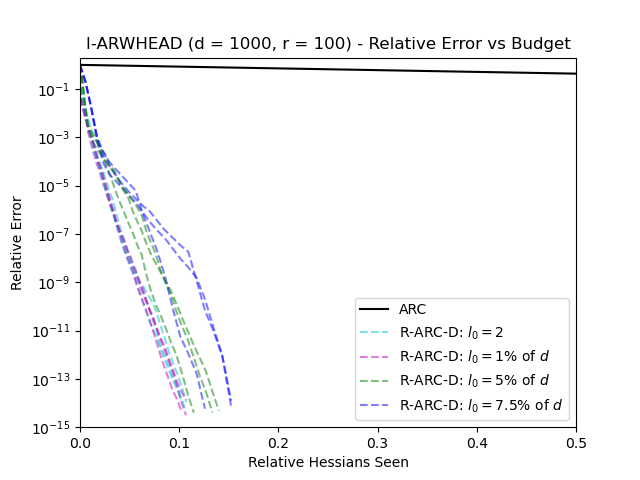}
    \includegraphics[width=0.39\linewidth]{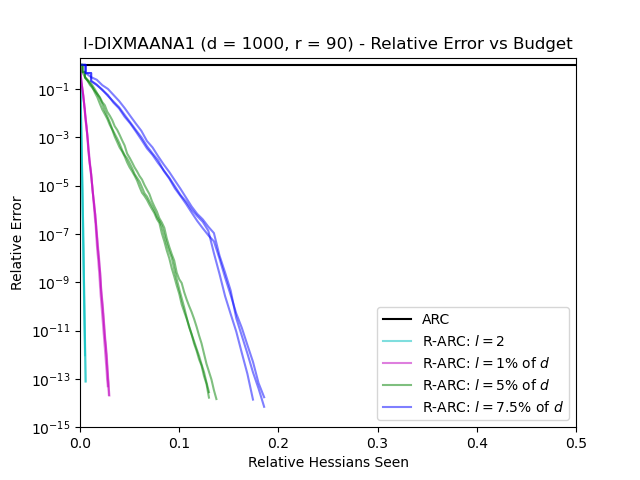}
    \includegraphics[width=0.39\linewidth]{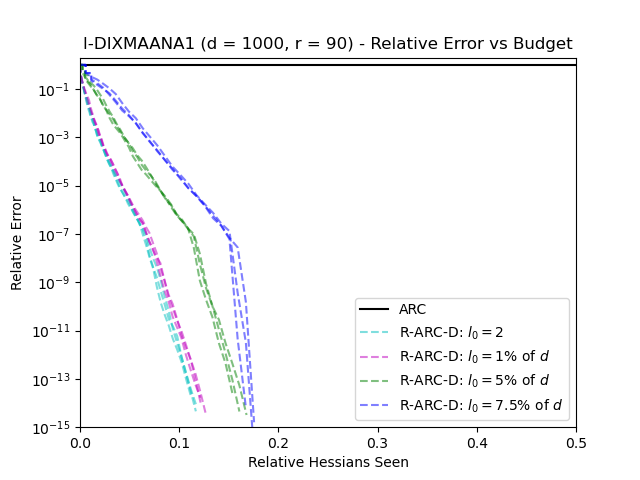}
    \includegraphics[width=0.39\linewidth]{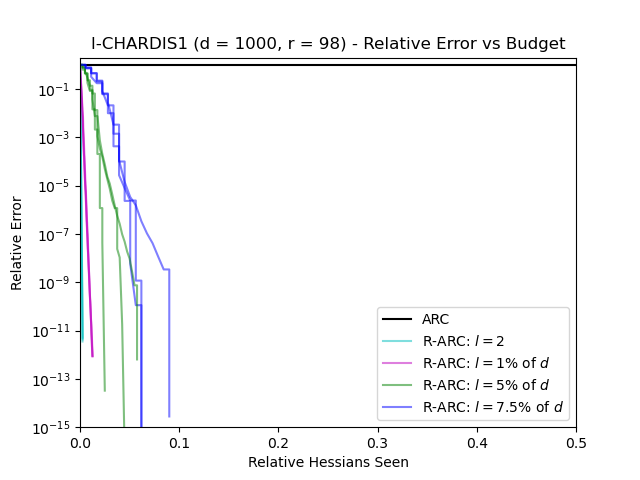}
    \includegraphics[width=0.39\linewidth]{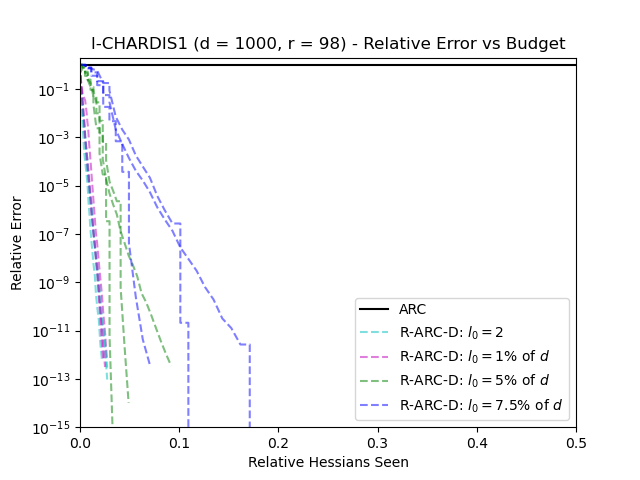}
    \includegraphics[width=0.39\linewidth]{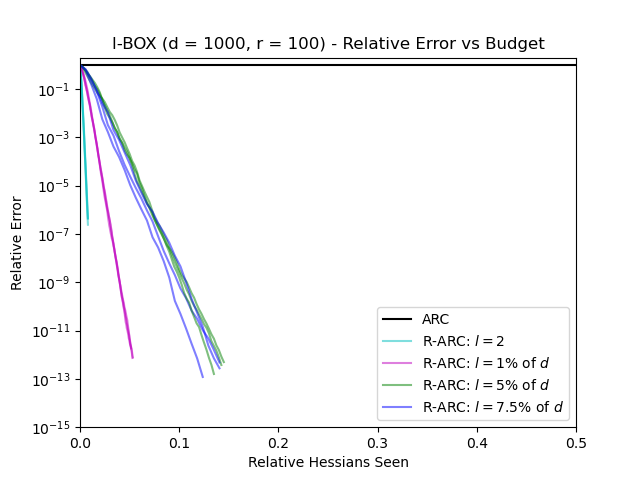}
    \includegraphics[width=0.39\linewidth]{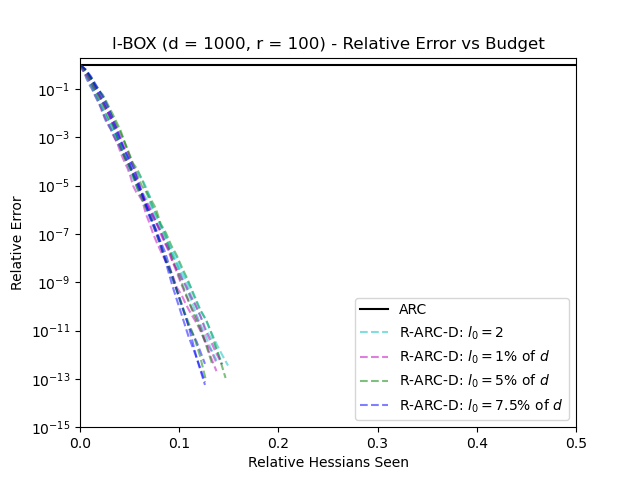}
    \caption{Low-rank individual problems from CUTEst, comparing R-ARC with ARC: budget plots}\label{fig:low_rank_individual_problems}
\end{figure}

\clearpage

\subsection{Sampling matrices}

\subsubsection{Full-rank problems}

\begin{figure}[H]
    \centering
    \includegraphics[width=0.4\linewidth]{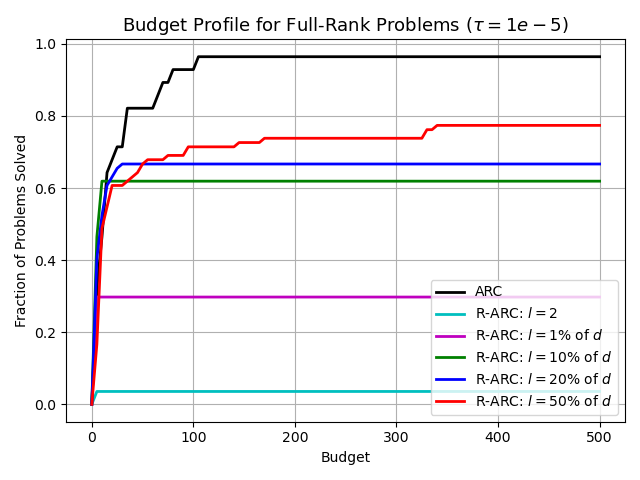}
    \includegraphics[width=0.4\linewidth]{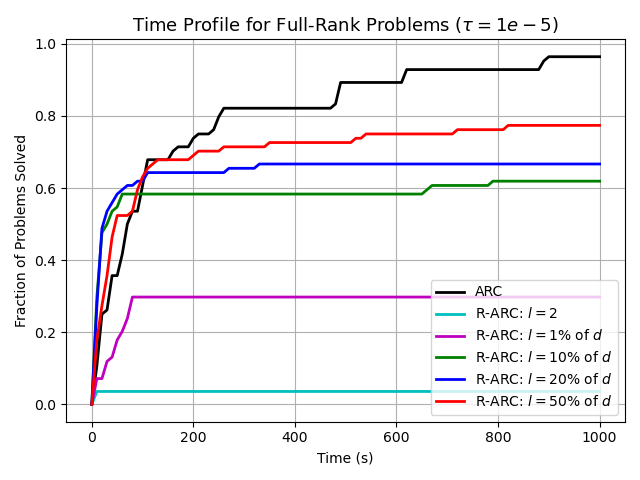}
    \includegraphics[width=0.4\linewidth]{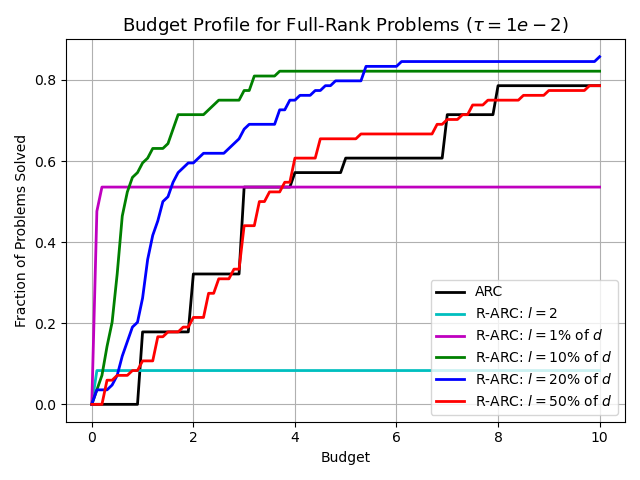}
    \includegraphics[width=0.4\linewidth]{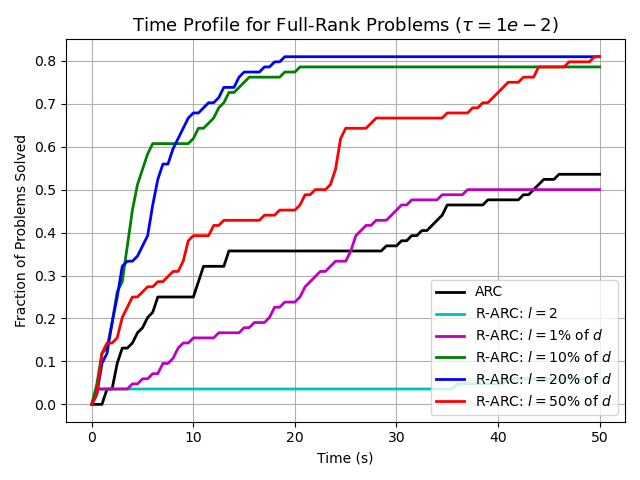}
    \caption{Varying the sketch dimension using scaled sampling matrices in R-ARC; plotting budget and time; full-rank problems}
    \label{fig:sampling_rarc}
\end{figure}

\begin{figure}
    \centering
    \includegraphics[width=0.4\linewidth]{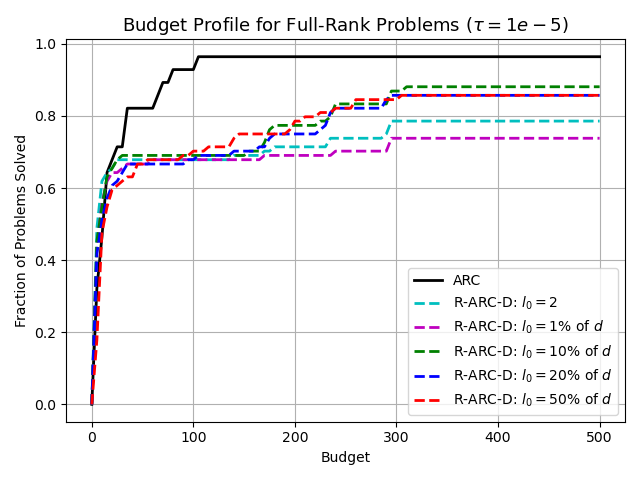}
    \includegraphics[width=0.4\linewidth]{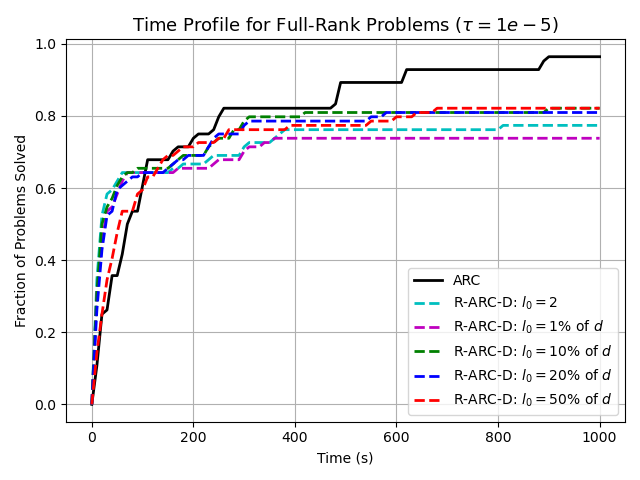 }
    \includegraphics[width=0.4\linewidth]{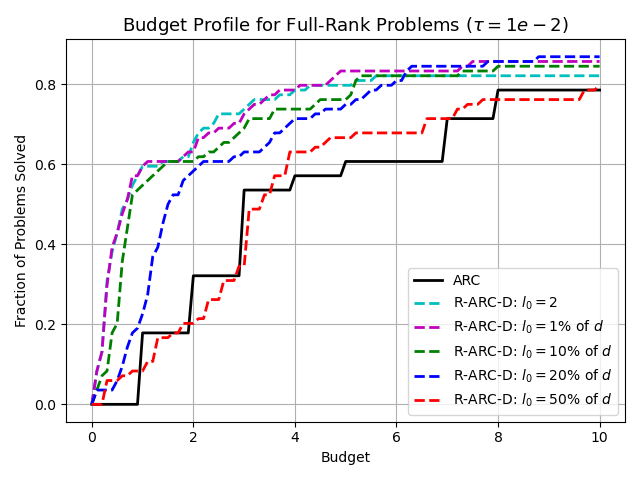}
    \includegraphics[width=0.4\linewidth]{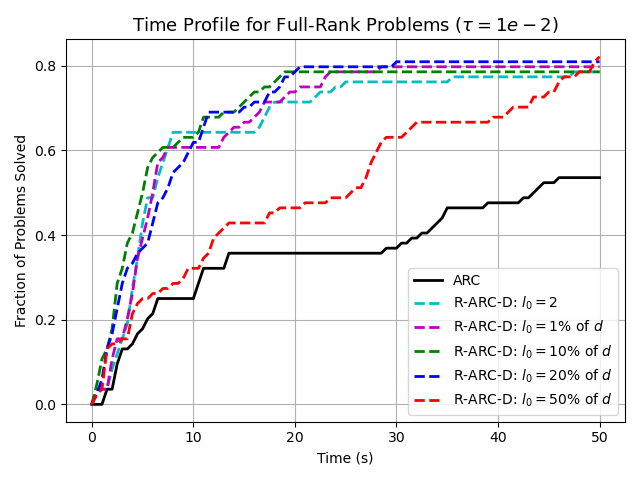}
    \caption{Varying the sketch dimension using scaled sampling matrices in R-ARC-D; plotting budget and time; full-rank problems}
    \label{fig:sampling_rarcd}
\end{figure}

\begin{figure}
    \centering
    \includegraphics[width=0.4\linewidth]{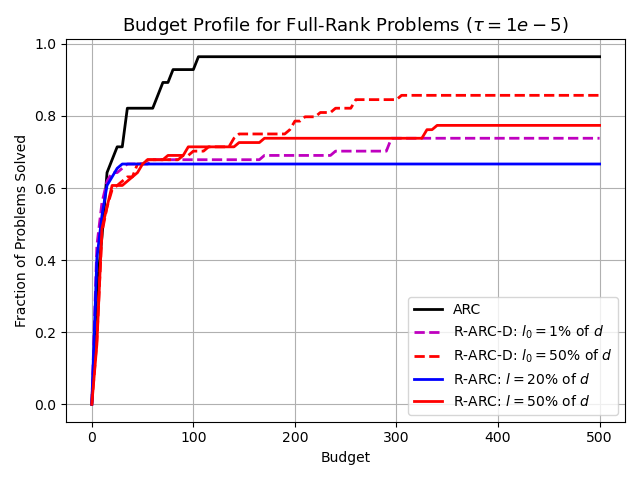}
    \includegraphics[width=0.4\linewidth]{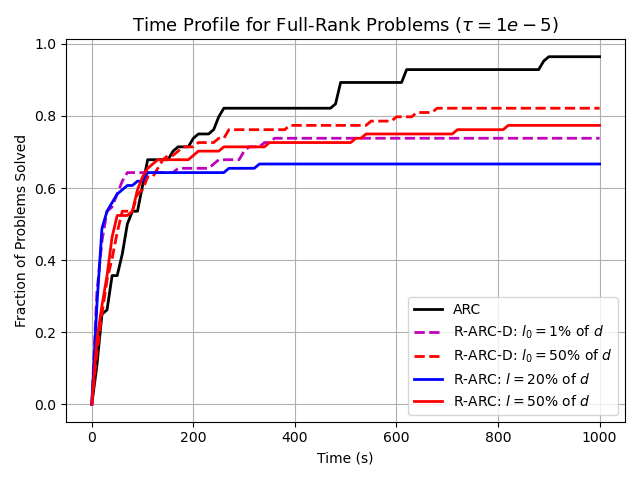 }
    \includegraphics[width=0.4\linewidth]{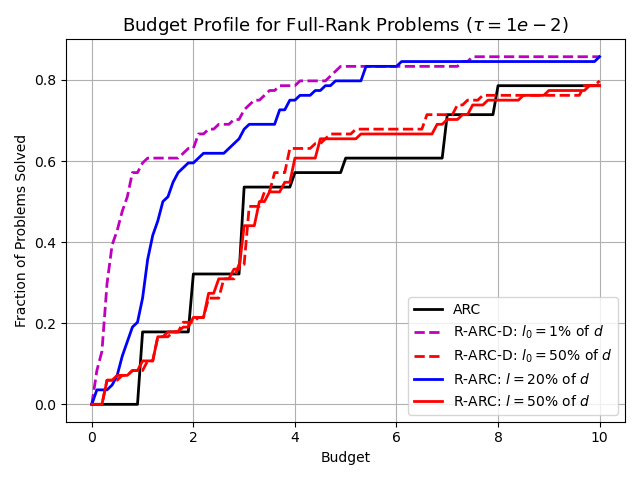}
    \includegraphics[width=0.4\linewidth]{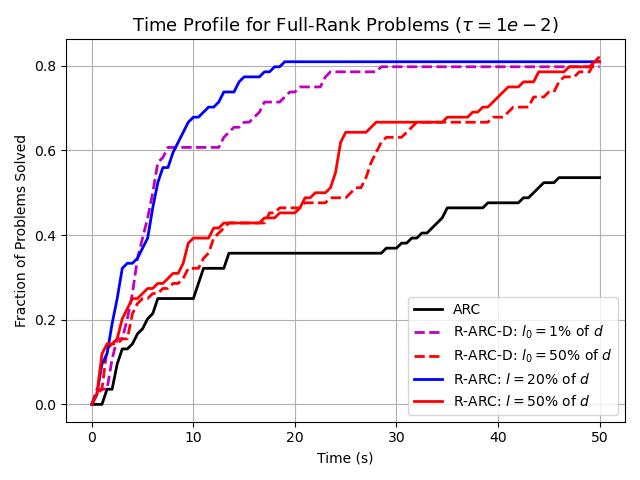}
    \caption{Varying the sketch dimension using scaled sampling matrices in R-ARC/R-ARC-D; plotting budget and time; full-rank problems}
    \label{fig:sampling_rarc_vs_rarcd}
\end{figure}

\clearpage
\subsubsection{Low-rank problems}

\begin{figure}[H]
    \centering
    \includegraphics[width=0.4\linewidth]{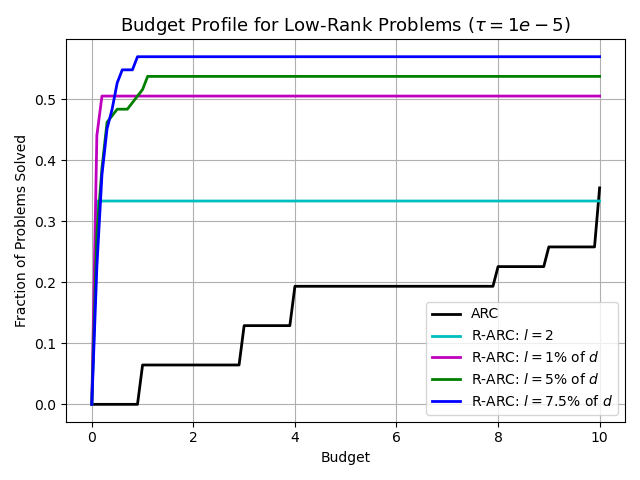}
    \includegraphics[width=0.4\linewidth]{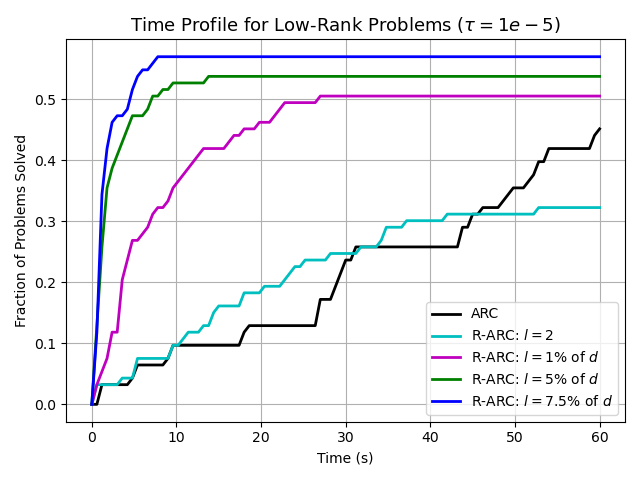}
    \includegraphics[width=0.4\linewidth]{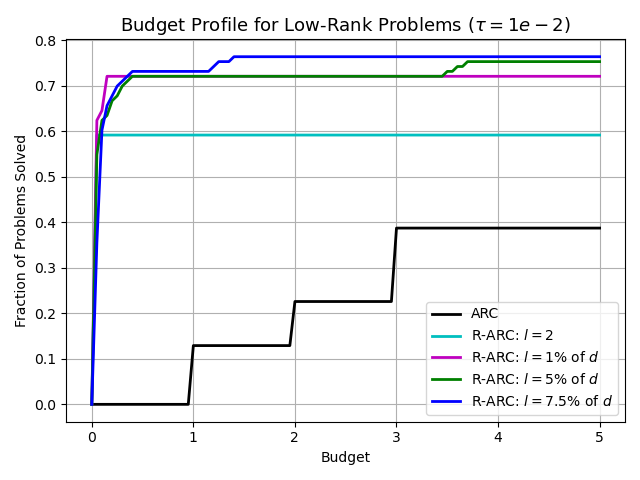}
    \includegraphics[width=0.4\linewidth]{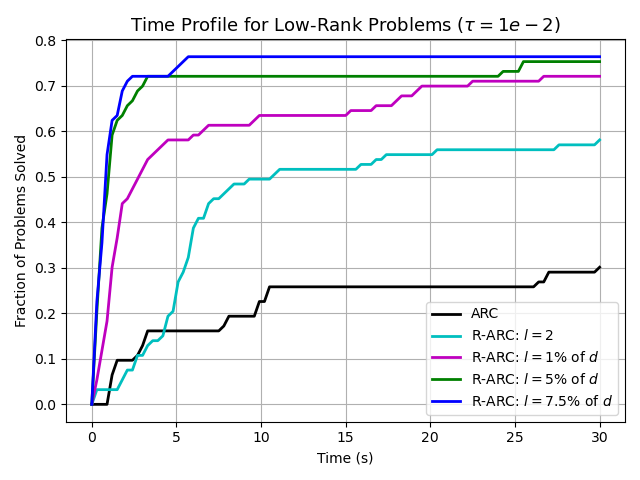}
    \caption{Varying the sketch dimension using scaled sampling matrices in R-ARC; plotting budget and time; low-rank problems}
    \label{fig:lr_sampling_rarc}
\end{figure}

\begin{figure}
    \centering
    \includegraphics[width=0.4\linewidth]{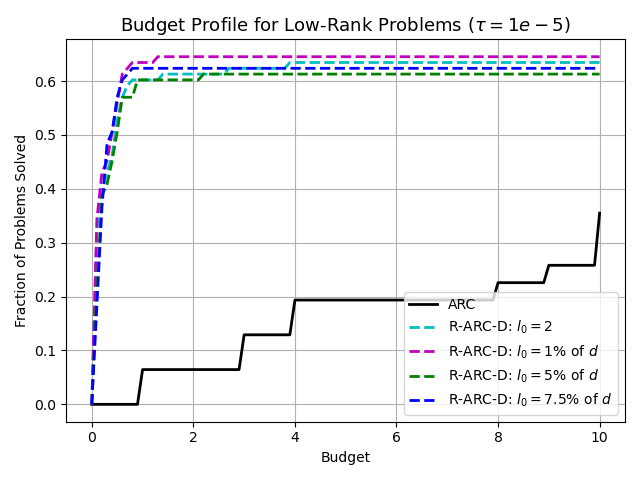}
    \includegraphics[width=0.4\linewidth]{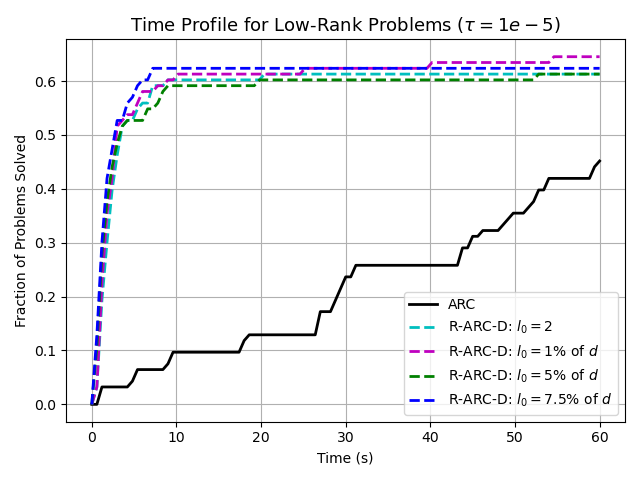 }
    \includegraphics[width=0.4\linewidth]{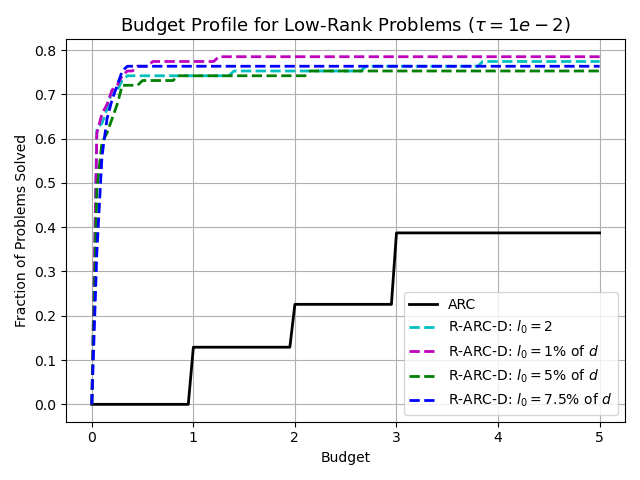}
    \includegraphics[width=0.4\linewidth]{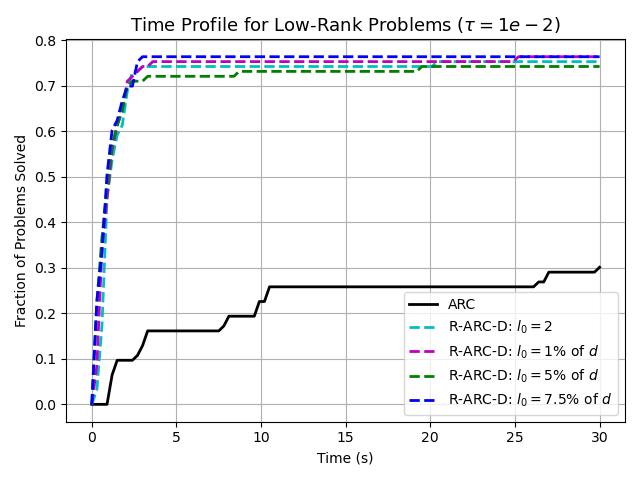}
    \caption{Varying the sketch dimension using scaled sampling matrices in R-ARC-D; plotting budget and time; low-rank problems}
    \label{fig:lr_sampling_rarcd}
\end{figure}

\begin{figure}
    \centering
    \includegraphics[width=0.4\linewidth]{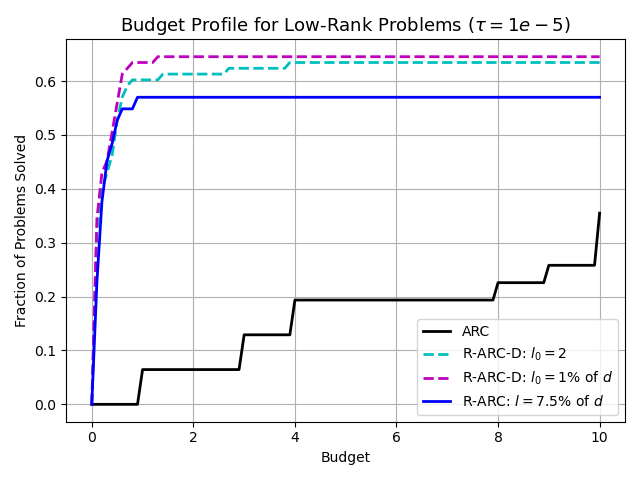}
    \includegraphics[width=0.4\linewidth]{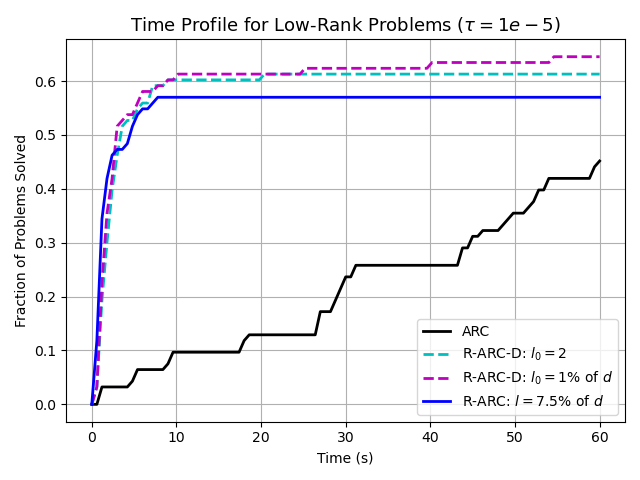 }
    \includegraphics[width=0.4\linewidth]{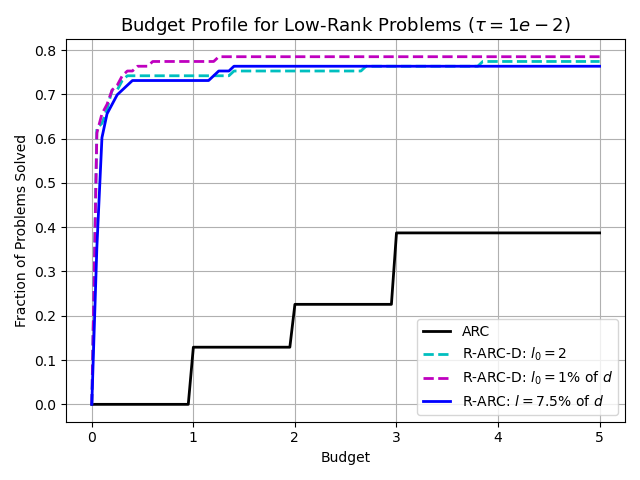}
    \includegraphics[width=0.4\linewidth]{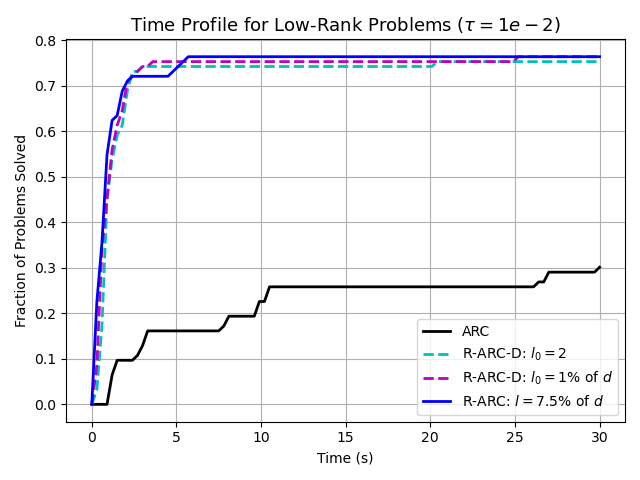}
    \caption{Varying the sketch dimension using scaled sampling matrices in R-ARC/R-ARC-D; plotting budget and time; low-rank problems}
    \label{fig:lr_sampling_rarc_vs_rarcd}
\end{figure}

\clearpage
\subsection{Haar matrices}

\subsubsection{Full-rank problems}

\begin{figure}[H]
    \centering
    \includegraphics[width=0.4\linewidth]{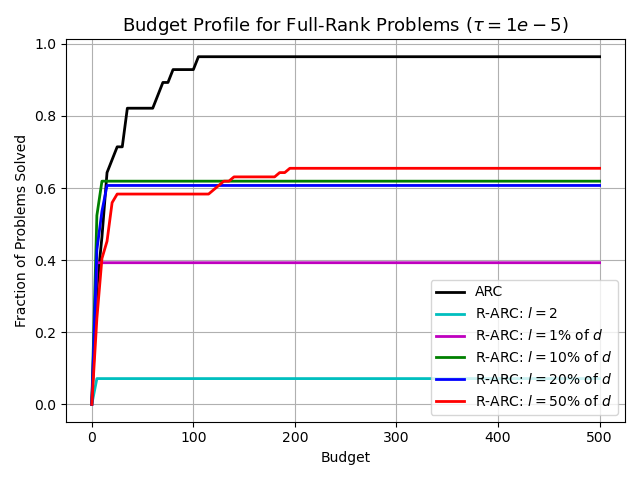}
    \includegraphics[width=0.4\linewidth]{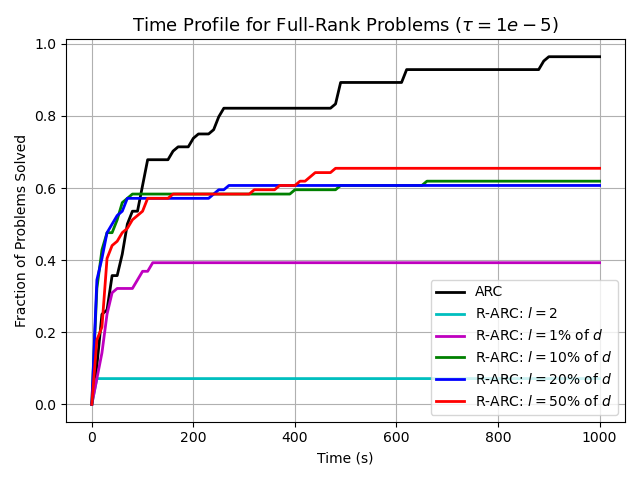}
    \includegraphics[width=0.4\linewidth]{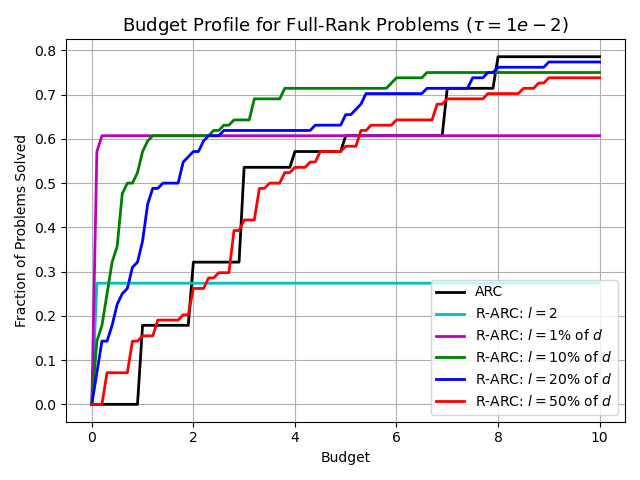}
    \includegraphics[width=0.4\linewidth]{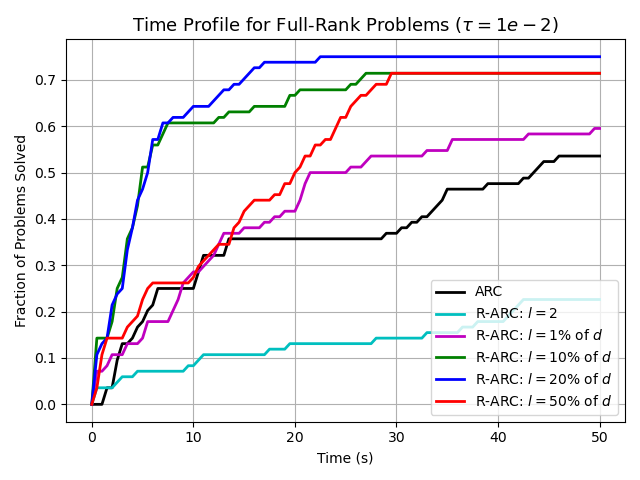}
    \caption{Varying the sketch dimension using scaled Haar matrices in R-ARC; plotting budget and time; full-rank problems}
    \label{fig:haar_rarc}
\end{figure}

\begin{figure}
    \centering
    \includegraphics[width=0.4\linewidth]{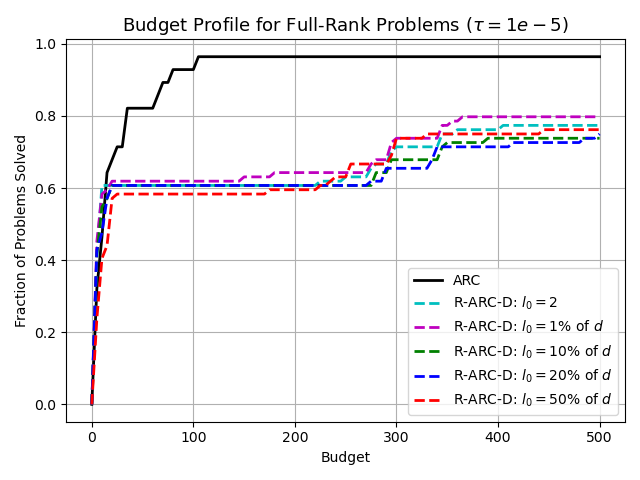}
    \includegraphics[width=0.4\linewidth]{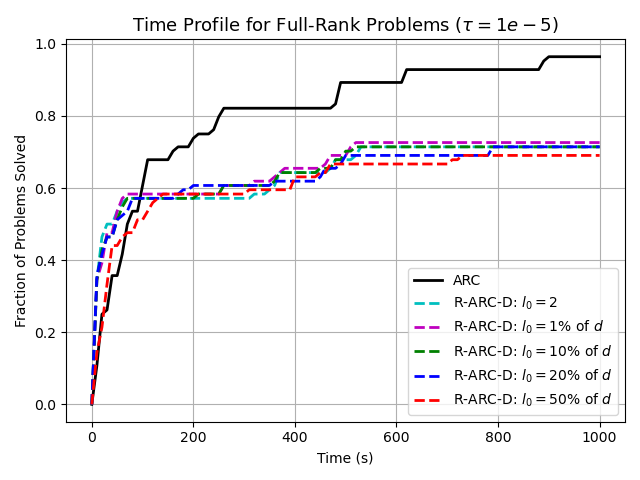 }
    \includegraphics[width=0.4\linewidth]{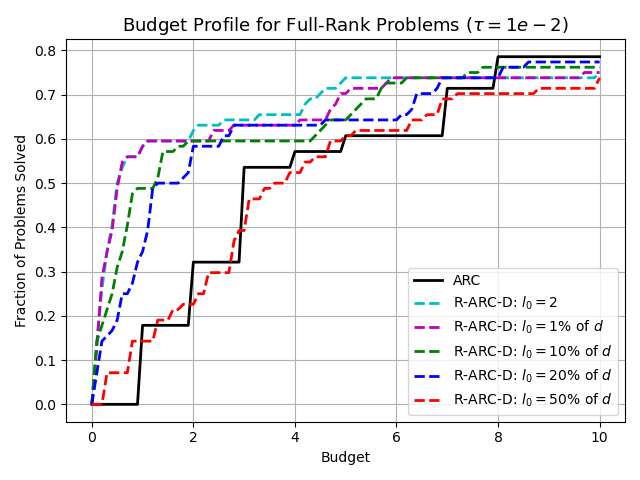}
    \includegraphics[width=0.4\linewidth]{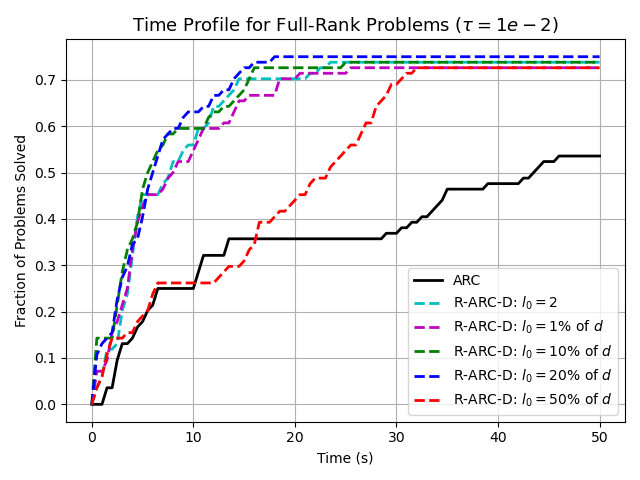}
    \caption{Varying the sketch dimension using scaled Haar matrices in R-ARC-D; plotting budget and time; full-rank problems}
    \label{fig:haar_rarcd}
\end{figure}

\begin{figure}
    \centering
    \includegraphics[width=0.4\linewidth]{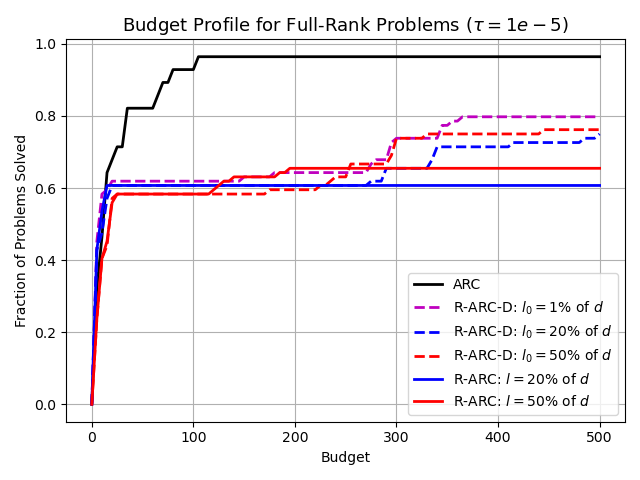}
    \includegraphics[width=0.4\linewidth]{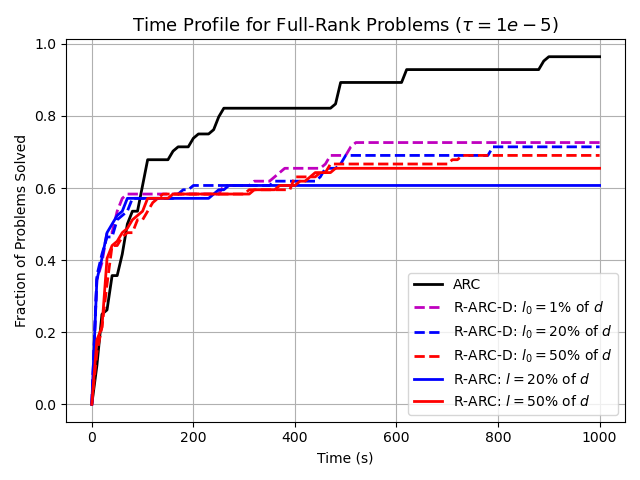 }
    \includegraphics[width=0.4\linewidth]{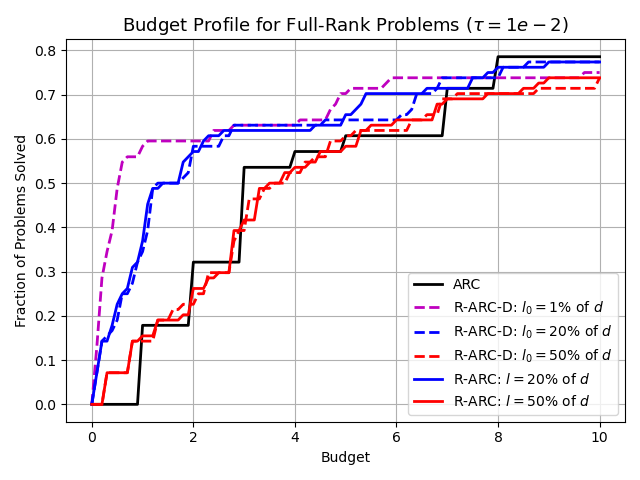}
    \includegraphics[width=0.4\linewidth]{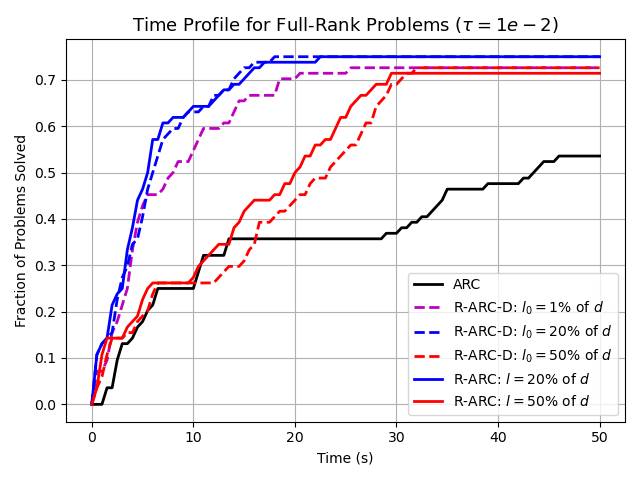}
    \caption{Varying the sketch dimension using scaled Haar matrices in R-ARC/R-ARC-D; plotting budget and time; full-rank problems}
    \label{fig:haar_rarc_vs_rarcd}
\end{figure}

\clearpage

\subsubsection{Low-rank problems}

\begin{figure}[H]
    \centering
    \includegraphics[width=0.4\linewidth]{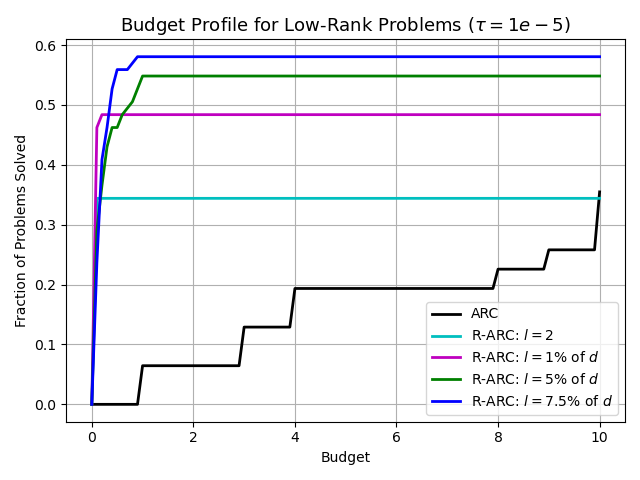}
    \includegraphics[width=0.4\linewidth]{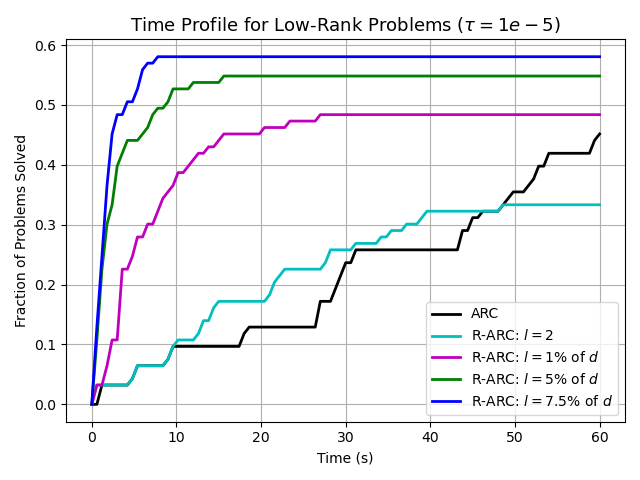}
    \includegraphics[width=0.4\linewidth]{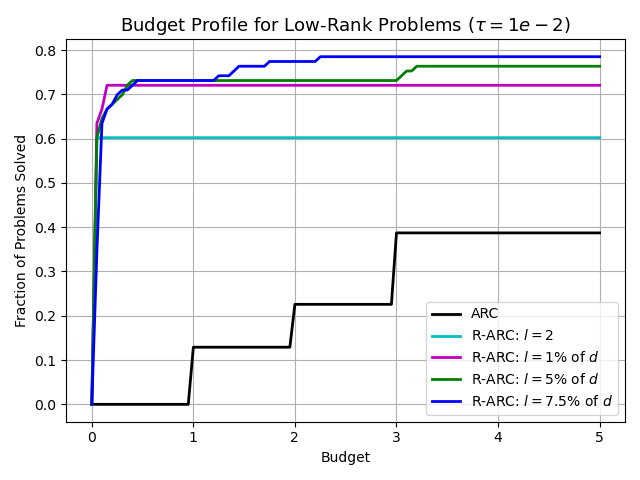}
    \includegraphics[width=0.4\linewidth]{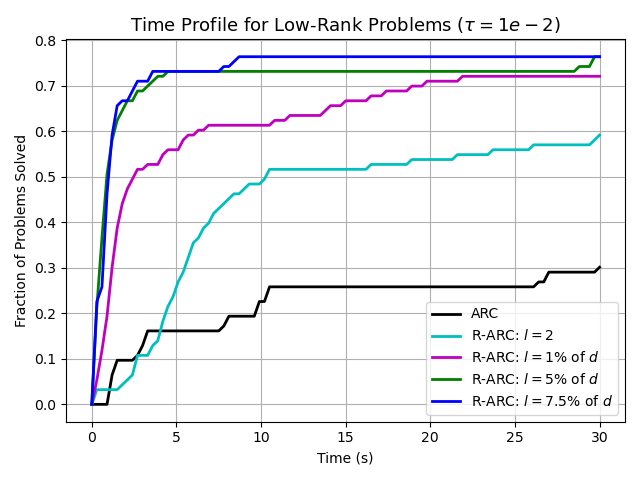}
    \caption{Varying the sketch dimension using scaled Haar matrices in R-ARC; plotting budget and time; low-rank problems}
    \label{fig:lr_haar_rarc}
\end{figure}

\begin{figure}
    \centering
    \includegraphics[width=0.4\linewidth]{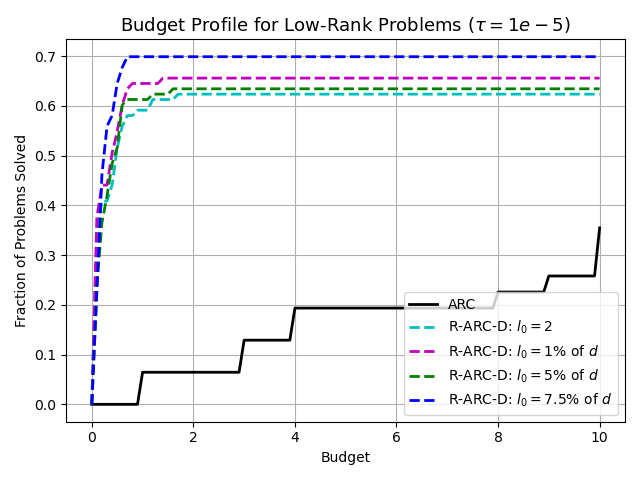}
    \includegraphics[width=0.4\linewidth]{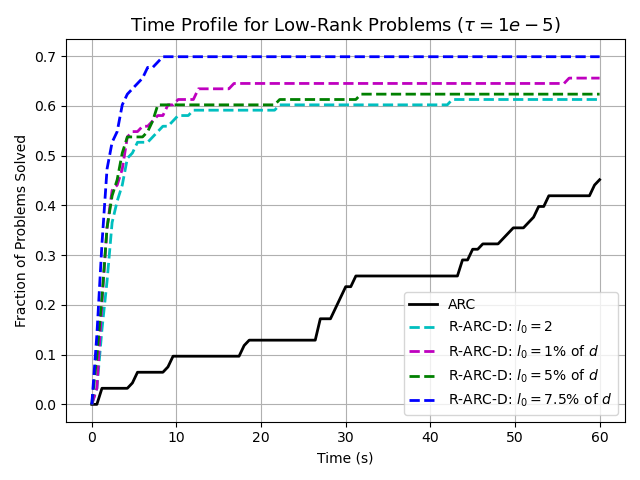 }
    \includegraphics[width=0.4\linewidth]{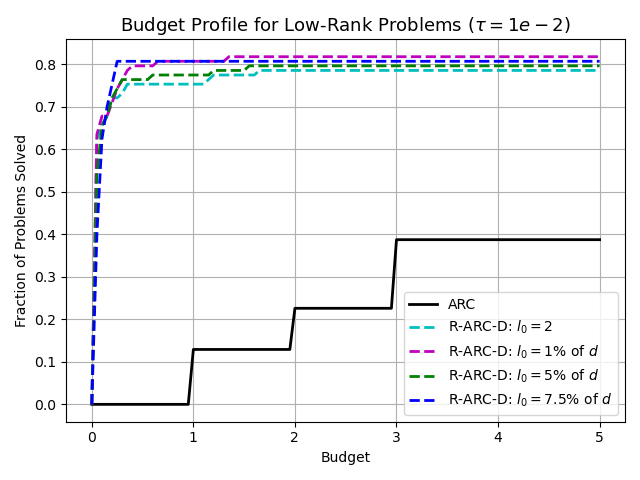}
    \includegraphics[width=0.4\linewidth]{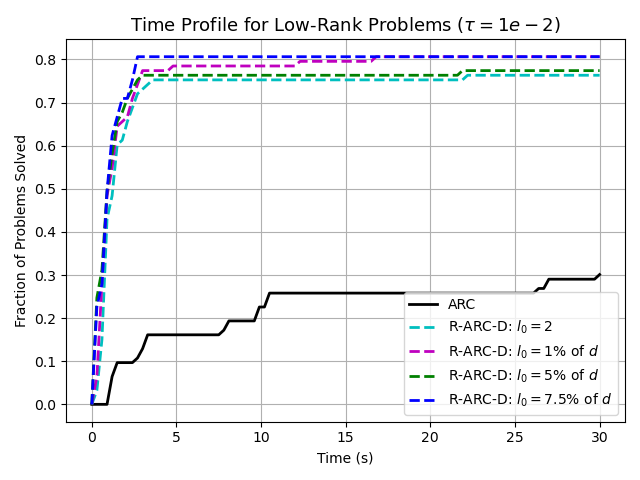}
    \caption{Varying the sketch dimension using scaled Haar matrices in R-ARC-D; plotting budget and time; low-rank problems}
    \label{fig:lr_haar_rarcd}
\end{figure}

\begin{figure}
    \centering
    \includegraphics[width=0.4\linewidth]{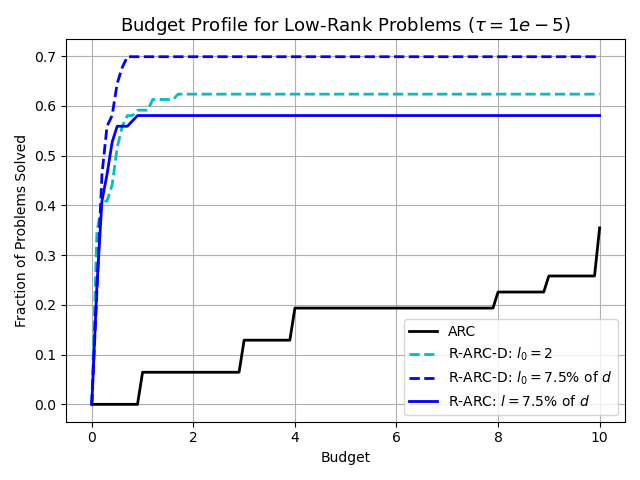}
    \includegraphics[width=0.4\linewidth]{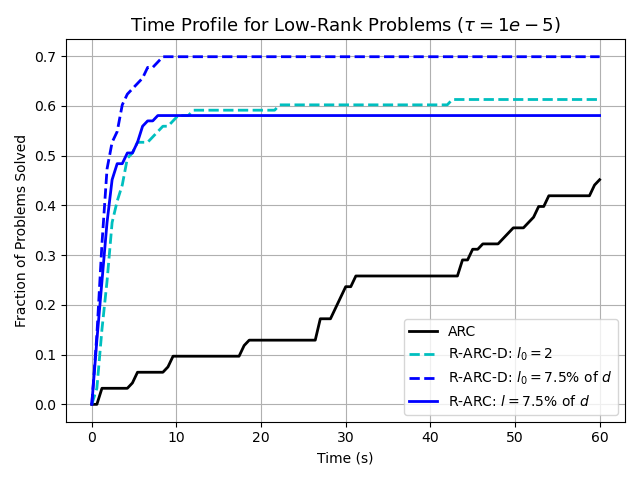 }
    \includegraphics[width=0.4\linewidth]{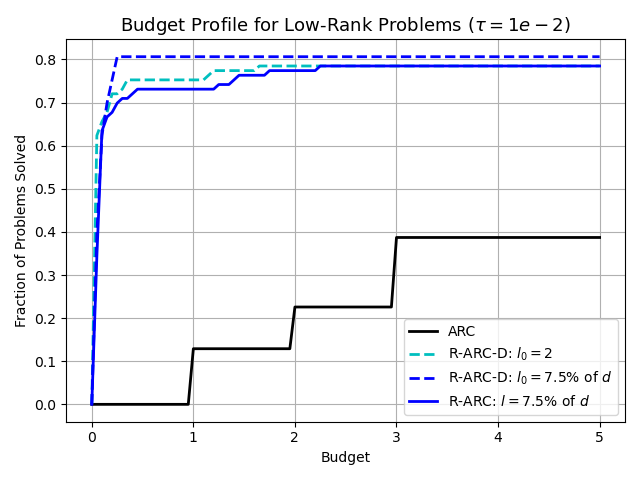}
    \includegraphics[width=0.4\linewidth]{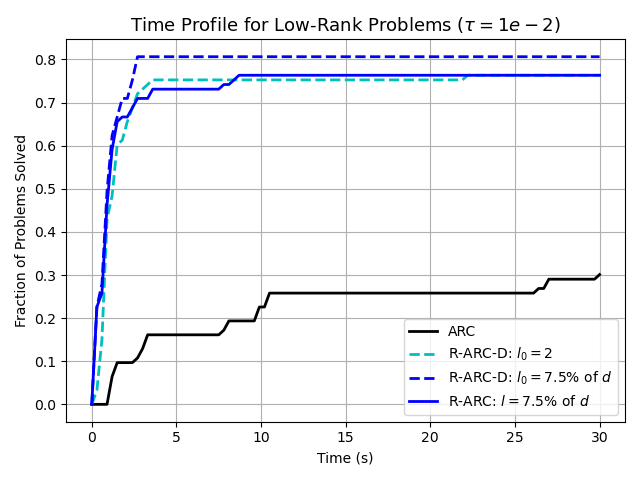}
    \caption{Varying the sketch dimension using scaled Haar matrices in R-ARC/R-ARC-D; plotting budget and time; low-rank problems}
    \label{fig:lr_haar_rarc_vs_rarcd}
\end{figure}

\end{document}

%% file: preamblesAndMacros.tex
\usepackage{silence}
    \usepackage[utf8]{inputenc}
    \usepackage[T1]{fontenc}
    \usepackage{siunitx}
    \usepackage{colonequals}
    \usepackage[super]{nth}
    \usepackage{algorithm}
    \usepackage{amsmath}
    \usepackage{amsthm}
    \usepackage{amssymb}
    \usepackage{tikz}
    \usepackage{url}
    \usepackage{float}
    \usepackage{tabularx}
    \usepackage{pbox}
    \usepackage{tabulary}
    \usepackage{breqn}
    \usepackage{longtable}
    \usepackage{mdframed}
    \usepackage{booktabs}

    \usepackage{graphicx}
    \graphicspath{{./images/}{./}}

    \usepackage{wrapfig}
    \usepackage[font={small}]{caption}
    \usepackage{enumitem}
    \usepackage{physics}
    \usepackage{lipsum}
    \usepackage{mwe}
    \usepackage{breqn}
    \usepackage[noend]{algpseudocode}
    \usepackage{subcaption}
    \usepackage{graphicx}
    \usepackage{enumitem}
    \usepackage{bm}
    \usepackage{textcmds}

    \usepackage{apptools}
    \AtAppendix{\counterwithin{lemma}{section}}
    \AtAppendix{\counterwithin{remark}{section}}
    
    \usepackage{hyperref}
    \hypersetup{
        colorlinks,
            citecolor=black,
        filecolor=black,
        linkcolor=black,
        urlcolor=black
    }

\newtheorem{corollary}{Corollary}

\newtheorem{theorem}{Theorem}
\newtheorem{lemma}{Lemma}
\newtheorem{remark}{Remark}

\newtheorem{definition}{Definition}
\newtheorem{assumption}{Assumption}

\numberwithin{equation}{section} 
\numberwithin{lemma}{section} 
\numberwithin{theorem}{section} 
\numberwithin{definition}{section} 
\numberwithin{corollary}{section} 

\DeclareMathOperator*{\argmin}{arg\,min}

\newcommand{\LH}{L_H}
\newcommand{\kappaS}{\kappa_S}

\newcommand{\logOneOverDeltaS}{\log \bracket{\frac{1}{\deltaS}}}

\newcommand{\nPreFactorTRCubic}{\left[\frac{7}{8}(1-\delta_1) - 1 + \frac{c}{(c+1)^2} \right]^{-1}}
\newcommand{\deltaSThreeExpression}{\pow{e}{-\frac{l(\epSTwo)^2}{C_l} + r +1 }}
\newcommand{\newL}{\tau_{\alpha}}

\newcommand{\epsOne}{\epS}
\newcommand{\nThreeEpH}{N_{\epH}^{(3)}}

\newcommand{\nTwoEpsH}{N_{\epH}^{(2)}}

\newcommand{\epH}{\epsilon_{H}}
\newcommand{\epsHalf}{\epsilon^{\frac{1}{2} }  }
\newcommand{\oneMinusEpSHalf}{(1-\epS)^{\frac{1}{2}}}
\newcommand{\NsKTSKHat}{\normTwo{S_k^T \sKHat}}
\newcommand{\deltaSThree}{\deltaS^{(3)}}

\newcommand{\fZeroMinusfStarOverH}{\frac{f(x_0) - f^*}{h(\epsilon, \alphaZero\gammaOne^{c+\newL}
)}}

\newcommand{\fun}[2]{#1\bracket{#2}}
\newcommand{\pow}[2]{#1^{#2}}
\newcommand{\impliesSince}[1]{\overset{#1}{\implies}}
\newcommand{\hK}{\grad^2 f(x_k)}

\newcommand{\qKHat}[1]{\hat{q}_k \bracket{#1}}
\newcommand{\epSOne}{\epsilon_S^{(1)}}
\newcommand{\epSTwo}{\epsilon_S^{(2)}}
\newcommand{\MK}{M_k}
\newcommand{\hessFK}{\grad^2 f(x_k)}
\newcommand{\SK}{S_k}

\newcommand{\oneMinusEpSTwo}{1 - \epSTwo}

\newcommand{\sqrtFrac}[2]{\sqrt{\frac{#1}{#2}}}
\newcommand{\gDeltaSDeltaOne}{g(\deltaS, \deltaOne)}

\newcommand{\mathO}[1]{\mathcal{O}\left( #1\right)}

\newcommand{\N}{\mathbb{N}}
\newcommand{\R}{\mathbb{R}}

\renewcommand{\P}{\mathbb{P}}

\newcommand{\alphaLow}{\alpha_{low}}
\newcommand{\nEps}{N_{\epsilon}}
\newcommand{\oneOverGammaOneC}{\frac{1}{\gamma_1^c}}

\newcommand{\conditionalP}[2]{\P \left[ #1 | #2 \right]}

\newcommand{\chernoffLowerExponential}{e^{-\frac{\delta_1^2}{2} (1-\delta_S) N}}

\newcommand{\nPreFactorTR}{\left[(1-\delta_S)(1-\delta_1) - 1 + \frac{c}{(c+1)^2} \right]^{-1}}

\newcommand{\twoCases}[4]{         \begin{cases}
             #1, & #2 \\
             #3, & #4.
        \end{cases} }

\newcommand{\sHat}{\hat{s}}
\newcommand{\mKHat}[1]{\hat{m}_k\left( #1\right)}
\newcommand{\xK}{x_k}
\newcommand{\xKPlusOne}{x_{k+1}}
\newcommand{\sK}{s_k}
\newcommand{\normTwo}[1]{\left\lVert#1\right\rVert_2}
\newcommand{\gradFK}{\grad f(\xK)}

\newcommand{\gammaOneC}{\gamma_1^c}

\newcommand{\fK}{f(x_k)}
\newcommand{\innerProduct}[2]{\langle #1, #2 \rangle}

\newcommand{\rLTimesD}{\R^{l\times d}}
\newcommand{\fKPlusOne}{f(x_k + s_k)}
\newcommand{\epS}{\epsilon_S}
\newcommand{\sMax}{S_{\small{max}}}
\newcommand{\alphaMax}{\alpha_{\small{max}}}
\newcommand{\alphaK}{\alpha_k}

\newcommand{\deltaSTwo}{\delta_S^{(2)}}
\newcommand{\aKOne}{A_k^{(1)}}
\newcommand{\aKTwo}{A_k^{(2)}}
\newcommand{\intersect}{\cap}
\newcommand{\barXK}{\bar{x}_k}
\newcommand{\probabilityGivenXK}[1]{\probability{#1 | x_k = \barXK}}

\renewcommand{\complement}[1]{\left( #1\right)^c}

\renewcommand{\abs}[1]{\lvert #1\rvert}

\newcommand{\alphaKPlusOne}{\alpha_{k+1}}
\newcommand{\gammaOne}{\gamma_1}
\newcommand{\gammaTwo}{\gamma_2}

\newcommand{\alphaZero}{\alpha_0}
\newcommand{\alphaMin}{\alpha_{\small{min}}}

\newcommand{\ceil}[1]{\left \lceil{#1}\right \rceil}
\newcommand{\logBaseGammaOne}[1]{\log_{\gammaOne}\left( #1\right)}
\newcommand{\minMe}[2]{\min \set{#1, #2}}

\newcommand{\set}[1]{\left\{ #1 \right\}}

\newcommand{\probability}[1]{\P \left( #1 \right)}
\renewcommand{\complement}[1]{\left( #1 \right)^c}
\newcommand{\squareBracket}[1]{\left[ #1 \right]}

\newcommand{\bracket}[1]{\left( #1\right)}

\newcommand{\sKHat}{\hat{s}_k}
\newcommand{\kappaT}{\kappa_T}

\newcommand{\deltaOne}{\delta_1}

\newcommand{\texteq}[1]{\text{\quad #1}}
\newcommand{\lambdaMin}[1]{\lambda_{min} \bracket{#1}}

\newcommand{\zK}{z^{k}}
\newcommand{\aZeroOne}{A_0^{(1)}}
\newcommand{\aZeroTwo}{A_0^{(2)}}

\newcommand{\logOneOverDeltaSTwo}{\log\bracket{1/\deltaSTwo}}

\newcommand{\deltaS}{\delta_S}

\newcommand{\reply}[1]{#1}

